\documentclass[reqno,a4paper,11pt]{amsart}
\usepackage[T1]{fontenc}
\usepackage[utf8]{inputenc}
\usepackage[english]{babel}
\usepackage{amsmath}
\usepackage{amssymb}
\usepackage{amsthm,thmtools}
\usepackage{setspace}
\usepackage{graphicx}
\usepackage{geometry}
\usepackage{tikz-cd}

\tikzcdset{arrow style=tikz, diagrams={>=Straight Barb}} 

\usepackage{xcolor}
\usepackage{bbm}
\usepackage{mdframed}

\usepackage{libertine}
\usepackage[libertine]{newtxmath}

\makeatletter 
\def\l@subsection{\@tocline{2}{0pt}{1pc}{5pc}{}} \def\l@subsection{\@tocline{2}{0pt}{2pc}{6pc}{}}
\makeatother

\declaretheoremstyle[headfont=\bfseries,bodyfont=\itshape]{myplain}
\declaretheoremstyle[headfont=\bfseries,bodyfont=\normalfont]{mydefinition}
\declaretheoremstyle[headfont=\bfseries,bodyfont=\normalfont,qed=$\diamond$]{myexample}
\declaretheoremstyle[headfont=\itshape,bodyfont=\normalfont, qed=$\diamond$]{myremark}

\declaretheorem[style=myplain, numberwithin=section, name=Theorem]{theo}
\declaretheorem[style=myplain, sharenumber=theo, name=Lemma]{lemma}
\declaretheorem[style=myplain, sharenumber=theo, name=Corollary]{coroll}
\declaretheorem[style=myplain, sharenumber=theo, name=Proposition]{prop}
\declaretheorem[style=myplain, sharenumber=theo, name=Proposition/Definition]{prop/def}

\declaretheorem[style=mydefinition, sharenumber=theo, name=Definition]{definition}

\declaretheorem[style=myexample, sharenumber=theo, name=Example]{example}
\declaretheorem[style=myremark, sharenumber=theo, name=Remark]{rem}

\numberwithin{equation}{section}

\usepackage[colorlinks=true]{hyperref}
\hypersetup{linkcolor=blue, citecolor=blue, filecolor=black, urlcolor=blue}
\everymath{\displaystyle}

\newcommand{\dA}{d_{D}}
\newcommand{\OA}{\Omega_{D}}
\newcommand{\rank}{\operatorname{rank}}
\newcommand{\im}{\operatorname{im}}
\newcommand{\mk}{\mathsf{M}K}
\newcommand{\mc}{\mathsf{M}R}
\newcommand{\pr}{\mathrm{pr}}
\newcommand{\filleddiamond}{\text{\scalebox{0.65}{$\blacksquare$}}}

\begin{document}
	\title{Shifted Contact Structures on differentiable stacks}
	
	\author{Antonio Maglio}
	\address{DipMat, Università degli Studi di Salerno, Via Giovanni Paolo II n°132, 84084 Fisciano (SA), Italy}
	\curraddr{}
	\email{\href{mailto:anmaglio@unisa.it}{anmaglio@unisa.it}}
	\thanks{}
	
	\author{Alfonso Giuseppe Tortorella}
	\address{DipMat, Università degli Studi di Salerno, Via Giovanni Paolo II n°132, 84084 Fisciano (SA), Italy}
	\curraddr{}
	\email{\href{mailto:atortorella@unisa.it}{atortorella@unisa.it}}
	\thanks{}
	
	\author{Luca Vitagliano}
	\address{DipMat, Università degli Studi di Salerno, Via Giovanni Paolo II n°132, 84084 Fisciano (SA), Italy}
	\curraddr{}
	\email{\href{mailto:lvitagliano@unisa.it}{lvitagliano@unisa.it}}
	\thanks{}
	
	
	\begin{abstract}
	We define \emph{$0$-shifted} and \emph{$+1$-shifted contact structures} on differentiable stacks, thus laying the foundations of \emph{shifted Contact Geometry}. As a side result we show that the kernel of a multiplicative $1$-form on a Lie groupoid (might not exist as a Lie groupoid but it) always exists as a differentiable stack, and it is naturally equipped with a stacky version of the curvature of a distribution. Contact structures on orbifolds provide examples of $0$-shifted contact structures, while prequantum bundles over $+1$-shifted symplectic groupoids provide examples of $+1$-shifted contact structures. Our shifted contact structures are related to shifted symplectic structures via a Symplectic-to-Contact Dictionary.
		\end{abstract}

	\maketitle

	\tableofcontents
	
	\section{Introduction}
	A \emph{contact structure} on a manifold $M$ is a hyperplane distribution $K\subseteq TM$ which is maximally non-integrable, i.e.~the \emph{curvature} $R_{K}\colon \wedge^2K \to L:=TM/K$, defined by setting $R_{K}(X,Y)=[X,Y]\operatorname{mod} K$, is non-degenerate. Dually, given a line bundle $L\to M$, and a nowhere zero $L$-valued $1$-form $\theta\in \Omega^1(M,L)$, we call $\theta$ a \emph{contact form}, if the \emph{curvature} $R_{\theta}\colon \wedge^2 K\to L$, defined by setting $R_{\theta}(X,Y)=- \theta([X,Y])$ is non-degenerate, where $K := \ker \theta$. Given a contact structure $K$, the projection $TM\to L:=TM/K$ is a contact form. Conversely, given a contact form $\theta$, the kernel $K=\ker \theta$ is a contact structure. Under this correspondence, $R_{K} = -R_{\theta}$. Contact structures are supported by odd dimensional manifolds. Contact Geometry is the geometry of contact structures, and can be seen as an odd dimensional analogue of Symplectic Geometry in many respects. In fact, there is a ``symplectic-like'' point of view on contact structures. Namely, the latter are equivalent to \emph{symplectic Atiyah forms} \cite{Vitagliano:djbundles} (see Section \ref{sec:atiyah} below, see also \cite{Grabowski:remarks} for an alternative approach). 
	
	Differentiable stacks \cite{Xu:stacks,DelHoyo:stacks} are (equivalent to) classes of Morita equivalent Lie groupoids and they model certain singular spaces like orbifolds, orbit spaces of smooth Lie group actions and leaf spaces of foliations. Roughly, geometric structures on  differentiable stacks are Morita invariant geometric structures on Lie groupoids, and they usually possess a grading, sometimes called the \emph{shift}, essentially inherited from the grading of the simplicial structure on the nerve of the groupoid. 	The reader can refer, e.g., to \cite{Getzler:shifted, Cueca:shiftedstructures} for a theory of shifted differential forms and shifted symplectic structures on (higher) Lie groupoids (see also, e.g., \cite{Pantev:shiftedstructures} for the algebraic geometric setting). A $0$-shifted symplectic structure on a Lie groupoid $G \rightrightarrows M$ \cite{Hoffman:stacky} is a closed differential $2$-form $\omega$ on $M$ which is \emph{basic} with respect to the partition of $M$ by $G$-orbits, and non-degenerate in the transverse direction (in other words $\ker \omega$ agrees with the tangent distribution to $G$-orbits). When the differentiable stack $[M/G]$ presented by $G$ is a smooth manifold, then $\omega$ defines an ordinary symplectic structure on it. The definition of a $+1$-shifted symplectic structure is somewhat more involved and it is relevant for Poisson Geometry (see \cite{Xu:momentum_maps} and \cite{Zhu:twisted}, but beware that different terminologies are adopted in those references). Indeed, $+1$-shifted symplectic structures are the global counterparts of \emph{twisted Dirac structures} \cite{Zhu:twisted}.  
	
	The main aim of this paper is to lay some foundations of \emph{shifted Contact Geometry} by defining \emph{$0$-shifted} and \emph{$+1$-shifted contact structures} on differentiable stacks. We remark that shifted Contact Geometry has been also initiated in Derived Algebraic Geometry in \cite{BerktavA,BerktavB}. However, in those references, the focus is rather on negatively shifted contact structures. It turns out that, for our purposes, the dual definition of a contact structure in terms of a line bundle valued differential $1$-form is more appropriate (than that in terms of a hyperplane distribution). There are five aspects of such definition for which we need to find (Lie groupoid and) Morita invariant versions and we highlight them here for future reference:
	\bigskip
	\begin{mdframed}
		\begin{enumerate}
	\item[(A1)] a line bundle $L \to M$,
	\item[(A2)] an $L$-valued differential $1$-form $\theta$,
	\item[(A3)] the kernel $K_\theta$ of $\theta$ (seen as a vector bundle),
	\item[(A4)] the curvature $R_\theta$ of $\theta$ (seen as a vector bundle map $R_\theta \colon K_\theta \to \operatorname{Hom}( K_\theta, L)$),
	\item[(A5)] the non-degeneracy condition on $R_\theta$,
	\end{enumerate}
	\end{mdframed}
	\bigskip
(in what follows we will sometimes refer to the items (A1)--(A5) above simply as ``the aspects'').	We will first discuss $0$-shifted contact structures, and then concentrate on $+1$-shifted contact structures. In order to take care of (A1), we begin with a line bundle $L \to G$ which is additionally a VB groupoid and its VB Morita equivalence class (see \cite{DelHoyo:VBmorita} for the notion of VB Morita equivalence) that, for the purposes of this paper, we interpret, without any claim of generality, as a line bundle in the category of differentiable stacks. For (A2) we equip $G$ with a multiplicative $L$-valued $1$-form $\theta \in \Omega^1 (G, L)$. At this point, we would like to define the kernel and the curvature of $\theta$. However, the ordinary notions of kernel and curvature are only well-defined when $\theta$ is nowhere zero, which is a too strong requirement for our purposes (it actually violates Morita invariance). Fortunately, in the $+1$-shifted case, we can show that, while the ordinary kernel of $\theta$ does not always exist in the categories of smooth manifolds and Lie groupoids, it always exists in the category of differentiable stacks. Namely, for (A3), we are able to define a VB groupoid $\mk_\theta \to G$ which plays the same role as the kernel of $\theta$ (and reduces to that when the kernel exists) up to Morita equivalences. We call $\mk_\theta \to G$ the \emph{Morita kernel} of $\theta$ and show that the latter is a Morita invariant notion. Similarly, for (A4), we can define a Morita invariant notion of curvature of $\theta$. It is a VB groupoid map $\mc_\theta\colon \mk_\theta \to \operatorname{Hom}(\mk_\theta, L)$ which we call the \emph{Morita curvature}. Finally, for (A5), there is a natural and Morita invariant notion of \emph{non-degeneracy} for $\mc_\theta$ which allows us to give a Morita invariant definition of $+1$-shifted contact structure on a differentiable stack. We expect that the Morita kernel and the Morita curvature will be also relevant for other purposes, e.g., one could give a definition of \emph{involutive hyperplane distribution on a differentiable stack} using these same notions. Notice that the definition of shifted symplectic structure can be straighforwardly translated into a definition of shifted symplectic Atiyah form. We also prove that, similarly as in ordinary Contact Geometry, $+1$-shifted contact structures are equivalent to $+1$-shifted symplectic Atiyah forms. This result strongly supports our definition.
	
The paper is organized as follows. In Section \ref{sec:atiyah} we recall what are symplectic Atiyah forms and how they can encode contact structures (see \cite{Vitagliano:djbundles,Vitagliano:holomorphic}). In Section \ref{sec:lb} we recall the necessary aspects of the theory of \emph{VB groupoids}, i.e.~vector bundles in the category of Lie groupoids \cite{Mackenzie:generaltheory, Gracia:VBgroupoids}, including their Morita theory \cite{DelHoyo:VBmorita}, with a special emphasis on what we call LB groupoids. This takes care of the aspect (A1) in the definition of a contact structure. Let $G \rightrightarrows M$ be a Lie groupoid. The \emph{fiber of a VB groupoid} over a point of $M$ is a $2$-term cochain complex of vector spaces and a VB groupoid morphism induces a cochain map between the fibers. There is an extremely useful characterization of \emph{VB Morita maps} (i.e.~VB groupoid morphisms which are also Morita maps), due to del Hoyo-Ortiz \cite{DelHoyo:VBmorita}, which we use throughout the paper. Namely a VB groupoid morphism is a VB Morita map if and only if it is a Morita map on bases and, additionally, if it induces quasi-isomorphisms on fibers. We also introduce in this section a notion of natural isomorphism between two VB groupoid morphisms adapted to the vector bundle structures, which we call \emph{linear natural isomorphism} and we will need in the sequel. We show that a linear natural isomorphism between VB groupoid morphisms covering the same Lie groupoid map is equivalent to a homotopy between the cochain maps induced on fibers (smoothly depending on the base point). To the best of our knowledge, this is new (but see \cite[Section 6.1]{DelHoyo:VBmorita}).

In Section \ref{sec:0-shifted_cs}, after recalling what are \emph{$0$-shifted symplectic structures}, we define \emph{$0$-shifted symplectic Atiyah forms} and \emph{$0$-shifted contact structures}. This requires some work in order to take due care of all the aspects (A2)--(A5) of the definition. Such work will be of some inspiration for the subsequent $+1$-shifted case. However, notice that, in some respects, the $0$-shifted case is more complicated than the $+1$-shifted case (see below).

In Section \ref{sec:def1shif} we recall the definition of \emph{$+1$-shifted symplectic structure} \cite{Xu:momentum_maps,Zhu:twisted, Cueca:shiftedstructures,Getzler:shifted} and we translate it into a definition of \emph{$+1$-shifted symplectic Atiyah form} in Section \ref{sec:Atiyah_shifted}. This is straightforward. $+1$-shifted symplectic Atiyah forms can serve as a first definition of $+1$-shifted contact structure, indicating the effectiveness of the language of Atiyah forms in Contact Geometry.
	
In the remainder of Section \ref{sec:contact} we provide a definition of \emph{$+1$-shifted contact structure on a differentiable stack} more similar in spirit to the definition of a contact structure in terms of contact forms and/or hyperplane distributions. We begin with a \emph{multiplicative line bundle valued $1$-form} (aspect (A2)). Next, we have to define \emph{stacky versions} of the kernel and the curvature of a (line bundle valued) $1$-form, the \emph{Morita kernel} (aspect (A3)) and the \emph{Morita curvature} (aspect (A4)). Finally, we propose a definition of when is the Morita curvature ``non-degenerate'' (aspect (A5)). According to our main purpose, we pay special attention to Morita invariance of all the constructions. As a strong motivation for our definition, we show in this section that $+1$-shifted contact structures are equivalent to $+1$-shifted sympectic Atiyah forms.

We conclude Section \ref{sec:contact} with two main examples. First of all, the pre-quantization of a (pre-quantizable) quasi-symplectic groupoid \cite{Xu:prequantization} is a $+1$-shifted contact groupoid. In this case, the $1$-form is nowhere-zero and its Morita kernel can be replaced by the honest kernel (up to Morita equivalence). This suggests that there might be a more general notion of \emph{pre-quantization of a $+1$-shifted symplectic stack}, more stacky friendly, and where the contact $1$-form is non-necessarily nowhere-zero. Conjecturally, this might enlarge the class of pre-quantizable $+1$-shifted symplectic forms. This line of thought will be explored elsewhere. Finally, pre-contact groupoids in the sense of \cite{Vitagliano:djbundles} (in particular contact groupoids) do also provide instances of $+1$-shifted contact structures. Recall that pre-contact groupoids are the contact analogues of quasi-symplectic groupoid and they are the global counterparts of \emph{Dirac-Jacobi bundles} (the contact analogues of Dirac manifolds). 
	
	We assume that the reader is familiar with Lie groupoids and Lie algebroids. Our main references for this material are \cite{CF2011, Mackenzie:generaltheory}. For simplicity, we work by default with Hausdorff Lie groupoids, but most of our results are also valid for the larger class of Lie groupoids admitting an Ehresmann connection.
	
	\bigskip
	
	{\bf Acknowledgments.} \ \ 
	We thank Chenchang Zhu for having suggested to us that the prequantization of $+1$-shifted symplectic structures might provide examples of $+1$-shifted contact structures. We also thank Miquel Cueca for helpful comments. Finally, we wish to warmly thank the anonymous referees for extremely useful suggestions that helped improving a lot the original manuscript. The authors are members of the GNSAGA of INdAM.

	\section{Contact Structures as Symplectic Atiyah Forms}\label{sec:atiyah}
	
	In this section we recall from \cite{Vitagliano:djbundles} the alternative approach to contact structures based on \emph{Atiyah forms} (see also \cite{Grabowski:remarks} and Remark \ref{rem:hom_sympl} for a closely related, but different, approach). We begin with some facts about the Atiyah algebroid of a vector bundle.
	
	Let $E\to M$ be a vector bundle (VB in what follows) over $M$. A \emph{derivation of $E$} is an $\mathbbm{R}$-linear operator $\Delta\colon\Gamma(E)\to \Gamma(E)$ such that, for any $f\in C^{\infty}(M)$ and $e\in \Gamma(E)$, the following Leibniz rule is satisfied: $$\Delta(fe)=X(f)e +f\Delta(e)$$for a, necessarily unique, vector field $X\in \mathfrak{X}(M)$, called the \emph{symbol} of $\Delta$ and also denoted by $\sigma(\Delta)$.
	
	Derivations of $E$ are sections of a Lie algebroid $DE\to M$, the \emph{Atiyah algebroid of $E$}, defined as follows. The fiber $D_xE$ of $DE$ over $x\in M$ consists of $\mathbbm{R}$-linear operators $\delta\colon\Gamma(E)\to E_x$ such that, for any $f\in C^{\infty}(M)$ and $e\in \Gamma(E)$,
	\begin{equation*}
		\delta(fe)= v(f)e_x + f(x)\delta(e)
	\end{equation*}
	for a, necessarily unique, tangent vector $v\in T_xM$, the \emph{symbol} of $\delta$, which is also denoted by $\sigma(\delta)$. The Lie bracket on sections of $DE$ is the commutator of derivations, while the anchor $\sigma\colon DE\to TM$ maps a derivation to its symbol.
	The symbol map fits in the following short exact sequence of vector bundles
	\begin{equation}\label{eq:Spencer}
		\begin{tikzcd}
			0 \arrow[r] &\operatorname{End}E \arrow[r] &DE \arrow[r, "\sigma"] &TM \arrow[r] &0
		\end{tikzcd} ,
	\end{equation}
	where $\operatorname{End}E$ is the vector bundle of endomorphisms, and the map $\operatorname{End}E\to DE$ is the inclusion. A connection on $E$ is a right splitting $\nabla\colon TM \to DE$ of \eqref{eq:Spencer}, so it determines a direct sum decomposition
	\begin{equation*}
		DE \cong TM \oplus \operatorname{End}E, \quad \delta \mapsto \big(\sigma(\delta), f_{\nabla}(\delta)\big),
	\end{equation*}
	where $f_{\nabla}\colon DE \to \operatorname{End}E$ is the associated left splitting: 
	\begin{equation*}
		f_{\nabla}(\delta) e_x= \delta(e)- \nabla_{\sigma(\delta)}e, \quad \text{for all } \delta \in D_xE, \quad e\in \Gamma(E).
	\end{equation*}
	It follows that $\rank DE= \dim M + (\rank E)^2$.
	
	The Atiyah algebroid $DE$ naturally acts on $E$: the action of a derivation on a section is just the tautological one. Then there is a de Rham-like differential $\dA$ on the graded vector space $\OA^{\bullet}(E):= \Gamma(\operatorname{Alt}^{\bullet}(DE, E))$ of $E$-valued alternating forms on $DE$. The cochain complex $(\OA^{\bullet}(E), \dA)$ is sometimes called the \emph{der-complex} \cite{Rubtsov:dercomplex} and it is actually acyclic. Even more, it possesses a canonical contracting homotopy given by the contraction $\iota_{\mathbbm{I}}\colon \OA^{\bullet}(E)\to \OA^{\bullet-1}(E)$ with the \emph{identity derivation} $\mathbbm{I}\colon \Gamma(E)\to \Gamma(E)$. In the sequel, cochains in  $(\OA^{\bullet}(E), \dA)$ will be called \emph{Atiyah forms} (on $E$).
	
	The correspondence $E\leadsto DE$ is functorial in the following sense. Let $E\to M$, $E' \to M'$ be two vector bundles. First of all, we say that a VB morphism $(F, f)\colon E \to E'$ covering a smooth map $f \colon M \to M'$, is \emph{regular} if, for any $x \in M$, the restriction $F_x := F|_{E_x} \colon E_x \to E'_{f(x)}$ of $F$ to fibers is an isomorphism. Given a regular VB morphism $(F, f) \colon E\to E'$, a section $e'\in \Gamma(E')$ can be pulled-back to a section $F^{\ast}e'\in \Gamma(E)$, defined by $(F^{\ast}e')_x = F_x^{-1}(e'_{f(x)})$, $x\in M$. The map $DF \colon DE \to DE'$ defined by 
	\begin{equation*}
		DF(\delta)e':=F\big(\delta(F^{\ast}e')\big), \quad \delta\in DE, \quad e'\in \Gamma(E'),
	\end{equation*}
	is a Lie algebroid morphism covering $f$.
	It easy to see that
	\begin{equation}
		\label{eq:DFcommutes}
		\sigma \circ DF = df \circ \sigma.
	\end{equation}
	Notice that any regular VB morphism $(F, f)\colon E\to E'$ induces a \emph{pullback of Atiyah forms}, $F^{\ast}\colon \OA^{\bullet}(E')\to \OA^{\bullet}(E)$: for all $x \in M$,
	\begin{equation*}
		(F^{\ast}\omega)_{x}(\delta_1, \dots, \delta_k):=F_x^{-1}\big(\omega_{f(x)}(DF(\delta_1), \dots, DF(\delta_k))\big), \quad \delta_1, \dots, \delta_k \in D_xE, \quad \omega \in \OA^k(E').
	\end{equation*}
	Finally, $(F, f)\colon E\to E'$ does also induce a \emph{pullback of vector bundle valued forms}, $F^{\ast}\colon \Omega^{\bullet}(M',E')\to \Omega^{\bullet}(M,E)$: for all $x \in M$,
	\begin{equation*}
		(F^{\ast}\theta)_x(v_1, \dots, v_k):=F_x^{-1}\big(\theta_{f(x)}(df(v_1), \dots,df(v_k))\big), \quad v_1,\dots, v_k\in T_xM, \quad \theta\in \Omega^k(M',E').
	\end{equation*}
	
	\begin{rem}
		\label{rem:Atiyahderivation}
		Let $E\to M$ be a vector bundle. The Atiyah algebroid $DE\to M$ is the Lie algebroid of the general linear groupoid $\operatorname{GL}(E)$. Indeed any derivation $\delta\in D_xE$ is the velocity of a curve of isomorphisms $\Upsilon(\varepsilon)\colon E_x\to E_{\gamma(\varepsilon)}$ with $\Upsilon(0)=\operatorname{id}_{E_x}$, where $\gamma(\varepsilon)$ is a curve on $M$, i.e. 		
		\begin{equation*}
			\delta(e)= \frac{d}{d\varepsilon}|_{\varepsilon=0} \Upsilon(\varepsilon)^{-1}(e_{\gamma(\varepsilon)}), \quad e\in \Gamma(E).
		\end{equation*}
		In this case we write $\delta= \tfrac{d}{d\varepsilon}|_{\varepsilon=0} \Upsilon(\varepsilon)$. Notice that $\sigma (\delta) =\tfrac{d}{d\varepsilon}|_{\varepsilon=0} \gamma(\varepsilon)$.

		If $(F,f)\colon (E\to M)\to (E'\to M')$ is a regular VB morphism, then the induced map $DF\colon DE\to DE'$ works as follows: for any $\delta= \tfrac{d}{d\varepsilon}|_{\varepsilon=0} \Upsilon(\varepsilon)\in D_xE$, $x \in M$, the derivation $DF(\delta)\in D_{f(x)}E'$ is given by 
		\begin{equation*}
			\frac{d}{d\varepsilon}|_{\varepsilon=0} F_{\gamma(\varepsilon)}\circ \Upsilon(\varepsilon)\circ F_x^{-1}.
		\end{equation*}
	\end{rem}

	In the case when $E=L$ is a line bundle, $\operatorname{End} E = \operatorname{End} L = \mathbbm R_M := M \times \mathbbm R$, the trivial line bundle over $M$, the sequence \eqref{eq:Spencer} reduces to
\begin{equation}\label{eq:Spencer_L}
		\begin{tikzcd}
			0 \arrow[r] &\mathbbm R_M \arrow[r] &DL \arrow[r, "\sigma"] &TM \arrow[r] &0
		\end{tikzcd},
	\end{equation}
and, given a connection $\nabla$ on $L$, the associated left splitting $f_\nabla \colon DL \to \mathbbm R_M$ is a fiber-wise linear function on $DL$. Moreover, in this case, every first order linear differential operator $\Gamma(L)\to \Gamma(L)$ is a derivation, so $DL \cong \operatorname{Hom}(J^1L,L)$ and $J^1L\cong \operatorname{Hom}(DL,L)$, where $J^1L$ is the first jet bundle of $L$. 
	
	Let $L\to M$ be a line bundle. There is an alternative description of Atiyah forms on $L$ as pairs of forms in $\Omega^{\bullet}(M,L)$. To see this, begin noticing that, for any $k$, the map $\sigma^{\ast}\colon \Omega^k(M,L) \to \OA^k(L)$ defined by:
	\begin{equation*}
		(\sigma^{\ast}\theta)(\Delta_1, \dots, \Delta_k) =\theta\big(\sigma(\Delta_1), \dots, \sigma(\Delta_k)\big) \quad \text{for all } \Delta_1, \dots ,\Delta_k \in \Gamma(DL), 
	\end{equation*}
	is injective and $\im \sigma^{\ast}= \ker \iota_{\mathbbm{I}}$. Hence there is a short exact sequence of $C^{\infty}(M)$-modules
	\begin{equation}
		\label{ses_atiyah}
		\begin{tikzcd}
			0 \arrow[r] &\Omega^{\bullet}(M,L) \arrow[r, "\sigma^{\ast}"] &\OA^{\bullet}(L) \arrow[r] &\Omega^{\bullet-1}(M,L)\arrow[r] &0,
		\end{tikzcd}
	\end{equation}
	where the projection $\OA^{\bullet}(L)\to \Omega^{\bullet -1}(M,L)$ maps the Atiyah form $\omega \in \OA^{\bullet}(L)$ to the form $\omega_0\in \Omega^{\bullet -1}(M,L)$ uniquely determined by $\sigma^{\ast}\omega_0=\iota_{\mathbbm{I}}\omega$. The map 
	\[
	\Omega^{\bullet -1}(M,L)\to \OA^{\bullet}(L), \quad \tau\mapsto \dA(\sigma^{\ast}\tau),
	\]
	is a canonical $\mathbbm{R}$-linear splitting of \eqref{ses_atiyah}. Accordingly, there is a vector space isomorphism
	\begin{equation}
		\label{eq:components}
		\Omega^{\bullet-1}(M,L)\oplus \Omega^{\bullet}(M,L) \to \OA^{\bullet}(L), \quad
		(\omega_0, \omega_1) \mapsto \omega:=\sigma^{\ast}\omega_1+ \dA\sigma^{\ast}\omega_0.
	\end{equation}
	The two forms $\omega_0$ and $\omega_1$ will be called the \emph{components} of $\omega$, and we write $\omega \rightleftharpoons (\omega_0, \omega_1)$. Moreover, if $\omega\rightleftharpoons (\omega_0,\omega_1)$, then $\dA\omega \rightleftharpoons (\omega_1,0)$ and $\iota_{\mathbbm{I}}\omega \rightleftharpoons (0, \omega_0)$. In particular $\omega$ is closed if and only if $\omega_1=0$.
	
	Now, let $(F, f) \colon (L \to M) \to (L' \to N)$ be a regular VB morphism between two line bundles. In this case we call $(F, f)$ an \emph{LB morphism} (LB for line bundle). It is easy to see that
	\begin{equation*}
		F^{\ast}(\sigma^{\ast}\theta)= \sigma^{\ast}(F^{\ast}\theta), \quad \theta\in \Omega^{\bullet}(N,L').
	\end{equation*}
	It follows that, if $\omega\rightleftharpoons (\omega_0, \omega_1)$ is an Atiyah form on $L'$, then $F^{\ast}\omega\rightleftharpoons (F^{\ast}\omega_0, F^{\ast}\omega_1)$.
	
	The next proposition will be useful later.	
	\begin{prop}\label{prop:ker_omega_comp}
		Let $\omega\in \OA^2(L)$ be an Atiyah $2$-form and let $(\omega_0,\omega_1)\in \Omega^1(M,L)\oplus\Omega^2(M,L)$ be its components. Let $\delta, \delta'\in D_xL$, and set $v=\sigma(\delta), v'=\sigma(\delta')\in T_xM$, $x \in M$. Then, for any connection $\nabla$ on $L$,
		\begin{equation*}
			\omega(\delta, \delta')= \omega_1(v, v') + d^{\nabla}\omega_0(v,v') +f_{\nabla}(\delta) \omega_0(v') - f_{\nabla}(\delta') \omega_0(v),
		\end{equation*}
		where $d^\nabla \colon \Omega^\bullet (M, L) \to \Omega^{\bullet +1} (M, L)$ is the connection differential. In particular, when $\omega$ is $\dA$-closed (i.e. $\omega_1=0$),
		\begin{equation}
			\label{eq:omegaandcomponents}
			\omega(\delta, \delta')=d^{\nabla}\omega_0(v,v') +f_{\nabla}(\delta) \omega_0(v') - f_{\nabla}(\delta') \omega_0(v).
		\end{equation}
	\end{prop}
	\begin{proof}
		Let $\delta, \delta', v, v'$ be as in the statement, and let $\Delta,\Delta'\in \Gamma(DL)$ be such that $\Delta_x=\delta,\Delta'_x=\delta'$. Set $V := \sigma (\Delta),V' := \sigma (\Delta') \in \mathfrak{X}(M)$ so that $V_x = v, V'_x = v'$. Then
		\begin{align*}
			\dA\sigma^{\ast}\omega_0(\delta,\delta') &=\delta(\sigma^{\ast}\omega_0(\Delta')) - \delta'(\sigma^{\ast}\omega_0(\Delta)) - \sigma^{\ast}\omega_0([\Delta, \Delta']_x) \\
			&=\delta(\omega_0(V')) - \delta'(\omega_0(V)) - \omega_0([V,V']_x)\\
			&= (\delta-\nabla_v)\omega_0(v') - (\delta'-\nabla_{v'})\omega_0(v) + \nabla_v\omega_0(V') -\nabla_{v'}\omega_0(V) -\omega_0([V,V']_x)\\
			&=d^{\nabla}\omega_0(v,v') +f_{\nabla}(\delta)\omega_0(v') - f_{\nabla}(\delta')\omega_0(v).
		\end{align*}
		Now, the claim follows from $\omega=\sigma^{\ast}\omega_1+\dA\sigma^{\ast}\omega_0$.
	\end{proof}
	
	An easy consequence of Proposition \ref{prop:ker_omega_comp} is the following
	\begin{coroll}
		\label{coroll:kernelofAtiyahform}
		Let $\omega\in \OA^2(L)$ and $\omega\rightleftharpoons(\omega_0,\omega_1)$. Then a derivation $\delta\in D_xL$ is in the kernel of $\omega$ if and only if
		\begin{itemize}
			\item[\emph{i)}] $\sigma(\delta)\in \ker(\omega_0)$, and
			\item[\emph{ii)}] $\iota_{\sigma(\delta)}(\omega_1+d^{\nabla}\omega_0) + f_{\nabla}(\delta)\omega_0=0$
		\end{itemize}
		for some, hence any, connection $\nabla$ on $L$.
	\end{coroll}
	
	Given a line bundle $L \to M$, a \emph{symplectic Atiyah form} is an Atiyah $2$-form $\omega\in \OA^2(L)$ such that $\omega$ is \emph{$\dA$-closed}, i.e. $\dA \omega=0$, and \emph{non-degenerate}, i.e.~its flat map, also denoted $\omega \colon DL\to J^1L$, $\delta \mapsto \omega (\delta, -)$ is a vector bundle isomorphism. There is a relationship between symplectic Atiyah forms and Contact Geometry \cite{Vitagliano:djbundles, Vitagliano:holomorphic}. Remember from the Introduction that a \emph{contact structure} on a manifold $M$ is the kernel of a \emph{contact $1$-form}, i.e., in our terminology, a $1$-form $\theta \in \Omega^1 (M, L)$ with values in a line bundle $L \to M$ such that
	\begin{itemize}
	\item $\theta$ is nowhere $0$ (hence $K_\theta := \ker \theta$ is a well-defined hyperplane distribution on $M$), and
	\item the \emph{curvature} $R_\theta \colon \wedge^2 K_\theta \to L$, $(X, Y) \mapsto - \theta ([X, Y])$ is non-degenerate.
	\end{itemize}

	\begin{rem}
	We have two simple remarks on the curvature $R_\theta$. First of all, it can be obviously encoded into a vector bundle map $K_\theta \to \operatorname{Hom}(K_\theta, L)$, its flat map (an isomorphism for a contact form), also denoted $R_\theta$, and in what follows, we will often take this point of view. Second, it is often useful to extend $R_\theta$ to the whole tangent bundle $TM$. This can be done by choosing a connection $\nabla$ on $L$. Indeed, if we do so, then $R_\theta = (d^\nabla \theta) |_{K_\theta}$.
	\end{rem}
	
	\begin{rem}
	In the Contact Geometry literature, it is often assumed that the line bundle $L$ is trivial. Moreover, a trivialization $L \cong \mathbbm R_M$ is fixed once for all, so that the contact form identifies with a global plain $1$-form $\theta \in \Omega^1 (M)$. For conceptual reasons it seems however better to us not to make this simplifying assumption and work in the general case.
	\end{rem}

	The relation between contact forms and symplectic Atiyah forms is given by the following
	\begin{theo}[{\cite[Proposition 3.6]{Vitagliano:djbundles}}]
		\label{prop:corrispondenza}
		The assignment $\theta\mapsto \omega\rightleftharpoons (\theta,0)$ establishes a bijection between contact forms and symplectic Atiyah forms.
	\end{theo}
	
	\begin{rem}\label{rem:hom_sympl}
	We recall here for later purposes that contact structures are also equivalent to \emph{homogeneous symplectic structures} \cite{Grabowski:remarks, Vitagliano:holomorphic}. Let $\mathbbm R^\times$ be the multiplicative group of non-zero reals, let $h$ be a principal action of $\mathbbm R^\times$ on a manifold $P$, let $M = P/\mathbbm R^\times$, and let $L_M \to M$ be the line bundle associated to the tautological representation of $\mathbbm R^\times = \operatorname{GL}(\mathbbm R, 1)$ on $\mathbbm R$. A differential form $\omega$ on $P$ is \emph{homogeneous of degree $k$} if $h_\varepsilon^\ast \omega = \varepsilon^k \omega$ for all $\varepsilon \in \mathbbm R^\times$. Sections of $L_M$ clearly identify with degree $1$ homogeneous functions on $P$. Now let $\theta \in \Omega^1 (M, L_M)$. Using that $\Gamma (L_M)$ embeds into $C^\infty (P)$ as homogeneous functions of degree $1$, we can interpret $\theta$ as a homogeneous $1$-form on $M$ of degree $1$. Denote the latter by $\Theta$. The assignment $\theta \mapsto d\Theta$ establishes a bijection between contact forms on $L_M$ and homogeneous symplectic forms of degree $1$ on $P$. The inverse bijection maps a degree $1$ homogeneous symplectic form $\omega$ to $\iota_{\mathcal E} \omega$, seen as an $L_M$-valued $1$-form on $M$, where $\mathcal E$ is the infinitesimal generator of the principal action $h$ (the fundamental vector field corresponding to the generator $1$ in the Lie algebra $\mathbbm R$ of $\mathbbm R^\times$). The reader may consult \cite{Vitagliano:holomorphic} for the reltionship between degree $1$ homogeneous symplectic forms and symplectic Atiyah forms.
	\end{rem}

\section{Line Bundles over Lie Groupoids}
\label{sec:lb}
From now on we will make extensive use of line bundles over Lie groupoids. We first recall the fundamentals of VB groupoids. Our main references for the latter are \cite{Gracia:VBgroupoids, Mackenzie:generaltheory}.
\subsection{VB Groupoids}
	A \emph{VB groupoid (VBG)} $(E\rightrightarrows E_M; G\rightrightarrows M)$ is a Lie groupoid in the category of vector bundles (or a vector bundle in the category of Lie groupoids), i.e.~a diagram
	\begin{equation*}
		\begin{tikzcd}
			E \arrow[r,shift left=0.5ex] \arrow[r, shift right=0.5 ex] \arrow[d] &E_M \arrow[d] \\
			G \arrow[r,shift left=0.5ex] \arrow[r, shift right=0.5 ex] & M
		\end{tikzcd}
	\end{equation*}
	where $E\rightrightarrows E_M$ and $G\rightrightarrows M$ are Lie groupoids, $E\to G$ and $E_M\to M$ are vector bundles, and all the structure maps of $E\rightrightarrows E_M$ are VB morphisms (see \cite{EspositoTortorellaVitagliano:infinitesimal_automorphism}). Abusing the notation, we denote the structure maps of both groupoids $G\rightrightarrows M$ and $E\rightrightarrows E_M$ in the same way: $s$, $t$, $m$, $i$, and $u$ for the source, target, multiplication, inverse and unit, respectively. We will always regard $E_M$ (resp.~$M$) as a submanifold of $E$ (resp.~$G$) via the unit. The inversion and the multiplication will be sometimes denoted ``$(-)^{-1}$'' and ``$\cdot$'', respectively. Moreover, we denote by $G^{(\bullet)}$ and $E^{(\bullet)}$ the nerves of $G$ and $E$. Notice that $E^{(\bullet)} \to G^{(\bullet)}$ is a \emph{simplicial vector bundle}, i.e.~a simplicial object in the category of vector bundles (or a vector bundle in the category of simplicial manifolds). The projections of $E^{(2)} = E \mathbin{{}_{s}\times_{t}} E$ onto the factors will be denoted $\mathrm{pr}_{1,2}$ (likewise for $G$).
	
	The \emph{core} of a VBG $(E\rightrightarrows E_M; G\rightrightarrows M)$ is $C=\ker(s\colon E \to E_M)|_M$, and the \emph{core-anchor} is the restriction $t|_C\colon C\to E_M$ of the target to the core $C$. Following \cite{DelHoyo:VBmorita}, the core-anchor will be often regarded as a $2$-term cochain complex of vector bundles over $M$ 
	\begin{equation*}
		\begin{tikzcd}
			0 \arrow[r]& C \arrow[r, "t|_C"]& E_M \arrow[r] & 0,
		\end{tikzcd}
	\end{equation*}
	called the \emph{core complex} of the VBG $E\rightrightarrows E_M$. We will assume that the core complex is concentrated in degrees $-1,0$. The value at $x\in M$ of the core complex
	\begin{equation*}
		\begin{tikzcd}
			0 \arrow[r]& C_x \arrow[r]& E_{M,x} \arrow[r]&0
		\end{tikzcd}
	\end{equation*}
	is also called the \emph{fiber of $E$ over $x$} \cite[Section 3]{DelHoyo:VBmorita}.
	
	When the core $C$ of the VBG $E\rightrightarrows E_M$ is trivial, i.e.~$C = 0$, the groupoid $E$ is isomorphic to an action groupoid \cite[Proposition 3.3.5]{DelHoyo:stacks} (see also, e.g., \cite[Example 3.2]{EspositoTortorellaVitagliano:infinitesimal_automorphism}), and, following \cite{DelHoyo:stacks, DelHoyo:VBmorita, EspositoTortorellaVitagliano:infinitesimal_automorphism}, we say that $E$ is a \emph{trivial core VBG}. \color{black}The action of $G\rightrightarrows M$ on $E_M$ is given by
	\begin{equation*}
		G\mathbin{{}_{s}\times} E_M \to E_M, \quad (g,v)\mapsto t(s_g^{-1}(v)),
	\end{equation*}
	and the isomorphism $G\mathbin{{}_{s}\times} E_M \cong E$ is given by
	\begin{equation*}
		G \mathbin{{}_{s}\times} E_M \to E, \quad (g,v)\mapsto s_g^{-1}(v).
	\end{equation*}
Moreover, when the core $C$ is trivial, all simplicial structure maps of the nerve $E^{(\bullet)}$ are regular VB morphisms, and we can apply the functor $D$. In this case $DE^{(\bullet)} \to G^{(\bullet)}$ is the nerve of a VB-groupoid $DE\rightrightarrows DE_M$ over $G\rightrightarrows M$ whose structure maps are $Ds$, $Dt$, $Dm$, $Di$ and $Du$ (up to the canonical VB isomorphism $D(E^{(2)}) \cong DE\mathbin{{}_{Ds}\times_{Dt}} DE$ identifying $\delta \in D(E^{(2)})$ with $(D \mathrm{pr}_1 (\delta), D \mathrm{pr}_2 (\delta))$) \cite{EspositoTortorellaVitagliano:infinitesimal_automorphism}.
	
	\begin{example}[The twisted dual VBG]
	\label{ex:adjoint}
	A VBG $(V\rightrightarrows V_M; G\rightrightarrows M)$ with core $C$ possesses a \emph{dual VBG} $(V^{\ast}\rightrightarrows C^{\ast};G\rightrightarrows M)$ (see \cite[Section 11.2]{Mackenzie:generaltheory}). Similarly, if $(E\rightrightarrows E_M; G\rightrightarrows M)$ is any VBG with trivial core, one can define a VBG $(V^{\dag}\rightrightarrows C^{\dag}; G\rightrightarrows M)$, where $V^{\dag}:=\operatorname{Hom}(V, E)$ and $C^{\dag}:= \operatorname{Hom}(C, E_M)$, as follows: for any $\psi\in V^{\dag}_g$ the source and target are defined by
	\begin{align*}
		\big\langle s(\psi), c \big\rangle = - s\big\langle \psi , 0_g \cdot c^{-1}\big\rangle, \quad \text{and} \quad \big\langle t(\psi), c' \big\rangle = t \big\langle \psi, c' \cdot 0_g \big\rangle, \quad c\in C_{s(g)}, \, c'\in C_{t(g)},
	\end{align*}
	Here and in the following we use $\langle -,-\rangle$ to denote the tautological $E$-valued pairing between $V$ and $V^{\dag}$, i.e.~$\langle -,-\rangle\colon V^{\dag}\times_G V\to E$, $\langle\phi,v\rangle := \phi(v)$, as well as the $E_M$-valued pairing between $C$ and $C^{\dag}$, etc.\\
	The unit of $V^\dag$ maps $\psi\in C^{\dag}_x$ to $\psi \circ \pr_C \in \operatorname{Hom}(V|_M, E_M) \subseteq V^{\dag}$, where the projection $\pr_C\colon V|_M \to C$ is given by the canonical splitting
	\begin{equation*}
		\begin{tikzcd}
			0 \arrow[r] &C \arrow[r] & V|_M \arrow[r, "s"] & V_M \arrow[r] \arrow[l, bend left, "u"]&0.
		\end{tikzcd}
	\end{equation*}
	For any two composable arrows $\psi\in V_g^{\dag}$ and $\psi'\in V_{g'}^{\dag}$ the multiplication is defined by
	\begin{equation*}
		\langle \psi  \psi', v v'\rangle =s_{gg'}^{-1}\left(g'^{-1} . s\langle \psi, v \rangle + s\langle \psi', v'\rangle \right)
	\end{equation*}
	for all composable arrows $v\in V_g$ and $v'\in V_{g'}$, where the dot ``$.$'' indicates the $G$-action on $E_M$. Finally, for any $\psi\in V_g^{\dag}$ the inverse is defined by
	\begin{equation*}
		\langle \psi^{-1}, v\rangle = - s_{g^{-1}}^{-1}(t\langle\psi, v^{-1} \rangle), \quad v\in V_{g^{-1}}.
	\end{equation*}
	We call the VBG $(V^{\dag}\rightrightarrows C^{\dag}; G\rightrightarrows M)$ the \emph{$E$-twisted dual VBG} of $V\rightrightarrows V_M$ (or simply the \emph{twisted dual VBG}, if there is no risk of confusion). Notice that the core of the twisted dual VBG is $V_M^{\dag}= \operatorname{Hom}(V_M, E_M)$, and the induced core complex is
	\begin{equation*}
		\begin{tikzcd}
			0 \arrow[r] & V_M^{\dag}\arrow[r, "t|_C^{\dag}"] & C^{\dag}\arrow[r]&0,
		\end{tikzcd}
	\end{equation*}
	where $t|_C^{\dag}$ is the \emph{twisted transpose map} to $t|_C$, i.e.~$t|_C^{\dag}(\psi)= \psi \circ t|_C$ for all $\psi\in V_M^{\dag}$.
	\end{example}

	A \emph{VBG morphism} $(F,f)$ between the VBGs $(W\rightrightarrows W_N; H\rightrightarrows N)$ and $(V\rightrightarrows V_M; G\rightrightarrows M)$ is a commutative diagram 
	\begin{equation*}
	{\scriptsize
		\begin{tikzcd}
			W \arrow[rr, shift left=0.5ex] \arrow[rr, shift right=0.5ex] \arrow[dd] \arrow[dr, "F"] & &W_N \arrow[dd] \arrow[dr, "F"] \\
			&  V \arrow[rr, shift left= 0.5ex, crossing over] \arrow[rr, shift right =0.5ex, crossing over] & &V_M \arrow[dd]\\
			H \arrow[rr, shift left=0.5ex] \arrow[rr, shift right=0.5ex] \arrow[dr, "f"] & & N \arrow[dr, "f"] \\ 
			&  G \arrow[from=uu, crossing over]\arrow[rr, shift left= 0.5ex] \arrow[rr, shift right =0.5ex] & &M
		\end{tikzcd}}
	\end{equation*}
	where $F$ and $f$ are Lie groupoid morphisms and $(F,f)$ are VB morphisms. In the sequel, sometimes, we only write $F$ for the VBG morphism $(F, f)$.
	
	The following Proposition is straightforward.
	
	\begin{prop}
		If $F\colon (W\rightrightarrows W_M)\to (V\rightrightarrows V_M)$ is a VBG morphism over the identity $\operatorname{id}\colon (G\rightrightarrows M)\to (G\rightrightarrows M)$ and $(E\rightrightarrows E_M; G\rightrightarrows M)$ is a VBG with trivial core, then the twisted transpose map $F^{\dag}\colon (V^{\dag}\rightrightarrows C^{\dag})\to (W^{\dag}\rightrightarrows D^{\dag})$ given by $F^{\dag}(\psi)=\psi \circ F$, for all $\psi\in V^{\dag}$, is a VBG morphism (over the identity), where $C$ and $D$ are the cores of $V$ and $W$, respectively.
	\end{prop}
	
	Any VBG morphism $(F,f)\colon (W\rightrightarrows W_N; H\rightrightarrows N)\to (V\rightrightarrows V_M; G\rightrightarrows M)$ induces a cochain map between the core complexes
	\begin{equation}\label{eq:fiber_VBGM}
		\begin{tikzcd}
			0 \arrow[r] &D \arrow[r] \arrow[d, "F"'] &W_N \arrow[r] \arrow[d, "F"]&0 \\
			0 \arrow[r]& C \arrow[r] &V_M \arrow[r] &0
		\end{tikzcd},
	\end{equation}
	where $C$ and $D$ are the cores of $V$ and $W$ respectively. Both components $F \colon D \to C$ and $F\colon W_N \to V_M$ are VB morphisms covering $f \colon N \to M$, and we will often consider the restriction 
	\begin{equation}\label{eq:fiber_x_VBGM}
		\begin{tikzcd}
			0 \arrow[r] &D_x \arrow[r] \arrow[d, "F_x"'] &W_{N, x} \arrow[r] \arrow[d, "F_x"]&0 \\
			0 \arrow[r]& C_{f(x)} \arrow[r] &V_{M, f(x)} \arrow[r] &0
		\end{tikzcd},
	\end{equation}
	of \eqref{eq:fiber_VBGM} to the fibers over $x \in N$ and $f(x) \in M$. 
	
	Following~\cite{DelHoyo:VBmorita} we now discuss the fundamentals of the Morita theory of VBGs.
	
	\begin{definition}[{\cite[Section 3]{DelHoyo:VBmorita}}]
		A VBG morphism $(F,f)$ is a \emph{VB Morita map} if $F$ is a Morita map (it then follows that $f$ is also a Morita map, see Theorem \ref{theo:caratterizzazioneVBmorita} below). Two VBGs $V$ and $W$ are \emph{Morita equivalent} if there exists a VBG $E$ and two VB Morita maps $V\leftarrow E\to W$.  We call \emph{VB stack} a Morita equivalence class of VBGs, and we denote by $[V_M/V]\to [M/G]$ the VB stack represented by the VBG $(V\rightrightarrows V_M; G\rightrightarrows M)$.
	\end{definition}
	
	As proved in~\cite[Theorem 3.5]{DelHoyo:VBmorita} there exists a very useful characterization of VB Morita maps that we will often use in the subsequent sections and we recall below.
	
	\begin{theo}[\textbf{Del Hoyo-Ortiz}  {\cite[Theorem 3.5]{DelHoyo:VBmorita}}]
		\label{theo:caratterizzazioneVBmorita}
		Let $(F, f)\colon (W\rightrightarrows W_N; H\rightrightarrows N)\to (V\rightrightarrows V_M; G\rightrightarrows M)$ be a VBG morphism and let $C$ and $D$ be the cores of $V$ and $W$ respectively. The following conditions are equivalent:
		\begin{enumerate}
			\item[\emph{i)}] $(F, f)$ is a VB Morita map;
			\item[\emph{ii)}] $f$ is a Morita map and, for all $x \in N$, the cochain map \eqref{eq:fiber_x_VBGM} between the fibers over $x\in N$ and $f(x) \in M$ is a quasi-isomorphism.
		\end{enumerate}
	\end{theo}

	The next two statements are easy consequences of Theorem \ref{theo:caratterizzazioneVBmorita} and will be useful in the sequel. To the best of our knowledge they didn't appear before anywhere.
	
	\begin{coroll}
		Let $(F,f)\colon (W\rightrightarrows W_N; H \rightrightarrows N)\to (V\rightrightarrows V_M; G\rightrightarrows M)$ be a VBG morphism between trivial core VBGs. The following conditions are equivalent
		\begin{itemize}
			\item[\emph{i)}] $(F,f)$ is a VB Morita map;
			\item[\emph{ii)}] $f \colon H \to G$ is Morita and $F\colon (W\to H)\to (V\to G)$ is a regular VB morphism;
			\item[\emph{iii)}] $f \colon H \to G$ is Morita and $F\colon (W_N\to N)\to (V_M\to M)$ is a regular VB morphism.
		\end{itemize}
	\end{coroll}
	
	Now, let $(W\rightrightarrows W_N; H\rightrightarrows N), (V\rightrightarrows V_M; G\rightrightarrows M)$ be Morita equivalent VBGs and let $(E'\rightrightarrows E'_N; H\rightrightarrows N), (E\rightrightarrows E_M; G\rightrightarrows M)$ be Morita equivalent trivial core VBGs. We want to show that the twisted dual VBGs (see Example \ref{ex:adjoint}) $W^\dag = \operatorname{Hom}(W, E')$ and $V^\dag = \operatorname{Hom}(V, E)$ are Morita equivalent as well. It is enough to consider the case when $W, V$ and $E', E$ are related by a VB Morita map. This is discussed in the following

	\begin{prop}
		\label{prop:VBMoritaadjoint}
		Let $W,V$, $E', E$ be as above, and let $(F,f)\colon (W\rightrightarrows W_N; H \rightrightarrows N)\to (V\rightrightarrows V_M; G\rightrightarrows M)$ and $(f_E,f)\colon (E'\rightrightarrows E'_N; H \rightrightarrows N)\to (E\rightrightarrows E_M; G\rightrightarrows M)$ be VB Morita maps. Then the twisted dual VBGs $W^{\dag}, V^\dag$ are Morita equivalent.
	\end{prop}
	
	\begin{proof}
	Denote by $D, C$ the cores of $W, V$ (while the cores of $E', E$ are trivial). Consider the pull-back VBGs $(f^\ast V^\dag \rightrightarrows f^\ast C^\dag; H \rightrightarrows N)$, whose core is $f^\ast V^\dag_M$. As $f \colon H \to G$ is a Morita map, from \cite[Corollary 3.7]{DelHoyo:VBmorita} we have that the canonical VBG morphism $f \colon f^{\ast}V^\dag \to V^\dag$ is a VB Morita map. Now, consider the map $F^{\dag}\colon f^{\ast}V^{\dag}\to W^{\dag}$ defined as follows: for any $(\psi,h)\in f^{\ast}V^{\dag}$, $h\in H$, $F^{\dag}(\psi)\in W^{\dag}_{h}$ is given by $\langle F^{\dag}(\psi), w\rangle := f_{E, h}^{-1}\langle\psi, F(w)\rangle$, $w\in W_{h}$. It is straightforward to check that $F^{\dag}$ is a VBG morphism covering the identity of $H$ and inducing the following cochain map on fibers: for all $x \in N$,
	
	\begin{equation}\label{eq:fiber_x_VBGM^dag}
		\begin{tikzcd}
			0 \arrow[r] &V^\dag_{M, f(x)} \arrow[r] \arrow[d, "F^\dag_x"'] &C^\dag_{f(x)} \arrow[r] \arrow[d, "F^\dag_x"]&0 \\
			0 \arrow[r]& W^\dag_{M, x} \arrow[r] & D^\dag_x \arrow[r] &0
		\end{tikzcd},
	\end{equation}
	where ($W_M^\dag := \operatorname{Hom}(W_M, E')$, $D^\dag := \operatorname{Hom}(D, E')$ and) all the arrows are the twisted transpose maps of the corresponding arrows in \eqref{eq:fiber_x_VBGM}. As \eqref{eq:fiber_x_VBGM} is a quasi-isomorphism from Theorem \ref{theo:caratterizzazioneVBmorita}, then \eqref{eq:fiber_x_VBGM^dag} is a quasi-isomorphism as well. It follows that $F^\dag$ is a VB Morita map fitting in the diagram of VB Morita maps
		\begin{equation*}
			\begin{tikzcd}
				& f^{\ast}V^{\dag} \arrow[dl, "F^\dag"'] \arrow[dr, "f"] \\
				W^{\dag} & & V^{\dag}
			\end{tikzcd},
		\end{equation*}
		whence the claim.
	\end{proof}
	
We conclude this section discussing natural isomorphisms in the VBG setting (see \cite[Section 6.1]{DelHoyo:VBmorita} for the special case of VBG morphisms covering the identity). Let $(F,f), (F',f')\colon (W\rightrightarrows W_N; H\rightrightarrows N)\to (V\rightrightarrows V_M; G\rightrightarrows M)$ be VBG morphisms. 
	\begin{definition}\label{def:LNT}
		A \emph{linear natural isomorphism} $(T, \tau) \colon (F,f) \Rightarrow (F',f')$ from $(F,f)$ to $(F',f)$ is a VB morphism $(T, \tau) \colon (W_N \to N) \to (V \to G)$ such that $T \colon F \Rightarrow F'$ is a natural isomorphism (whence $\tau \colon f \Rightarrow f'$ is a natural isomorphism as well). 
	\end{definition}

	\begin{theo}
		\label{theo:VBtransformation}
		Let $F, F'\colon W\to V$ be VBG morphisms covering the same map $f$ and let $(T, f) \colon F \Rightarrow F'$ be a linear natural isomorphism covering $f \colon N \to M \subseteq G$. For any $x\in N$, the map $h_T \colon W_{N,x} \to C_{f(x)}$, defined by $h_T(w)=T(w)- F(w)$, is a well-defined homotopy between the cochain maps on the fibers over $x, f(x)$ induced by $F, F'$:
			\begin{equation*}
				\begin{tikzcd}
					0 \arrow[r] &D_x \arrow[d, "F", shift left=0.5ex] \arrow[d, "F'"', shift right=0.5ex]\arrow[r]& W_{N,x} \arrow[d, "F", shift left=0.5ex] \arrow[d,"F'"', shift right=0.5ex] \arrow[r] \arrow[dl, "h_T"'] &0\\
					0 \arrow[r] &C_{f(x)}\arrow[r] &V_{M,{f(x)}} \arrow[r]&0
				\end{tikzcd}.
			\end{equation*}
		The assignment $(T, f) \mapsto h_T$ establishes a bijection between linear natural isomorphisms $(T, f) \colon F \Rightarrow F'$ and smooth homotopies between the cochain maps $F, F'$ from the core complex $0 \to D \to W_N \to 0$ of $W$ to the pull-back $0 \to f^\ast C \to f^\ast V_M \to 0$ of the core complex of $V$.
	\end{theo}
	\begin{proof}
		Let $(T, f)$ be a linear natural isomorphism as in the statement. For every $w \in W_M$, $T(w)$ is an arrow in $V$ from $F(w)$ to $F'(w)$, $T(w)\colon F(w) \to F'(w)$. Moreover, the naturality says that $F'(\omega) T(w) = T(w') F(\omega)$, for every arrow $\omega \colon w \to w'$ in $W$. The map $h_T$ in the statement is well-defined, indeed, for any $w \in W_N$, $s (h_T (w)) = s (F(w)) - s(T(w)) = F(w) - F(w) = 0$, i.e.~$h_T(w) \in C$. For the homotopy condition $h_T\circ T = F' - F$ in degree $-1$, take $d \in D_x$, i.e.~$d \colon 0_x^{W_N} \to t(d)$, hence, from both the linearity and the naturality, $T(t(d)) = F'(d) F(d^{-1})$, and compute
		\[
		\begin{aligned}
		h_T\big(t(d)\big) & = T\big(t (d)\big) - F\big(t(d)\big) = T\big(t(d)\big) - t\big(F(d)\big) = F'(d)  F(d^{-1}) - F(d)F(d^{-1}) \\
		&= \big(F'(d) - F(d)\big) \cdot \big(F(d^{-1}) - F(d^{-1})\big) = \big(F'(d) - F(d)\big) \cdot 0^{V}_{f(x)} = (F'-F)(d),
		\end{aligned}
		\]
		where we used that in a VBG both the multiplication and the unit are VB morphisms. The rest, including the homotopy condition $t \circ h_T = F' - F$ in degree $0$, is straightforward and we leave the details to the reader.
	\end{proof}

	Clearly, given a linear natural isomorphism between two VBG morphisms $F,F'$ covering the same map, then $F$ is a VB Morita map if and only if so is $F'$.
	
	\begin{rem}
	There is a version of Theorem \ref{theo:VBtransformation} for VBG morphisms covering different maps whose statement and proof need the technology of representations up to homotopy \cite{AC:RUTHs}, but we will not need this level of generality in the present paper.
	\end{rem}

	\subsection{LB Groupoids}
	From now on we will be mainly interested in VBGs which are line bundles.
	\begin{definition}
		A \emph{line bundle groupoid} (\emph{LBG}) is a VBG $(L\rightrightarrows L_M; G\rightrightarrows M)$ where $L\to G$ and $L_M\to M$ are line bundles. An \emph{LBG morphism} is a VBG morphism $(F,f)\colon (L'\rightrightarrows L'_{N}; H\rightrightarrows N)\to (L\rightrightarrows L_M, G\rightrightarrows M)$ between LBGs such that $F\colon L' \to L$ and $F\colon L'_N \to L_M$ are LB morphisms (i.e.~VB morphisms between line bundles with the additional property of being isomorphisms on fibers). 
	\end{definition}
	Notice that, for any $x\in M$, the restriction of the source map $s_x\colon L_x \to L_{M,x}$ is a surjective linear map between $1$-dimensional vector spaces, so its kernel is trivial. Hence the core of an LBG $(L\rightrightarrows L_M; G \rightrightarrows M)$ is automatically trivial, and all the simplicial structure maps of the nerve $L^{(\bullet)}$ are LB morphisms (i.e.~they are isomorphisms on fibers, and we can take, e.g., pull-back of sections along them). This is an easy consequence of the fact that, as recalled in the previous subsection, $G$ acts on $L_M$ and $L \rightrightarrows L_M$ is isomorphic to the action groupoid $G \ltimes L_M \rightrightarrows L_M$. We stress that we prefer not to take the action groupoid point of view, despite the fact that we will occasionally use the $G$-action on $L_M$. 
	
	Another advantage of the core of $L$ being trivial is that it makes sense to consider the VBG $(DL \rightrightarrows DL_M; G \rightrightarrows M)$. The core $C$ of $DL$ is canonically isomorphic to the Lie algebroid $A$ of $G$. Namely, from Equation \eqref{eq:DFcommutes}, the restriction of the symbol map $\sigma\colon DL \to TG$ to the core $C$ is a VB isomorphism $\sigma \colon C\to A$. The injectivity follows from the fact that a derivation of $L$ is completely determined by its symbol and its action on the pullback sections $s^{\ast}\lambda_M$, with $\lambda\in \Gamma(L_M)$. The surjectivity then follows from dimensional reasons. From Remark \ref{rem:Atiyahderivation} the inverse isomorphism $A \to C$ maps  $a= \tfrac{d}{d\varepsilon}|_{\varepsilon=0} \, g(\varepsilon)\colon x\to x(\varepsilon) \in A_x$ to $\tfrac{d}{d\varepsilon}|_{\varepsilon=0}  \, s_{g(\varepsilon)}^{-1} \circ s_x \colon L_x\to L_{g(\varepsilon)} \in C_x$, where $g(\varepsilon)$ is a curve in $s^{-1}(x)$ starting from $x$. In the following, we will always understand the isomorphism $C \cong A$ and say that $A$ is the core of the VBG $DL \rightrightarrows DL_M$. Accordingly, we indicate the core-anchor by $\mathcal{D}\colon A\to DL_M$.

	\begin{lemma}
		The core-anchor $\mathcal{D}\colon A\to DL_M$ agrees with the infinitesimal action of $A$ on $L_M$ corresponding to the $G$-action.
	\end{lemma}
	\begin{proof}
		The action of $G$ on $L_M$ can be seen as a Lie groupoid morphism $G\to \operatorname{GL}(L_M)$. Explicitly, for any $g\colon x \to y$ in $G$, we have the linear map $t \circ s_g^{-1}\colon L_{M,x} \to L_{M,y}$.
		Applying the Lie functor, we get the infinitesimal action $\mathcal D^A\colon A\to DL_M$. Now if $a=\tfrac{d}{d\varepsilon} |_{\varepsilon=0} \, g(\varepsilon)$, with $g(\varepsilon)\colon x\to x(\varepsilon)$ a curve in $s^{-1}(x)$ starting from $x \in M$, then 
		\begin{equation*}
			\mathcal{D}_a=Dt\left(\frac{d}{d\varepsilon}|_{\varepsilon=0} \, s_{g(\varepsilon)}^{-1}\circ s_x\right)= \frac{d}{d\varepsilon}|_{\varepsilon=0} \, t_{g(\varepsilon)}\circ s_{g(\varepsilon)}^{-1} \circ s_x \circ t_x^{-1}= \frac{d}{d\varepsilon} |_{\varepsilon=0} \, t_{g(\varepsilon)}\circ s_{g(\varepsilon)}^{-1} = \mathcal D^A_a,
		\end{equation*}
		where, in the second-last step, we used that source and target agree on units.
	\end{proof}

	Given a connection $\nabla$ on $L_M$ we will consider the linear form $F_\nabla=f_{\nabla}\circ \mathcal{D}\colon A \to \mathbbm R_M$, where $f_\nabla \colon DL_M \to \mathbbm R_M$ is the right splitting of \eqref{eq:Spencer_L} corresponding to $\nabla$. The following diagram summarizes the situation.
	\begin{equation}
		\label{eq:fnabla}
		\begin{tikzcd}
			& & A \arrow[d, "\mathcal{D}"] \arrow[dl, " F_\nabla"']\\
			0 \arrow[r] & \operatorname{End}(L_M)\arrow[r] & DL_M \arrow[r, "\sigma"] \arrow[l, bend left, "f_{\nabla}"] & TM \arrow[r] \arrow[l, "\nabla", bend left] & 0 
		\end{tikzcd}.
	\end{equation}
	
	\begin{rem}
	Applying the twisted dual VBG construction to the VBG $DL\rightrightarrows DL_M$ and the LBG $L\rightrightarrows L_M$ we obtain a VBG structure on the $1$-jet bundle: $J^1L \rightrightarrows A^{\dag}$.
	\end{rem}

	We now discuss the $L$-twisted dual and the gauge groupoids of Morita equivalent VBGs.
		
	\begin{prop}
		Let $V\rightrightarrows V_M$ and $W \rightrightarrows W_M$ be VBGs over the same Lie groupoid $G \rightrightarrows M$, let $C$ and $D$ be their cores and let $L\rightrightarrows L_M$ be an LBG over $G$. A VBG morphism $F\colon (W\rightrightarrows W_M) \to (V\rightrightarrows V_M)$ (over the identity) is a VB Morita map if and only if so is the twisted transpose map $F^{\dag}\colon (V^{\dag}=\operatorname{Hom}(V,L)\rightrightarrows C^{\dag}=\operatorname{Hom}(C, L_M)) \to (W^{\dag}=\operatorname{Hom}(W, L)\rightrightarrows D^{\dag}=\operatorname{Hom}(D, L_M))$.
	\end{prop}
	
	\begin{proof}
	Let $x \in M$ and let \eqref{eq:fiber_x_VBGM} be the map induced by $F$ on the fibers over $x$. Then the cochain map induced by $F^\dag$ on fibers is the twisted transpose of \eqref{eq:fiber_x_VBGM}. The claim now follows from $F^{\dag\dag} = F$.
	\end{proof}
	
	The next statement is analogous to \cite[Corollary 3.8]{DelHoyo:VBmorita}.
	
	\begin{prop}
		\label{prop:DFMorita}
		Let $(F,f)\colon (L'\rightrightarrows L'_N;H\rightrightarrows N)\to (L\rightrightarrows L_M; G\rightrightarrows M)$ be an LBG morphism. Then the VBG morphism $(DF,f)\colon (DL'\rightrightarrows DL'_N, H\rightrightarrows N)\to (DL\rightrightarrows DL_M, G\rightrightarrows M)$ is a VB Morita map if and only if $f\colon (H\rightrightarrows N)\to (G\rightrightarrows M)$ is a Morita map.
	\end{prop}
	
	\begin{proof}
	If $(DF, f)$ is VB Morita, then $f$ is Morita. For the converse, suppose that $f$ is a Morita map and let $x \in M$. The map induced by $DF$ between the fibers of $DL'$ and $DL$ over $x$ and $f(x)$ is
	\begin{equation}\label{eq:DF_fiber}
	\begin{tikzcd}
			0 \arrow[r] &A_{H,x} \arrow[r, "\mathcal D_H"] \arrow[d, "D F"'] & D_x L'_N \arrow[r] \arrow[d, "DF"]&0 \\
			0 \arrow[r]& A_{G,f(x)} \arrow[r, "\mathcal D_G", swap] &D_{f(x)} L_M \arrow[r] &0
		\end{tikzcd},
	\end{equation}
	where $A_H,A_G$ are the Lie algebroids of $H, G$. But \eqref{eq:DF_fiber} fits in the following short exact sequence
	\begin{equation}
		{\scriptsize
		\begin{tikzcd}[column sep={{{{2.5em,between origins}}}},
			row sep=1ex]
			&&&&&&&0\arrow[ddr]&&&&0\arrow[ddr]&&&&&&\\
			&\phantom{x}&&&&&&&&&&&&&&&&\\
			0\arrow[rrrr]&&&&0\arrow[ddr]\arrow[rrrr]\arrow[ddddd]&&&&\smash[b]{A_{H,x}}\arrow[ddr, "\mathcal{D}_H"]\arrow[rrrr, equal]\arrow[ddddd, "DF"' near end, swap]&&&&\smash[b]{A_{H,x}}\arrow[ddr, "\rho_H"]\arrow[rrrr]\arrow[ddddd, "df" near end]&&&&0&\\
			&\phantom{x}&&\phantom{x}&\phantom{x}&\phantom{x}&&&&&&&&&&&&\\
			&0\arrow[rrrr, crossing over]&&&&\mathbbm{R}\arrow[rrrr, crossing over]&&&&\smash[b]{D_xL_N'}\arrow[ddr]\arrow[rrrr, crossing over, "\sigma"]&&&&\smash[b]{T_xN}\arrow[ddr]\arrow[rrrr]&&&&0\\
			&&&&&&&0\arrow[ddr]&&&&0\arrow[ddr]&&&&&&\\
			&&&&&&\phantom{x}&&&&0&&&&0&&&\\
			0\arrow[rrrr]&&&&0\arrow[ddr]\arrow[rrrr]&&&&\smash[b]{A_{G,f(x)}}\arrow[ddr, "\mathcal{D}_G", swap]\arrow[rrrr, equal]&&&&\smash[b]{A_{G,f(x)}}\arrow[ddr, "\rho_G", swap]\arrow[rrrr]&&&&0&\\
			&&\phantom{x}&&&&&&&&&&&&&&&\\
			&0\arrow[rrrr]&&&&\mathbbm{R}\arrow[rrrr]\arrow[from=uuuuu, crossing over, equal]&&&&\smash[b]{D_{f(x)}L_M}\arrow[ddr]\arrow[rrrr, "\sigma", swap]\arrow[from=uuuuu, crossing over, "DF"', swap]&&&&\smash[b]{T_{f(x)}M}\arrow[ddr]\arrow[rrrr]\arrow[from=uuuuu, crossing over, "df"]&&&&0\\
			&&&&\phantom{x}&&&&&&&&&&&&&\\
			&&&&&&&&&&0&\phantom{x}&&&0&&&\\
		\end{tikzcd}}.
	\end{equation}
	Since $f$ is Morita, the vertical arrows in (both the leftmost and) the rightmost square form a quasi-isomorphism \cite[Corollary 3.8]{DelHoyo:VBmorita}. It follows that \eqref{eq:DF_fiber} is a quasi-isomorphism as well, hence $(DF, f)$ is VB Morita.
	\end{proof}
	
	We conclude this section with a short discussion on VB Morita equivalent LBGs.
	
	\begin{lemma}
	Let $(L\rightrightarrows L_M; G \rightrightarrows M)$ be an LBG, let $(V \rightrightarrows V_N; H \rightrightarrows N)$ be a VBG and let $(F, f) \colon V \to L$ be a VB Morita map. Then there exists an LBG $L'\to H$ and a VB Morita map $L' \to V$.
	\end{lemma}
	
	\begin{proof}
	Denote by $C \to N$ the core of $V$, and let $x \in N$. Consider the cochain map induced by $F$ on the fibers:
	\[
	\begin{tikzcd}
		0 \arrow[r] &C_x \arrow[r, "t|_C"] \arrow[d] & V_{N,x} \arrow[r] \arrow[d, "F"]&0 \\
		0 \arrow[r]& 0 \arrow[r] &L_{M,f(x)} \arrow[r] &0
	\end{tikzcd}.
	\]
	As $F$ is a VB Morita map, the latter is a quasi-isomorphism. Hence, the core-anchor $t|_C \colon C \to V_N$ is an injective vector bundle map whose cokernel is a line bundle. In particular, $V$ is a regular VBG, i.e. the core-anchor has constant rank (see \cite[Definition 6.3]{Gracia:VBgroupoids}). Therefore, from \cite[Lemma 6.10]{Gracia:VBgroupoids} it is isomorphic to the direct sum of an LBG $L' \to H$ and a VBG $W \to H$ whose core-anchor is bijective. The embedding $L' \to V$ is clearly a VB Morita map.
	\end{proof}
	
	\begin{coroll}\label{cor:VB_Mor_equiv_LBGs}
	Let $L_1, L_2$ be VB Morita equivalent LBGs. Then the VB Morita equivalence can be realized via an LBG, i.e. there exist an LBG $L'$ togheter with VB Morita maps
	\[
		\begin{tikzcd}
			& L' \arrow[dl] \arrow[dr] \\
			L_1 & & L_2
		\end{tikzcd}.
	\]
	\end{coroll}
	
	\begin{definition}
	An \emph{LB stack} is the VB stack represented by an LBG. 
	\end{definition}

\section{$0$-Shifted Contact Structures}\label{sec:0-shifted_cs}

In this section we introduce \emph{$0$-shifted contact structures}. This is not completely straightforward as we have to take care of all the aspects (A1)--(A5) in the dual definition of a contact structure listed in the Introduction in such a way to get a Lie groupoid compatible and Morita invariant definition. This effort is a good warm up for the next section on $+1$-shifted contact structures.
	
\subsection{$0$-Shifted Symplectic Structures}

We begin recalling the definition of $0$-shifted symplectic structures on a Lie groupoid (see, e.g., \cite[Example 2.18]{Cueca:shiftedstructures}, see also \cite[Section 5.5]{Hoffman:stacky}). Let $G\rightrightarrows M$ be a Lie groupoid, let $A$ be the Lie algebroid of $G$, and let $\rho \colon A \to TM$ be the anchor. The simplicial structure of the nerve $G^{(\bullet)}$ determines a cochain complex
	\begin{equation*}
		\begin{tikzcd}
			0 \arrow[r] & \Omega^{\bullet}(M) \arrow[r, "\partial"] &\Omega^{\bullet}(G)\arrow[r, "\partial"] &\Omega^{\bullet}(G^{(2)})\arrow[r,"\partial"] &\cdots
		\end{tikzcd}
	\end{equation*}
	where the differential $\partial$ is the alternating sum of the pullbacks along the face maps. For a differential form $\omega\in \Omega^\bullet(M)$,  the cocycle condition $\partial \omega=0$ means that $s^{\ast}\omega=t^{\ast}\omega$, and for a $2$-form it follows that $\im \rho \subseteq \ker \omega$. The latter inclusion can be rephrased by saying that, for all $x \in M$, the flat map $\omega \colon T_x M \to T^\ast_x M$ gives rise to a ($-1$-shifted) cochain map between the fibers of the tangent and the cotangent groupoids of $G$, i.e.~the following diagram commutes:
\begin{equation}\label{eq:cm_0-shifted}
			\begin{tikzcd}
				0 \arrow[r] &A_x \arrow[r, "\rho"] \arrow[d] & T_xM \arrow[r] \arrow[d, "\omega"] & 0 \arrow[r] \arrow[d] & 0 \\
				0 \arrow[r] & 0 \arrow[r] & T_x^{\ast}M \arrow[r, "\rho^{\ast}"'] & A_x^{\ast} \arrow[r] &0
			\end{tikzcd},
		\end{equation}
	where $\rho^\ast\colon T^\ast M \to A^\ast$ is the transpose of $\rho$. This slightly more sophisticated point of view does actually generalize to higher shifted symplectic structures \cite{Cueca:shiftedstructures, Getzler:shifted} and it is also of inspiration for the case of $+1$-shifted contact structures.
	
	\begin{definition}
		\label{def:0-symplectic}
		A \emph{$0$-shifted symplectic structure} on the Lie groupoid $G$ is a $2$-form $\omega\in \Omega^2(M)$ such that $\partial \omega=0$, $d\omega=0$ and the cochain map \eqref{eq:cm_0-shifted} is a quasi-isomorphism for all $x \in M$, in other words 1) $\rho$ is injective and 2) $\im \rho = \ker \omega$ (or, equivalently, $\omega$ induces an isomorphism $T M / \im \rho \cong \ker \rho^\ast$ in the \emph{transverse direction}).
	\end{definition}
	
	\begin{rem}
	The notion of $0$-shifted symplectic structure is Morita invariant in the following sense. Let $f \colon H \to G$ be a Morita map between Lie groupoids $H \rightrightarrows N$ and $G\rightrightarrows M$. In particular $f$ induces a cochain map $f^\ast \colon (\Omega^\filleddiamond (G^{(\bullet)}), \partial) \to (\Omega^\filleddiamond (H^{(\bullet)}), \partial)$, actually a quasi-isomorphism \cite[Corollary 3]{B2004}. Even more, the assignment $\omega \mapsto f^\ast \omega$ establishes a bijection between $0$-shifted symplectic structures on $G$ and $0$-shifted symplectic structures on $H$. Indeed, from the quasi-isomorphism property, $f^\ast$ maps bijectively $\partial$-closed $2$-forms on $M$ to $\partial$-closed $2$-forms on $N$. Moreover, $f^\ast$ preserves both properties 1) and 2) in Definition \ref{def:0-symplectic} (see \cite[Lemma 2.28]{Cueca:shiftedstructures}, the simplest possible case: $m = 0$, $n = 1$). Accordingly, $\omega$ can be regarded as a ($0$-shifted symplectic) structure on the differentiable stack $[M/G]$. 
	\end{rem}

\subsection{$0$-Shifted Symplectic Atiyah Forms}
			
	Now we provide a first translation of Definition \ref{def:0-symplectic} to the Contact Geometry realm, using the language of Atiyah forms. We begin with an LBG $(L\rightrightarrows L_M; G\rightrightarrows M)$. The simplicial structure of the nerve $L^{(\bullet)}$ determines a cochain complex of Atiyah forms
	\begin{equation}
		\label{eq:Atiyah_complex}
		\begin{tikzcd}
			0 \arrow[r] & \OA^{\bullet}(L_M)\arrow[r, "\partial"] & \OA^{\bullet}(L) \arrow[r, "\partial"] & \OA^{\bullet}(L^{(2)}) \arrow[r, "\partial"] & \cdots
		\end{tikzcd}
	\end{equation}
	where the differential is the alternating sum of the pullbacks of Atiyah forms along the face maps of $L^{(\bullet)}$ (recall from the previous section that the face maps are regular VB morphisms in this case). 
 Exactly as for plain differential forms, for an Atiyah form $\omega\in \OA^\bullet(L_M)$ on $L_M$, the cocycle condition $\partial \omega=0$ means that $s^{\ast}\omega=t^{\ast}\omega$ and for an Atiyah $2$-form it follows that $\im \mathcal{D} \subseteq \ker \omega$. Again we can rephrase the latter inclusion by saying that $\omega \colon D_x L_M \to J^1_x L_M$ is a cochain map between the fibers of the VB-groupoids $DL \rightrightarrows DL_M$ and $J^1 L \rightrightarrows A^\dag$ over $x \in M$ (up to a shift): 
	\begin{equation}
			\label{eq:0-Atiyah}
			\begin{tikzcd}
				0 \arrow[r] &A_x \arrow[r, "\mathcal{D}"] \arrow[d] & D_xL_M \arrow[r] \arrow[d, "\omega"] & 0 \arrow[r] \arrow[d] & 0 \\
				0 \arrow[r] & 0 \arrow[r] & J^1_x L_M \arrow[r, "\mathcal{D}^{\dag}"'] & A_x^{\dag} \arrow[r] &0
			\end{tikzcd}.
		\end{equation}
		
			\begin{definition}
		A \emph{$0$-shifted symplectic Atiyah form} on the LBG $L$ is an Atiyah $2$-form $\omega \in \OA^2(L_M)$ such that $\partial\omega=0$, $\dA\omega=0$ and the cochain map \eqref{eq:0-Atiyah} is a quasi-isomorphism for all $x \in M$, i.e.~1)$\mathcal D$ is injective and 2) $\im \mathcal D = \ker \omega$.
	\end{definition}
	
	The notion of $0$-shifted symplectic Atiyah form is Morita invariant in the sense that, given a VB Morita map $F\colon L' \to L$ between LBGs, the assignment $\omega \mapsto F^\ast \omega$ establishes a bijection between $0$-shifted symplectic Atiyah forms on $L$ and $0$-shifted symplectic Atiyah forms on $L'$. Instead of proving this directly, we prefer to prove the Morita invariance of the equivalent notion of $0$-shifted contact structure (Theorem \ref{theo:Mor_inv_0-shift}). The Morita invariance of the notion of $0$-shifted symplectic Atiyah form will then easily follows from Theorem \ref{theor:0-contact=0-Atiyah} below.
	
	\subsection{Morita Kernel}
	
Our next aim is to provide a definition of $0$-shifted contact structure. Our strategy is translating the dual definition of a standard contact structure in terms of a contact form to the realm of Lie groupoids, taking care of all the aspects (A1)--(A5) listed in the Introduction, in such a way to guarantee Morita invariance at each step. So let $(L\rightrightarrows L_M;G\rightrightarrows M)$ be an LBG (aspect (A1)). Like for Atiyah forms, the simplicial structure of the nerve $L^{(\bullet)}$ determines a complex of line bundle valued forms:
	\begin{equation}
		\label{eq:complex_valuedform}
		\begin{tikzcd}
			0 \arrow[r]& \Omega^{\bullet}(M,L_M) \arrow[r, "\partial"] & \Omega^{\bullet}(G,L)\arrow[r, "\partial"] & \Omega^{\bullet}(G^{(2)},L^{(2)})\arrow[r, "\partial"] & \cdots
		\end{tikzcd}
	\end{equation}
	where the differential $\partial$ is the alternating sum of the pullbacks along the face maps of $L^{(\bullet)}$ (of vector valued forms). 
	
	\begin{rem}\label{rem:partial_components}
	Let $\omega \rightleftharpoons (\omega_0 , \omega_1) \in \Omega^\filleddiamond_D (L^{(\bullet)})$ (so that $\omega_0 \in \Omega^{\filleddiamond-1}(G^{(\bullet)}, L^{(\bullet)})$ and $\omega_1 \in \Omega^\filleddiamond (G^{(\bullet)}, L^{(\bullet)})$). We stress for future reference that $\partial \omega \rightleftharpoons (\partial \omega_0, \partial \omega_1)$.
	\end{rem}
	
	Now take an $L_M$-valued $1$-form $\theta\in \Omega^1(M,L_M)$ such that $\partial\theta=0$ (aspect (A2)). This means that $s^{\ast}\theta=t^{\ast}\theta$. It follows that $\im \rho \subseteq \ker \theta$. Similarly as we did for Atiyah $2$-forms, we now rephrase the latter inclusion, in an apparently oversophisticated way, by noting that it is equivalent to $\theta \colon T M \to L_{M}$ being a cochain map between the core complexes of the tangent groupoid $TG$ and that of the LB groupoid $L$:
	\begin{equation}\label{eq:theta_cm}
	\begin{tikzcd}
				0 \arrow[r] &A \arrow[r, "\rho"] \arrow[d] & T M \arrow[r] \arrow[d, "\theta"]  & 0 \\
				0 \arrow[r] & 0 \arrow[r] & L_{M} \arrow[r]  &0
			\end{tikzcd}.
	\end{equation}

	We would like to take the kernel of $\theta$. However, first of all, $\ker \theta$ might not be a well-defined (smooth) distribution, as its dimension might jump (see Example \ref{exmpl:0-shifted_nonreg} below). More importantly, it turns out that, in order to have Morita invariance of the kernel (aspect (A3)), it is actually necessary to take the \emph{homotopy kernel} of $\theta$ as a cochain map \eqref{eq:theta_cm}. In other words, we have to take the mapping cone of  \eqref{eq:theta_cm} which is (up to a conventional sign):
\begin{equation}
		\label{eq:Morita0-kernel}
		\begin{tikzcd}
			0 \arrow[r] & A\arrow[r, "\rho"] & TM \arrow[r, "\theta"] & L_{M}\arrow[r] &0
		\end{tikzcd},
	\end{equation}
where the non-trivial terms are concentrated in degrees $-1,0,1$.

	\begin{rem}\label{rem:MK_RUTH}
	The complex \eqref{eq:Morita0-kernel} of a $\theta \in \Omega^1 (M, L_M)$ such that $\partial \theta = 0$ is not just a cochain complex. Actually it can be promoted to a \emph{representation up to homotopy} of $G$ \cite{AC:RUTHs, Gracia:VBgroupoids}. See Appendix \ref{app:RUTH} for a reminder about representations up to homotopy (RUTH in what follows).
	Let $(L \rightrightarrows L_M ; G \rightrightarrows M)$ be an LBG and let $\theta \in \Omega^1 (M, L_M)$ be such that $\partial \theta = 0$. Choose once for all an Ehresmann connection on $G$,	 and use it to promote the core complex of $TG$
	\begin{equation*}
	\begin{tikzcd}
			0 \arrow[r] & A \arrow[r, "\rho"] & TM \arrow[r]&0
		\end{tikzcd}
	\end{equation*}
	to the adjoint RUTH $(C(G; A \oplus TM), \partial^{\mathrm{Ad}})$, see \cite{AC:RUTHs, Gracia:VBgroupoids}. We also have the RUTH $(C(G; L_M), \partial^{L})$ on the core complex of $L$
	\[
\begin{tikzcd}
			0 \arrow[r] & L_M \arrow[r] & 0
		\end{tikzcd}
	\]
	coming from the action of $G$ on $L_M$. The $1$-form $\theta$ can be seen as a degree $0$ graded vector bundle morphism $A \oplus TM \to 0 \oplus L_M$ (Diagram \eqref{eq:theta_cm}) that we call $\Phi_0$.
	Then, using \eqref{eq:struct_sect_RUTH_mor}, it is easy to see that, setting $\Phi_k = 0$ for $k > 0$, defines a morphism of RUTHs:
	\[
	\Phi\colon \left(C(G; A \oplus TM), \partial^{\mathrm{Ad}} \right) \to \left(C(G; L_M), \partial^{L}\right).
	\]
	The mapping cone of $\Phi$ now completes the complex \eqref{eq:Morita0-kernel} to a RUTH on $A \oplus TM \oplus L_M$, as claimed (see \cite[Example 3.21]{AC:RUTHs} for why the mapping cone of a morphism of RUTHs is a RUTH as well). We leave the straightforward details to the reader.
	\end{rem}

	\begin{definition}\label{def:MK}
		The RUTH on $A\oplus TM \oplus L_M$ obtained in the previous remark is called the \emph{Morita kernel} of $\theta$. The value of \eqref{eq:Morita0-kernel} at the point $x \in M$ is the \emph{Morita kernel at $x$}.
	\end{definition}

	\begin{rem}\label{rem:0_inv_mc}
	Denote by $\{R_k\}_{k \geq 0}$ the structure operators of the Morita kernel. We stress for future reference that, from \eqref{eq:struct_sect_RUTH} - cases $k = 1,2$, it follows that, for every $g \colon x \to y$ in the groupoid $G$, the vertical arrows in
	\begin{equation}
			\label{eq:mor_ker_orbit}
			\begin{tikzcd}
				0 \arrow[r] & A_x \arrow[r, "\rho"] \arrow[d, "R_1(g)"] & T_xM \arrow[r, "\theta"] \arrow[d, "R_1(g)"] & L_{M,x} \arrow[r] \arrow[d, "R_1(g)"] & 0 \\
				0 \arrow[r] & A_y \arrow[r, "\rho"'] & T_yM \arrow[r, "\theta"'] & L_{M,y} \arrow[r] & 0
			\end{tikzcd}
		\end{equation}
	are a quasi-isomorphism (actually $R_1$ induces a $G$-action in the $R_0$-cohomology). In particular, the complexes at different points of the same orbit of $G$ are quasi-isomorphic. 
	
	We conclude this remark stressing that, unfortunately, the equivalence between RUTHs concentrated in non-negative degrees and higher VBGs \cite{dHT2021} does not apply to the Morita kernel. Indeed, the latter should be thought of as being concentrated in degrees $-1,0,1$. Only in this way it comes with an obvious morphism of RUTHs into the adjoint RUTH and can be duly seen as a \emph{homotopy kernel}. This informal discussion suggests that the Morita kernel is actually a \emph{higher derived stack} rather than a \emph{higher stack}. Exploring this speculation goes beyond the scopes of the present paper (but see Section \ref{sec:Morita_kernel} about the Morita kernel of a $+1$-shifted contact structure, where the situation is surprisingly simpler, indeed the Morita kernel is actually a VB stack in the $+1$-shifted case).
	\end{rem}

	The terminology in Definition \ref{def:MK} is motivated by the fact that the Morita kernel is Morita invariant up to quasi-isomorphisms, as desired, in the sense of the following
	\begin{prop}\label{prop:mor_ker_mor_inv}
		Let $(F,f)\colon (L'\rightrightarrows L'_N; H\rightrightarrows N)\to (L\rightrightarrows L_M; G\rightrightarrows M)$ be a VB Morita map between LBGs, and let $\theta\in \Omega^1(M;L_M)$ be an $L_M$-valued $1$-form such that $\partial\theta=0$. Then, for all $x \in N$, $(F, f)$ induces a quasi-isomorphism between the Morita kernel of $\theta' :=F^{\ast}\theta \in \Omega^1(N,L'_N)$ at the point $x\in N$ and the Morita kernel of $\theta$ at the point $f(x)\in M$.
	\end{prop}
	\begin{proof}
		The statement can be proved by bare hands. We propose a more conceptual proof that might conjecturally be generalized to more complicated situations (hopefully higher shifted contact structures on higher groupoids). Denote by $A_H, A_G$ the Lie algebroids of $H, G$ and let $\rho_H, \rho_G$ be their anchors.
		Now, fix $x \in N$, and consider the following diagram of cochain complexes and cochain maps:	
		\begin{equation}\label{eq:funct_mapping_cone}
		{\scriptsize
		\begin{tikzcd}
				0 \arrow[rr] & & A_{H,x} \arrow[rr, "\rho_H" near start] \arrow[dd] \arrow[dr, "df"] & & T_xN \arrow[rr] \arrow[dd, "\,\theta'" near end] \arrow[dr, "df"] &  & 0 \\
				& 0 \arrow[rr, crossing over] & & A_{G,f(x)} \arrow[rr, crossing over, "\rho_G" near start] & & T_{f(x)}M \arrow[rr, crossing over] \arrow[dd, "\,\theta" near end]&  & 0 \\
				0 \arrow[rr] & &0\arrow[rr] & & L'_{N, x} \arrow[rr] \arrow[dr, "F"]&  & 0 \\
				& 0 \arrow[rr] & & 0 \arrow[from=uu, crossing over] \arrow[rr] \arrow[ul, equal]& & L_{M, f(x)} \arrow[from=uu, crossing over] \arrow[rr]  &  & 0 
			\end{tikzcd}}
		\end{equation}
As $(F, f)$ is VB Morita, then $f$ is a Morita map and, from \cite[Corollary 3.8]{DelHoyo:VBmorita}, $(df, f) \colon TH \to TG$ is a VB Morita map between the tangent groupoids. It follows that the diagonal maps in the upper square of \eqref{eq:funct_mapping_cone} are a quasi-isomorphism. As $F$ is a VB Morita map then it is an LB morphism, and the diagonal maps of the lower square are also a quasi-isomorphism. Then, from standard Homological Algebra, the induced cochain map between the mapping cones (of the back and front squares):
		\begin{equation*}
			\begin{tikzcd}
				0 \arrow[r] & A_{H,x} \arrow[r, "\rho_H"] \arrow[d, "df"] & T_xN \arrow[r, "\theta'"] \arrow[d, "df"] & L'_{N,x} \arrow[r] \arrow[d, "F"] & 0 \\
				0 \arrow[r] & A_{G,f(x)} \arrow[r, "\rho_G"'] & T_{f(x)}M \arrow[r, "\theta"'] & L_{M,f(x)} \arrow[r] & 0
			\end{tikzcd}
		\end{equation*}
		is a quasi-isomorphism as well.
	\end{proof}
	
	\begin{rem}\label{rem:Mor_ker_theta_neq_0}
	When $\theta_x \neq 0$ for all $x \in M$, then $K_\theta := \ker \theta \subseteq TM$ is a well-defined distribution on $M$. In this case, the Morita kernel can be replaced, up to quasi-isomorphisms, by a simpler RUTH. Specifically, it is easy to see that the structure operators of the adjoint RUTH, restrict to the $2$-term complex of vector bundles
	\begin{equation}\label{eq:ker_RUTH}
		\begin{tikzcd}
			0 \arrow[r] & A \arrow[r, "\rho"] & K_{\theta} \arrow[r]&0
		\end{tikzcd}.
	\end{equation}
	Additionally, the cochain map
		\begin{equation}\label{eq:Mor_ker_theta_neq_0}
			\begin{tikzcd}
				0 \arrow[r] &A \arrow[r, "\rho"] \arrow[d, equal]& K_{\theta} \arrow[r] \arrow[d, swap, "\mathrm{in}"] & 0 \arrow[r] \arrow[d] & 0 \\
				0 \arrow[r] & A \arrow[r, "\rho"'] & TM \arrow[r, "\theta"'] & L_{M} \arrow[r] &0
			\end{tikzcd},
		\end{equation}
		where $\mathrm{in}\colon K_\theta \to TM$ denotes the inclusion, is a strict RUTH quasi-isomorphism.
		\end{rem}
		
		\subsection{Morita Curvature}
	
	Now take again $\theta \in \Omega^1 (M, L_M)$ such that $\partial \theta = 0$. Our next aim is defining the ``curvature'' of $\theta$. When $\ker \theta$ is not a well-defined distribution, there is no hope of defining the curvature in the usual way. Moreover, we would like to have a Morita invariant notion of curvature (aspect (A4)). In fact, in this setting, the role of the curvature is played by a RUTH morphism between the Morita kernel of $\theta$ and its \emph{twisted dual RUTH}. Begin noticing that, given a RUTH $E$ of $G \rightrightarrows M$ with structure operators $\{R_k\}_{k \geq 0}$, we get an \emph{$L$-twisted dual RUTH} $E^\dag$ with structure operators $\{R^\dag_k\}_{k \geq 0}$ by putting $E^\dag = \operatorname{Hom}(E, L_M)$ and
	\[
	R^\dag_k (g_1, \ldots, g_k) = R_k (g_k^{-1}, \ldots, g_1^{-1})^\dag \colon (E^\dag)^\bullet \to (E^\dag)^{1+k-\bullet}
	\]
	(see \cite[Example 3.19]{AC:RUTHs} for the analogous construction of the \emph{dual RUTH} but beware of the sign as we follow a different convention!). Hence the complex:
	\begin{equation}\label{eq:adjoint_mor_ker_0}
		\begin{tikzcd}
			0 \arrow[r] & \mathbbm{R}_M \arrow[r, "\theta^{\dag}"] & T^{\dag}M \arrow[r, "\rho^{\dag}"] & A^{\dag} \arrow[r] & 0
		\end{tikzcd},
	\end{equation}
	where $\mathbbm{R}_M:=M \times \mathbbm{R}$ is the trivial line bundle over $M$, can be completed to the $L$-twisted dual RUTH of the Morita kernel. Next, we define a morphism of RUTHs
\[
\Phi \colon C(G; A \oplus TM \oplus L_M) \to  C(G; \mathbbm R_M \oplus T^\dag M \oplus A^\dag)
\]
between the Morita kernel and its twisted dual RUTH.
	
	\begin{prop}
		Let $\theta\in \Omega^1(M,L_M)$ be such that $\partial\theta=0$ and let $\nabla$ be a connection on $L_M$. The vertical arrows in the diagram
		\begin{equation}
			\label{eq:0moritakernel}
			\begin{tikzcd}
				0 \arrow[r] & A \arrow[r, "\rho"] \arrow[d, "-F_\nabla"] & TM \arrow[r, "\theta"] \arrow[d, "d^{\nabla}\theta"] & L_{M} \arrow[r] \arrow[d, "F_{\nabla}^{\dag}"] & 0 \\
				0 \arrow[r] & \mathbbm{R} \arrow[r, "\theta^{\dag}"'] & T^{\dag}M \arrow[r, "\rho^{\dag}"'] &A^{\dag} \arrow[r] & 0
			\end{tikzcd}
		\end{equation}
		form a cochain map.
	\end{prop}
	\begin{proof}
		From the skew symmetry of $d^{\nabla}\theta$ it is enough to prove that the first square in \eqref{eq:0moritakernel} commutes. To do this, consider the Atiyah form $\omega \rightleftharpoons(\theta,0)\in \OA^2(L_M)$. As $\partial \theta = 0$, we also have $\partial \omega = 0$. It follows that $\im\mathcal{D}\subseteq \ker\omega$. Now the claim easily follows from Corollary \ref{coroll:kernelofAtiyahform}.\emph{ii)}.
	\end{proof}
	
	The cochain map \eqref{eq:0moritakernel} will be denoted $\Phi_0$ and it is the $0$-th component of the morphism of RUTHs
\[
\Phi \colon C(G; A \oplus TM \oplus L_M) \to  C(G; \mathbbm R_M \oplus T^\dag M \oplus A^\dag)
\]
between the Morita kernel and its twisted dual RUTH which we are looking for. To see this we will also need the $2$-term RUTH corresponding to the VBG $(DL \rightrightarrows DL_M; G \rightrightarrows M)$ (see Appendix \ref{app:RUTH}).

Besides $\Phi_0$, $\Phi$ possesses just one more component:
\[
\Phi_1 \colon G \to \operatorname{Hom}\left(s^\ast (A \oplus TM \oplus L_M), t^\ast (\mathbbm R_M \oplus T^\dag M \oplus A^\dag)\right)
\]
consisting of two summands $\Phi_1 \colon G \to \operatorname{Hom} (s^\ast TM, t^\ast \mathbbm R_M)$ and $\Phi_1 \colon G \to \operatorname{Hom}(s^\ast L_M, t^\ast T^\dag M)$, given by
\[
\Phi_1 (g) v := \nabla_{R_1^T (g) v}- R_1^D (g) \nabla_v  \in \operatorname{End} L_{M, t(g)} \cong \mathbbm R, \quad g \in G, \quad v \in T_{s(g)} M,
\]
and
\[
\big\langle \Phi_1 (g) \lambda, w \big\rangle := g. \big((\Phi_1 (g^{-1})w)  \lambda  \big) \in L_{M, t(g)}, \quad g \in G, \quad \lambda \in L_{M,s(g)}, \quad w \in T_{t(g)}M.
\]

\begin{prop}\label{prop:MR_RUTH_morph}
The maps $\{\Phi_0, \Phi_1\}$ are the components of a morphism of RUTHs
\[
\Phi \colon C(G; A \oplus TM \oplus L_M) \to  C(G; \mathbbm R_M \oplus T^\dag M \oplus A^\dag),
\]
where all the other components are trivial.
\end{prop}

The proof of Proposition \ref{prop:MR_RUTH_morph} is a long computation. We present some of the details in Appendix \ref{app:RUTH_mor}.

\begin{definition}\label{def:Mor_curv_0-shift}
The morphism of RUTH $\Phi$ in Proposition \ref{prop:MR_RUTH_morph} is called the \emph{Morita curvature of} $\theta$. The value of the cochain map \eqref{eq:0moritakernel} at the point $x\in M$ is the \emph{Morita curvature at $x$}.
\end{definition}

The terminology in Definition \ref{def:Mor_curv_0-shift} is motivated by the fact that the Morita curvature is Morita invariant in an appropriate sense (see Proposition \ref{prop:mor_curv_mor_inv} below, see also Proposition \ref{prop:0_MC_conn}).

	\begin{rem}\label{rem:MC_RUTH}
	We stress here for future reference that it follows from the RUTH morphism identities \eqref{eq:struct_sect_RUTH_mor} that, for every arrow $g \colon x \to y$ in the groupoid $G$, the diagram
\begin{equation}\label{eq:mor_curv_orbit}
	{\scriptsize
	\begin{tikzcd}
		0 \arrow[rr] & & A_x \arrow[rr] \arrow[dd, pos=0.2, swap, "-F_{\nabla}"] \arrow[dr, "R_1(g)"] & & T_xM \arrow[rr] \arrow[dd, pos=0.2, swap, "d^{\nabla}\theta"] \arrow[dr, "R_1(g)"] & & L_{M,x} \arrow[rr] \arrow[dd, pos=0.2, swap, "F_{\nabla}^\dagger"] \arrow[dr, "R_1(g)"] & & 0 \\
		& 0 \arrow[rr, crossing over] & & A_y \arrow[rr, crossing over] & & T_yM \arrow[rr, crossing over] & & L_{M,y} \arrow[rr] & & 0 \\
		0 \arrow[rr] & & \mathbbm{R} \arrow[rr] & & T_x^{\dag}M \arrow[rr] \arrow[dr, "R_1^{\dag}(g)"] & & A_x^{\dag} \arrow[rr] \arrow[dr, "R_1^{\dag}(g)"] & & 0 \\
		& 0 \arrow[rr] & & \mathbbm{R} \arrow[from=uu, crossing over, pos=0.2, "-F_{\nabla}"] \arrow[rr] \arrow[ul,equal] & & T_y^{\dag}M \arrow[from=uu, crossing over, pos=0.2, "d^{\nabla}\theta"] \arrow[rr] & & A^{\dag}_y \arrow[from=uu, crossing over, pos=0.2, "F_{\nabla}^\dagger"] \arrow[rr] & & 0 
	\end{tikzcd}},
\end{equation}
although not being commutative, does actually define a commutative diagram in the cohomology of the Morita kernel (and its twisted dual RUTH). As the diagonal arrows are quasi-isomorphisms (Remark \ref{rem:0_inv_mc}), we conclude that the Morita curvatures at different points of the same orbit of $G$ are related by isomorphisms in cohomology. 
\end{rem}
	
	In order to define the Morita curvature, we fixed a connection. However, the Morita curvature is independent of the choice of the connection $\nabla$ up to homotopies between the $0$-th components. More precisely, we have the following
	\begin{prop}\label{prop:0_MC_conn}
		Let $\nabla$ and $\nabla'$ be two connections on $L_M$. Then there is a homotopy between the $0$-th components of the corresponding Morita curvatures.
	\end{prop}
	\begin{proof}
		The difference $h = \alpha_{\nabla, \nabla'} := \nabla - \nabla'$ can be seen as a $1$-form on $M$: $h \in \Omega^1(M)$. The diagonal arrows in the diagram
		\begin{equation*}
		{\footnotesize
			\begin{tikzcd}
				0 \arrow[rr] & & A \arrow[rr, "\rho"] \arrow[dd, "-F_\nabla", shift left=0.5ex] \arrow[dd, "-F_{\nabla'}"', shift right=0.5ex] & &  TM \arrow[ddll, "h"] \arrow[rr, "\theta"] \arrow[dd, "d^{\nabla}\theta", shift left=0.5ex] \arrow[dd, "d^{\nabla'}\theta"', shift right=0.5ex] & & L_{M} \arrow[ddll, "-h^\dag"] \arrow[r] \arrow[dd, "F^\dag_\nabla", shift left=0.5ex] \arrow[dd, "F^\dag_{\nabla'}"', shift right=0.5ex] & 0 \\ \\
				0 \arrow[rr] & &  \mathbbm{R}_M \arrow[rr, "\theta^{\dag}"'] & &  T^{\dag}M \arrow[rr, "\rho^{\dag}"'] & & A^{\dag} \arrow[r] & 0
			\end{tikzcd}}
		\end{equation*}
		are a homotopy between the $0$-th components of the Morita curvatures. Indeed,  in degree $-1$, for every $x \in M$, take $a\in A_x$ and $\lambda \in \Gamma (L_M)$. Then
		\begin{equation*}
			(F_{\nabla'}-F_\nabla)(a)\lambda_x= 
			(\mathcal{D}_a -\nabla'_{\rho(a)} -\mathcal{D}_a + \nabla_{\rho(a)}) 
		\lambda = (\nabla_{\rho(a)}-\nabla'_{\rho(a)})\lambda = h(\rho(a))\lambda_x,
		\end{equation*}
		i.e. $F_{\nabla'}-F_\nabla=h\circ \rho$. In degree $0$, take $v,v'\in T_xM$, and let $V, V'\in \mathfrak{X}(M)$ be such that $V_x=v$ and $V'_x=v'$. Then
		\begin{align*}
			(d^{\nabla}\theta - d^{\nabla'}\theta)(v,v')&= \nabla_v\theta(V') - \nabla_{v'}\theta(V)- \theta([V,V']_x) -\nabla'_v\theta(V') + \nabla'_{v'}\theta(V)+ \theta([V,V']_x) \\
			&= (\nabla_v - \nabla'_v)\theta(v') - (\nabla_{v'} -\nabla'_{v'})\theta(v)
	=h(v)\theta(v')- h(v') \theta(v)\\
	&= \big(\theta^{\dag}(h(v)) - h^{\dag}(\theta(v))\big) (v'),
		\end{align*}
		i.e. $d^\nabla \theta - d^{\nabla'}\theta = \theta^\dag \circ h - h^\dag \circ \theta$. For degree $+1$, just notice that the last triangle is the twisted transpose of the first one.
		\end{proof}
		
	Not only the Morita curvature is independent of the connection (up to homotopies of the $0$-th components), it is actually Morita invariant in an appropriate sense. Namely, take a VB Morita map of LBGs $(F,f)\colon (L'\rightrightarrows L'_N; H\rightrightarrows N) \to (L\rightrightarrows L_M; G\rightrightarrows M)$, and let $\theta\in \Omega^1(M,L_M)$ be such that $\partial\theta=0$. Moreover, let $\nabla$ be a connection on $L_M$. Denote $\theta' = F^\ast \theta \in \Omega^1 (N, L'_N)$. Finally denote by $\nabla' = f^\ast \nabla$ the pull-back connection. For every $x \in N$ we then have the diagram
	\begin{equation}\label{diag:mor_curv_mor_inv}
			{\scriptsize
			\begin{tikzcd}
				0 \arrow[rr] & & A_{H,x} \arrow[rr] \arrow[dd] \arrow[dr, "df"] & & T_xN \arrow[rr] \arrow[dd] \arrow[dr, "df"] & & L'_{N,x} \arrow[rr] \arrow[dd] \arrow[dr, "F"] & & 0 \\
				& 0 \arrow[rr, crossing over] & & A_{G,f(x)} \arrow[rr, crossing over] & & T_{f(x)}M \arrow[rr, crossing over] & & L_{M,f(x)} \arrow[rr] & & 0 \\
				0 \arrow[rr] & & \mathbbm{R} \arrow[rr] & & T_x^{\dag}N \arrow[rr] & & A_{H,x}^{\dag} \arrow[rr] & & 0 \\
				& 0 \arrow[rr] & & \mathbbm{R} \arrow[from=uu, crossing over] \arrow[rr] \arrow[ul,equal] & & T_{f(x)}^{\dag}M \arrow[from=uu, crossing over] \arrow[rr] \arrow[ul, "df^{\dag}"] & & A^{\dag}_{G,f(x)} \arrow[from=uu, crossing over] \arrow[rr] \arrow[ul, "df^{\dag}"] & & 0 
			\end{tikzcd}},
		\end{equation}
where the front vertical arrows are the Morita curvature of $\theta$ determined by the connection $\nabla$ at the point $f(x)$, and the back vertical arrows are the Morita curvature of $\theta'$ determined by $\nabla'$ at the point $x$. We know already from Proposition \ref{prop:mor_ker_mor_inv} that the diagonal arrows on the top and the bottom of the diagram form quasi-isomorphisms. Moreover, a straightforward computation, that we leave to the reader, shows that the diagram \eqref{diag:mor_curv_mor_inv} commutes. We have thus proved the following

	\begin{prop}\label{prop:mor_curv_mor_inv}
		Let $(F,f), \theta', \theta, \nabla', \nabla$ be as above.
		Then the Morita curvatures of $\theta', \theta$ (with respect to $\nabla', \nabla$) at the points $x, f(x)$ are related by quasi-isomorphisms as in diagram \eqref{diag:mor_curv_mor_inv}.
	\end{prop}

\subsection{Definition and Examples}

	We are now ready to take care of the aspect (A5) of the definition of a contact structure in the Introduction, and give a definition of $0$-shifted contact structure. Let $(L \rightrightarrows L_M; G \rightrightarrows M)$ be an LBG.
	
	\begin{definition}\label{def:0-shift_cont}
		A \emph{$0$-shifted contact structure} on $L$ is an $L_M$-valued $1$-form $\theta\in \Omega^1(M,L_M)$ such that $\partial \theta=0$ and the Morita curvature at $x$ is a quasi-isomorphism for all points $x\in M$.
	\end{definition}
	
	The notion of $0$-shifted contact structure is Morita invariant in the sense of the following
	
	\begin{theo}\label{theo:Mor_inv_0-shift}
Let $(F,f)\colon (L'\rightrightarrows L'_N; H\rightrightarrows N) \to (L\rightrightarrows L_M; G\rightrightarrows M)$ be a VB Morita map of LBGs. Then the assignment $\theta \mapsto F^\ast \theta$ establishes a bijection between $0$-shifted contact structures on $L$ and $0$-shifted contact structures on $L'$.
\end{theo}

\begin{proof}
According to \cite[Theorem 8.8]{Drummond:DifferentialformsinVBG} the assignment $\theta \mapsto F^\ast \theta$ establishes a bijection between $\partial$-closed $L_M$-valued $1$-forms on $M$ and $\partial$-closed $L'_N$-valued $1$-forms on $N$. It remains to prove that $\theta$ is a $0$-shifted contact structure if and only if so is $F^\ast \theta$.  But this immediately follows from Proposition \ref{prop:mor_curv_mor_inv}, the last part of Remark \ref{rem:MC_RUTH} (see diagram \eqref{eq:mor_curv_orbit}), and the essential surjectivity of $f$.
\end{proof}

\begin{rem}
	For the last part of the previous proof, one can avoid using diagram \eqref{eq:mor_curv_orbit} and RUTHs, by working only with Morita maps which are surjective and submersive on objects. This is always possible as the latter (like generic Morita maps) do also generate Morita equivalence (see, e.g., \cite[Theorem 1.73]{Li2014}). 
\end{rem}

\begin{rem}
From Theorem \ref{theo:Mor_inv_0-shift} a $0$-shifted contact structure on an LBG $L \rightrightarrows L_M$ can actually be regarded as a structure on the LB stack $[L_M / L]$.
\end{rem}
	
	\begin{rem}
	If $\theta_x\neq 0$ for all $x\in M$, then, as discussed in Remark \ref{rem:Mor_ker_theta_neq_0}, the Morita kernel can be replaced by~\eqref{eq:ker_RUTH}, with its RUTH structure, up to quasi-isomorphisms. Accordingly the definition of $0$-shifted contact structure simplifies in this case. Namely, the standard curvature can also be seen as a (necessarily strict) RUTH morphism:
		\begin{equation}\label{eq:Mor_curv_theta_neq_0}
		\begin{tikzcd}
			0 \arrow[r] & A \arrow[r, "\rho"] \arrow[d] & K_{\theta} \arrow[r] \arrow[d, "R_{\theta}"'] & 0 \arrow[r] \arrow[d] & 0 \\
			0 \arrow[r] & 0 \arrow[r] & K_{\theta}^{\dag} \arrow[r, "\rho^{\dag}"'] & A^{\dag} \arrow[r] & 0
		\end{tikzcd}.
	\end{equation}

	To show this, denote by $\{R^T_k\}_{k \geq 0}$ the structure operators of the adjoint RUTH and by $\{R^{T^\dag}_k\}_{k \geq 0}$ the structure operators of the $L$-twisted dual RUTH. By degree reasons, it is enough to prove that, for any $g \colon x\to y$ in $G$
	\[
	R_\theta \circ R^T_1(g) = R^{T^\dag}_1 (g) \circ R_\theta \colon K_{\theta, x} \to K^\dag_{\theta, y}. 
	\]
	This boils down to the following identity
	\begin{equation}\label{eq:R_RUTH}
	d^\nabla \theta \big(v, dt (h_{g^{-1}}(v'))\big) = g^{-1}.d^\nabla \theta \big(dt(h_{g}(v)), v'\big)
	\end{equation}
	for all $v \in K_{\theta, x}$ and all $v' \in K_{\theta, y}$, where $h$ is the Ehresmann connection defining the adjoint RUTH. To prove \eqref{eq:R_RUTH}, let $\omega \rightleftharpoons (\theta, 0) \in \OA^2 (L_M)$, let $\delta \in D_x L_M, \delta' \in D_y L_M$ be such that $\sigma (\delta) = v, \sigma (\delta') = v'$, and let $h^D \colon s^\ast DL_M \to DL$ be as in Appendix \ref{app:RUTH} \eqref{eq:SES_Appendix}. Now, using Equation \eqref{eq:omegaandcomponents}, and $\partial \theta = 0$ (hence $\partial \omega = 0$) we find:
	\begin{align*}
		d^\nabla \theta \big(v, dt (h_{g^{-1}}(u))\big) &= \omega \big(\delta, Dt(h^D_{g^{-1}}(\delta'))\big)=\omega\big(Dt(\tilde \delta),  Dt(h^D_{g^{-1}}(\delta'))\big) \\
		&=t\big((t^\ast \omega)(\tilde \delta,h^D_{g^{-1}}(\delta'))\big) =t\big((s^\ast \omega)(\tilde \delta,h^D_{g^{-1}}(\delta'))\big) \\
		&=g^{-1}.\omega\big(Ds(\tilde \delta),Ds(h^D_{g^{-1}}(\delta'))\big) = g^{-1}.\omega\big(Dt(h^D_g(\delta)),\delta'\big) \\
		&=g^{-1}. d^\nabla \theta \big(dt(h_{g}(v)), v'\big)
	\end{align*}
	where $\tilde\delta= Di(h^D_g(\delta)) \in D_{g^{-1}}L$ and we used that $u,v\in K_{\theta}$.
	
	The RUTH morphism $R_\theta$ is related to the Morita curvature by a quasi-isomorphism:
	\[
	{\scriptsize
	\begin{tikzcd}
		0 \arrow[rr] & & A_x \arrow[rr] \arrow[dd] \arrow[dr, shift left=0.3ex, dash] \arrow[dr, shift right=0.3ex, dash] & & K_{\theta,x} \arrow[rr] \arrow[dd, pos=0.2, swap, "R_{\theta}"] \arrow[dr, "\mathrm{in}"] & & 0 \arrow[rr] \arrow[dd] \arrow[dr] & & 0 \\
		& 0 \arrow[rr, crossing over] & & A_x \arrow[rr, crossing over] & & T_xM \arrow[rr, crossing over] & & L_{M,x} \arrow[rr] & & 0 \\
		0 \arrow[rr] & & 0 \arrow[rr] & & K_{\theta,x}^{\dagger} \arrow[rr] & & A_x^{\dagger} \arrow[rr] & & 0 \\
		& 0 \arrow[rr] & & \mathbbm{R} \arrow[from=uu, crossing over, pos=0.2, "F_\nabla"] \arrow[rr] \arrow[ul] & & T_x^{\dagger}M \arrow[from=uu, crossing over, pos=.2, "d^\nabla \theta"] \arrow[rr] \arrow[ul, "\mathrm{in}^\dag"] & & A^{\dagger}_x \arrow[from=uu, crossing over, pos=0.2, "F_{\nabla}^\dag"] \arrow[rr] \arrow[ul, shift left=0.3ex, dash] \arrow[ul, shift right=0.3ex, dash] & & 0 
	\end{tikzcd}}
	\]
	showing that $\theta$ is a $0$-shifted contact structure if and only if the vertical arrows in \eqref{eq:Mor_curv_theta_neq_0} form a quasi-isomorphism, i.e.~$\rho$ is injective, hence $G$ is a \emph{foliation groupoid}, and $\im \rho = \ker R_\theta$.
	Notice however, that there are $0$-shifted contact structures $\theta$ for which $\theta_x = 0$ for some $x$ (but not all $x$, see Example \ref{exmpl:0-shifted_nonreg} below).
	\end{rem}
	
	The last result of this section is a bijection between $0$-shifted contact structures and $0$-shifted symplectic Atiyah forms which, in our opinion, represents a strong motivation for Definition \ref{def:0-shift_cont}.
	
	\begin{theo}\label{theor:0-contact=0-Atiyah}
		Let $(L\rightrightarrows L_M; G\rightrightarrows M)$ be an LBG. The assignment $\theta\mapsto \omega \rightleftharpoons (\theta,0)$ establishes a bijection between $0$-shifted contact structures and $0$-shifted symplectic Atiyah forms on $L$.
	\end{theo}
	
	\begin{proof}
		Choose once for all a connection $\nabla$ on $L_M$. We have to prove that the Morita curvature of $\theta$ is a quasi-isomorphism at the point $x\in M$ if and only if the cochain map \eqref{eq:0-Atiyah} is a quasi-isomorphism. Interestingly, the mapping cones of the Morita curvature and that of \eqref{eq:0-Atiyah} do actually agree. Namely, the mapping cone of \eqref{eq:0-Atiyah} is
		\begin{equation}
			\begin{tikzcd}\label{eq:map_con_omega}
				0 \arrow[r] & A_x \arrow[r, "-\mathcal{D}"] & D_xL_M \arrow[r, "\omega"] & J^1_xL_M \arrow[r, "\mathcal{D}^{\dag}"] & A_x^{\dag} \arrow[r] & 0
			\end{tikzcd}.
		\end{equation}
		Under the direct sum decomposition $DL_M \cong TM\oplus\mathbbm{R}$, $\delta\mapsto (\sigma(\delta), f_{\nabla}(\delta))$, the map $\mathcal{D}\colon A \to DL_M$ becomes
		\begin{equation*}
			(\rho, F_\nabla)\colon A \to TM\oplus \mathbbm{R}, \quad a\mapsto \big(\rho(a), F_\nabla(a)\big),
		\end{equation*}
		and $\omega \colon DL_M \to J^1 L_M $ becomes
		\begin{equation*}
			\begin{pmatrix}
				d^{\nabla}\theta & \theta^{\dag} \\
				-\theta &0
			\end{pmatrix}
			\colon TM\oplus \mathbbm{R} \to T^{\dag}M \oplus L_{M}.
		\end{equation*}
		Hence, the mapping cone \eqref{eq:map_con_omega} becomes
		\begin{equation}\label{eq:mc_mc}
		{\footnotesize
			\begin{tikzcd}[ampersand replacement=\&,column sep=huge,row sep=large]
				0 \arrow[r] \& A_x \arrow[r, "{(-\rho, -F_\nabla)}"] \& T_x M \oplus \mathbbm{R} \arrow[r, 
				"
				{\left(\begin{smallmatrix}
				d^\nabla \theta & \theta^\dag \\
				-\theta & 0
			\end{smallmatrix}\right)}
			"
				] \& T_x^{\dag}M\oplus L_{M,x} \arrow[r, "\rho^\dag + F_\nabla^\dag"] \& A_x^\dag \arrow[r] \& 0
			\end{tikzcd}
			},
		\end{equation}

		which is exactly the mapping cone of the Morita curvature of $\theta$ at the point $x\in M$. As the mapping cone is acyclic if and only if the cochain map is a quasi-isomorphism, this concludes the proof.	\end{proof} 
		
		\begin{rem}
	Let $(L\rightrightarrows L_M; G \rightrightarrows M)$ be an LBG equipped with a $0$-shifted contact structure $\theta$ and let $\omega$ be the corresponding $0$-shifted symplectic Atiyah form. It follows from the exactness of the sequence \eqref{eq:map_con_omega} and from $\omega$ being skew-symmetric, that $\mathbf{d} := 2 \dim M - \dim G$ is odd. The integer $\mathbf d$ is sometimes referred to as the \emph{dimension of the differentiable stack $[M/G]$} \cite{Xu:stacks}. When $G$ is a foliation groupoid, then $\mathbf d$ agrees with the dimension of the leaf space. 
	\end{rem}

	\begin{rem}\label{rem:hom_0-shift_sympl_grp}
	Let $Q \rightrightarrows P$ be a Lie groupoid and let $h$ be a principal action of $\mathbbm R^\times$ on $Q$ by Lie groupoid isomorphisms. Let $G = Q/\mathbbm R^\times$, and $M = P/\mathbbm R^\times$. It is clear that $G \rightrightarrows M$ is a Lie groupoid. Moreover, the line bundles $L \to G, L_M \to M$ associated to the principal $\mathbbm R^\times$-actions on $Q, P$ fit in an obvious LBG $(L \rightrightarrows L_M; G \rightrightarrows M)$. Now let $\theta \in \Omega^1 (M, L_M)$ and $\Theta \in \Omega^1 (P)$ be as in Remark \ref{rem:hom_sympl}. The assignment $\theta \mapsto d\Theta$ establishes a bijection between $0$-shifted contact structures on $L$ and homogeneous $0$-shifted symplectic structures of degree $1$ on $Q \rightrightarrows P$. This can be easily seen, e.g., using Theorem \ref{theor:0-contact=0-Atiyah} and the relationship between degree $1$ homogeneous differential forms and Atiyah forms, for which we refer the reader to \cite{Vitagliano:holomorphic}.
	\end{rem}
	
	We conclude this section with some examples.
	
	\begin{example}\label{exmpl:0-shifted_nonreg}
	Unlike for $0$-shifted symplectic structures, a Lie groupoid supporting a $0$-shifted contact structure $\theta$ needs not be regular. The rank of the anchor drops exactly at points $x$ where $\theta_x = 0$. Here we discuss a $0$-shifted contact structure $\theta$ such that $\theta_x$ is generically non-zero, but $\theta_x = 0$ along a submanifold (of positive codimension). Begin with $M = \mathbbm R^2$ with standard coordinates $(x, y)$ and the vector field $X = y \tfrac{\partial}{\partial y}$. Denote by $A \to M$ the Lie algebroid defined as follows. Set $A = \mathbbm R_{M} = \mathbbm R^2 \times \mathbbm R$, the anchor $\rho \colon A \to TM$ maps $1_M$, the constant function equal to $1$, to the vector field $X$, and the Lie bracket of two functions $f, g \in \Gamma (A) = C^\infty (M)$ is
	\[
	[f, g] := f X(g) - g X(f).
	\]
	The Lie algebroid $A$ is integrable (as all rank $1$ Lie algebroids) and it integrates to the Lie groupoid $G \rightrightarrows M$ defined by the flow $\Phi^X \colon (x, y; \varepsilon) \mapsto (x, e^\varepsilon y)$ of $X$: $G = M \times \mathbbm R$ with coordinates $(x, y; \varepsilon)$. Then
	\[
	s(x, y; \varepsilon) = (x, y), \quad t(x, y; \varepsilon) = (x, e^\varepsilon y), \quad m\big((x', e^{\varepsilon'} y'; \varepsilon), (x', y'; \varepsilon')\big) = (x', y'; \varepsilon + \varepsilon'), 
	\]
	and, moreover
	\[
	u (x, y) = (x, y; 0), \quad i(x, y; \varepsilon) = (x, e^\varepsilon y; - \varepsilon).
	\]
	We let $G$ act on the trivial line bundle $\mathbbm R_M$ with coordinates $(x, y; r)$ via
	\[
	(x, y; \varepsilon). (x,y; r) := (x, e^\varepsilon y; e^{\varepsilon}r).
	\]
	The infinitesimal action $\mathcal D \colon A \to D \mathbbm R_M$ then maps $1_M$ to $\mathcal D_{1_M} = y \tfrac{\partial}{\partial y} - \mathbbm I$. In particular $\mathcal D$ is injective. The $\mathbbm R_M$-valued $1$-form
	\[
	\theta = ydx \otimes 1_M
	\]
	is a $0$-shifted contact structure on the LBG $G \ltimes \mathbbm R_M$. To see this, first notice that $\theta$ is invariant under the $G$-action on $\mathbbm R_M$. Equivalently, $\partial \theta = 0$. For the ``non-degeneracy of the Morita curvature'' we prefer to use Theorem \ref{theor:0-contact=0-Atiyah}. So let $\omega \rightleftharpoons (\theta, 0)$ be the $\partial$-closed Atiyah $2$-form on $\mathbbm R_M$ corresponding to $\theta$. It is easy to see that
	\[
	\omega = (dy \wedge dx - y dx \wedge \mathbbm I^\ast)  \otimes 1_M
	\]
	where we are denoting by $(dx, dy, \mathbbm I^\ast)$ the basis of $\Gamma ((D \mathbbm{R}_M)^\ast)$ dual to the basis $(\tfrac{\partial}{\partial x}, \tfrac{\partial}{\partial y}, \mathbbm I)$ of $\Gamma (D\mathbbm R_M)$. It follows that
	\[
	\ker \omega = \operatorname{Span} (\mathcal D_{1_M}) = \im \mathcal D. 
	\]
	We conclude that $\omega$ is a $0$-shifted symplectic Atiyah form, hence $\theta$ is a $0$-shifted contact structure. Clearly $\theta_{(x, y)} = 0$ when $y = 0$.
		\end{example}
		
		The next example is a significant generalization of Example \ref{exmpl:0-shifted_nonreg} and has been suggested to us by one of the referees. 
		
		\begin{example}
		Let $h$ be an action of $\mathbbm R^\times$ on a manifold $P$.
		Denote by $\mathcal E$ the infinitesimal generator of $h$. The Lie group $\mathbbm R^\times$ acts on the trivial line bundle $\mathbbm R_P$ as follows: $\varepsilon.(p, r) := (h_\varepsilon(p), \varepsilon r)$, and we will need to consider both the action groupoid $G := \mathbbm R^\times \ltimes P \rightrightarrows P$ and the action LBG $L := \mathbbm R^\times \ltimes \mathbbm R_P \rightrightarrows \mathbbm R_P$. We know from Remark \ref{rem:hom_sympl} that, when $h$ is a principal action, then degree $1$ homogeneous symplectic forms on $P$ correspond bijectively to contact forms on $P/\mathbbm R^\times$ with values in the line bundle $\mathbbm R_P / \mathbbm R^\times$.
		We want to show that, when the action $h$ is \emph{not} principal, yet degree $1$ homogeneous symplectic forms on $P$ correspond bijectively to $0$-shifted contact structures on $L$. To see this begin with a $1$-form $\theta \in \Omega^1 (P) = \Omega^1 (P, \mathbbm R_P)$. The condition $\partial \theta$ is equivalent to $\theta$ being homogeneous of degree $1$ (beware that here $\partial$ is the simplicial differential on forms with values in the nerve of $L$, not the simplicial differential on ordinary forms). In this case $\omega = d \theta \in \Omega^2 (P)$ is also homogeneous of degree $1$, and $\theta = i_{\mathcal E} \omega$. Fix the trivial connection $\nabla$ on $\mathbbm R_P$ and notice that, with this choice, the sequence \eqref{eq:mc_mc} boils down to
		\begin{equation}
		{\footnotesize
			\begin{tikzcd}[ampersand replacement=\&,column sep=huge,row sep=large]
				0 \arrow[r] \& \mathbbm R \arrow[r, "{(-\mathcal E, \operatorname{id}_{\mathbbm R})}"] \& T_x P \oplus \mathbbm{R} \arrow[r, 
				"
				{\left(\begin{smallmatrix}
				\omega & \theta \\
				-\theta & 0
			\end{smallmatrix}\right)}
			"
				] \& T_x^{\ast}P\oplus \mathbbm R \arrow[r, "i_{\mathcal E} - \operatorname{id}_{\mathbbm R}"] \& \mathbbm R \arrow[r] \& 0
			\end{tikzcd}
			},
		\end{equation}
	which is exact if and only $\omega$ is non-degenerate. This shows that the assignment $\theta \mapsto d\theta$ establishes a bijection between $0$-shifted contact structures on $L$ and degree $1$ homogeneous symplectic structures on $P$ whose inverse is given by $\omega \mapsto i_{\mathcal E} \omega$. 	
		\end{example}
		
		\begin{example}
		Let $L_B \to B$ be a line bundle. The unit groupoid $L_B \rightrightarrows L_B$ is an LBG over the unit groupoid $B \rightrightarrows B$. A $0$-shifted contact structure on $L_B \rightrightarrows L_B$ is the same as an ordinary $L_B$-valued contact $1$-form $\theta \in \Omega^1 (B, L_B)$.  Now, let $M \to B$ be a surjective submersion, and let $G = M \times_B M \rightrightarrows M$ be the corresponding submersion groupoid. Set $L_M = M \mathbin{\times_B} L_B$. Then $G$ acts on $L_M$ in the obvious way. The associated action groupoid $L = G \ltimes L_M$ is an LBG over $G$. Moreover, the projection $\pi \colon L \to L_B$ is actually a VB Morita map onto the unit LBG. It follows from Theorem \ref{theo:Mor_inv_0-shift} that the $0$-shifted contact structures on $L$ are exactly the pull-backs $\pi^\ast \theta$ of some $L_B$-valued contact $1$-form $\theta$ on $B$. In other words, those $\theta \in \Omega^1 (M, L_M)$ for which there is a well-defined \emph{contact reduction} under the projection $L_M \to L_B$.  
				\end{example}
	
	\begin{example}[Contact structures on orbifolds]
Let $G \rightrightarrows M$ be a proper and \'etale groupoid (in particular the isotropy groups $G_x$, $x \in M$, are finite). The orbit space $X := M/G$ is an orbifold and $G$ defines an orbifold atlas on $X$ as follows. Let $x \in M$, and let $U \subseteq M$ be an open neighborhood of $x$ such that the restricted groupoid $G_U = (s \times t)^{-1} (U \times U) \rightrightarrows U$ identifies with an action groupoid $G_x \ltimes U \rightrightarrows U$, where $G_x$ acts (linearly) on $U$ via a diffeomorphism $U \cong T_x M$. The projection $U \to U/G_U \subseteq X$, together with the $G_x$-action on $U$, can be seen as an orbifold chart, and $X$ is covered by such charts. If $U, V \to X$ are two such charts, and $x \in U \cap V$, then a \emph{chart compatibility} is provided by any open subset $W \subseteq (s \times t)^{-1}(U \times V)$ such that $s \colon W \to U$ and $t \colon W \to V$ are both embeddings around $x$. Finally, let $L \rightrightarrows L_M$ be an LBG over $G \rightrightarrows M$, so that the orbit space $L_X := L_M / L$ is a line bundle (in the category of orbifolds) over $X$. With this reminder, it should be clear that a $0$-shifted contact structure on $L$ is equivalent to a contact structure on $X$ (i.e.~a group invariant contact structure on each chart which is additionally preserved by chart compatibilities \cite[Definition 2.3.1]{Herr:contactorbifolds}) with normal line bundle given exactly by $L_X$. From Theorem \ref{theo:Mor_inv_0-shift} the same is true for any \emph{orbifold groupoid}, i.e.~a proper foliation groupoid, presenting the orbifold $X$ (remember that a groupoid is a foliation groupoid if and only if it is Morita equivalent to a proper, \'etale groupoid). We leave the obvious details to the reader. 
	\end{example}

\section{$+1$-Shifted Contact Structures}
\label{sec:contact}

In this section we introduce \emph{$+1$-shifted contact structures}. The discussion will parallel that of Section \ref{sec:0-shifted_cs}. Specifically, we will begin with $+1$-shifted symplectic Atiyah forms (Section \ref{sec:Atiyah_shifted}), and then pass to a more classical definition in terms of line bundle valued $1$-forms, paying again attention to all the aspects listed in the Introduction and their interaction with Morita equivalences. Given a \emph{multiplicative line bundle valued $1$-form} (aspects (A1)--(A2)), we define its \emph{Morita kernel} (aspect (A3) - Section \ref{sec:Morita_kernel}), and its \emph{Morita curvature} (aspect (A4) - Section \ref{sec:Morita_curvature}). Eventually, we will explain what does it mean for the Morita curvature to be ``non-degenerate'' in a Morita invariant way (aspect (A5)), and we will get our ultimate definition (Section \ref{sec:+1-def_exmpl}). We will conclude with two examples. Although the situation is more involved for $+1$-shifted contact structures, there are some features which are surprisingly simpler here than in the $0$-shifted case. Namely, the Morita kernel is a plain VBG and the Morita curvature is a plain VBG morphism in the $+1$-shifted case (no need to use RUTHs).

\subsection{$+1$-Shifted Symplectic Structures}\label{sec:def1shif}
First we recall the definition of $+1$-shifted symplectic structure \cite{Xu:momentum_maps,Getzler:shifted,Zhu:twisted}. We adopt the equivalent definition from \cite[Section 5.2]{DelHoyo:VBmorita} for the ``non-degeneracy'' condition. So, let $G \rightrightarrows M$ be a Lie groupoid with Lie algebroid $A$. A \emph{multiplicative form} on $G$ is a differential form $\omega\in \Omega^{\bullet}(G)$ such that $\partial \omega=0$. The sharp map of a multiplicative $2$-form $\omega\in \Omega^2(G)$, also denoted $\omega$, is a VBG morphism $\omega\colon (TG\rightrightarrows TM)\to (T^{\ast}G\rightrightarrows A^{\ast})$ between the tangent and the cotangent groupoids.
	\begin{definition}\label{def:+1-shift_sympl_grpd}
		A \emph{$+1$-shifted symplectic structure} on $G$ is a pair $(\omega, \Omega)$ consisting of a $2$-form $\omega\in \Omega^2(G)$ and a $3$-form $\Omega\in \Omega^3(M)$ such that $\partial \omega = 0$ (i.e.~$\omega$ is multiplicative), $d\omega=\partial \Omega$, $d\Omega=0$, and, moreover, $\omega\colon TG\to T^{\ast}G$ is a VB Morita map. The \emph{gauge transformation} of a $+1$-shifted symplectic structure $(\omega, \Omega)$ on $G$ by a $2$-form $\alpha \in \Omega^2 (M)$ is the $+1$-shifted symplectic structure $(\omega + \partial \alpha, \Omega + d \alpha)$. A \emph{$+1$-shifted symplectic groupoid} is a Lie groupoid equipped with a $+1$-shifted symplectic structure.
	\end{definition}

	From Theorem \ref{theo:caratterizzazioneVBmorita}, $\omega\colon TG\to T^{\ast}G$ is a VB Morita map if and only if the cochain map induced on fibers
	\begin{equation*}
		\begin{tikzcd}
			0 \arrow[r] & A_x \arrow[r, "\rho"] \arrow[d, "\omega"'] & T_xM \arrow[r] \arrow[d, "\omega"] &0 \\
			0 \arrow[r] & T_x^{\ast}M \arrow[r, "\rho^{\ast}"'] & A_x^{\ast} \arrow[r] &0
		\end{tikzcd}
	\end{equation*}
	is a quasi-isomorphism for all $x\in M$. Notice that, as the map $\omega\colon T_xM\to A^{\ast}_x$ is just the opposite of the transpose of the map $\omega\colon A_x\to T_x^{\ast}M$, then the latter is an isomorphism in cohomology if and only if so is the former.
	
	\begin{rem} \label{rem:cohom_class_1-shft_sympl}
	That the gauge tranformation $(\omega'=\omega + \partial \alpha,\Omega'= \Omega + d \alpha)$ of a $+1$-shifted symplectic structure $(\omega, \Omega)$ is a $+1$-shifted symplectic structure as well has been proved in \cite[Proposition 4.6]{Xu:momentum_maps}. It also follows from Theorem \ref{theo:caratterizzazioneVBmorita}. Indeed $\partial \omega'= 0$, $d\omega'=\partial\Omega'$ and $d\Omega'=0$. Moreover, the VBG morphism $\omega' \colon TG \to T^\ast G$ induces the following map on the fibers over $x \in M$:
	\[
	\begin{tikzcd}
			0 \arrow[r] & A_x \arrow[r, "\rho"] \arrow[d, "\omega' = \omega + \partial \alpha"'] & T_xM \arrow[r] \arrow[d, "\omega' = \omega + \partial \alpha"] &0 \\
			0 \arrow[r] & T_x^{\ast}M \arrow[r, "\rho^{\ast}"'] & A_x^{\ast} \arrow[r] &0
		\end{tikzcd},
	\]
	but $\partial \alpha = s^\ast \alpha - t^\ast \alpha \colon A_x \to T_x M$ vanishes on $\ker \rho$, showing that $\omega$ and $\omega'$ do actually induce the same map in the cohomology of the fibers in degree $-1$ (hence in degree $0$ as well).
	\end{rem}
	
	The notion of $+1$-shifted symplectic groupoid is Morita invariant in an appropriate sense (see \cite{Xu:momentum_maps, Cueca:shiftedstructures}, see also \cite{Mayrand}). First of all, let $f \colon H \to G$ be a Morita map between Lie groupoids and let $(\omega, \Omega)$ be a $+1$-shifted symplectic structure on $G$. Then $(f^\ast \omega, f^\ast \Omega)$ is a $+1$-shifted symplectic structure on $H$ (see, e.g., \cite[Lemma 2.28]{Cueca:shiftedstructures}). Now, two $+1$-shifted symplectic groupoids $(G_1, (\omega_1, \Omega_1)), (G_2, (\omega_2, \Omega_2))$ are \emph{symplectic Morita equivalent} if there exist a Lie groupoid $H$, and Morita maps 
	\[
	\begin{tikzcd}
		& H \arrow[dl, "f_1"'] \arrow[dr, "f_2"] \\
		G_1 & & G_2
	\end{tikzcd}
	\]
such that the $+1$-shifted symplectic structures $(f_1^\ast \omega_1, f_1^\ast \Omega_1), (f_2^\ast \omega_2, f_2^\ast \Omega_2)$ agree up to a gauge transformation. Symplectic Morita equivalence is an equivalence relation. Moreover, given a $+1$-shifted symplectic groupoid $(G_1, (\omega_1, \Omega_1))$ and a Morita equivalent Lie groupoid $G_2$, then there exists a $+1$-shifted symplectic structure $(\omega_2, \Omega_2)$ on $G_2$ such that $(G_1, (\omega_1, \Omega_1))$ and $(G_2, (\omega_2, \Omega_2))$ are symplectic Morita equivalent. The $+1$-shifted symplectic structure $(\omega_2, \Omega_2)$ is unique up to gauge transformations. These latter facts motivate the following

	\begin{definition}\label{def:+1-shift_sympl_stack}
		A \emph{$+1$-shifted symplectic structure} on the differentiable stack $[M/G]$ is a symplectic Morita equivalence class of $+1$-shifted symplectic groupoids $(G \rightrightarrows M, (\omega, \Omega))$ representing $[M/G]$. 
	\end{definition}
	
	\subsection{$+1$-Shifted Symplectic Atiyah Forms}\label{sec:Atiyah_shifted}

	First we translate Definitions \ref{def:+1-shift_sympl_grpd} and \ref{def:+1-shift_sympl_stack} to the realm of Contact Geometry using Atiyah forms. Take again an LBG $(L\rightrightarrows L_M; G \rightrightarrows M)$, and let $A$ be the Lie algebroid of $G$. Similarly as for ordinary differential forms, we say that an Atiyah form $\omega\in \OA^{\bullet}(L)$ is \emph{multiplicative} if $\partial \omega=0$. The sharp map of a multiplicative Atiyah $2$-form $\omega \in \OA^{\bullet}(L)$, also denoted $\omega$, is a VBG morphism $\omega\colon (DL\rightrightarrows DL_M) \to (J^1L\rightrightarrows A^{\dag})$ (this can be easily proved either by hands, or combining the analogous result for ordinary $2$-forms with the homogenization techniques of \cite[Section 2]{Vitagliano:holomorphic}).
	
	\begin{definition}
		A \emph{$+1$-shifted symplectic Atiyah form} on the LBG $L$ is a pair $(\omega, \Omega)$ consisting of an Atiyah $2$-form $\omega\in \OA^{2}(L)$ and an Atiyah $3$-form $\Omega\in \OA^3(L_M)$ such that $\partial \omega = 0$ (i.e.~$\omega$ is multiplicative), $\dA\omega=\partial \Omega$, $\dA\Omega=0$, and, moreover, $\omega\colon DL\to J^1L$ is a VB Morita map. The \emph{gauge transformation} of a $+1$-shifted symplectic Atiyah form $(\omega, \Omega)$ on $L$ by an Atiyah $2$-form $\alpha \in \OA^2 (L_M)$ is the $+1$-shifted symplectic Atiyah form $(\omega + \partial \alpha, \Omega + \dA \alpha)$. A \emph{$+1$-shifted symplectic Atiyah LBG} is an LBG equipped with a $+1$-shifted symplectic Atiyah form.
	\end{definition}

	From Theorem \ref{theo:caratterizzazioneVBmorita}, $\omega\colon DL\to J^1L$ is a VB Morita map if and only if the cochain map induced on fibers
	\begin{equation}\label{eq:+1_map_fiber_omega_Atiyah}
		\begin{tikzcd}
			0 \arrow[r] & A_x \arrow[r, "\mathcal{D}"] \arrow[d, "\omega"'] & D_xL_M \arrow[d, "\omega"] \arrow[r] &0 \\
			0 \arrow[r] & J^1_x L_M \arrow[r, "\mathcal{D}^{\dag}"'] & A_x^{\dag} \arrow[r] &0
		\end{tikzcd}
	\end{equation}
	is a quasi-isomorphism for all $x\in M$. The map $\omega\colon D_xL_M\to A_x^{\dag}$ is the opposite of the twisted transpose of the map $\omega\colon A_x\to J^1_x L_M$, hence the latter is an isomorphism in cohomology if and only if so is the former. 
	
	\begin{rem}
	That the gauge transformation $(\omega + \partial \alpha, \Omega + d \alpha)$ of a $+1$-shifted symplectic Atiyah form $(\omega, \Omega)$ is a $+1$-shifted symplectic Atiyah form  as well can be proved exactly as in the symplectic case, see Remark \ref{rem:cohom_class_1-shft_sympl}. Notice however that in the symplectic Atiyah case, the situation is actually simpler as every $+1$-shifted symplectic Atiyah form $(\omega, \Omega)$ can be gauge transformed into one of the type $(\omega', 0)$, hence $\dA \omega' = 0$, in other words $\omega' \rightleftharpoons (\theta, 0)$. Indeed, as the der-complex is acyclic, it follows from $\dA \Omega = 0$ that $\Omega = - \dA \alpha$ for some $\alpha \in \OA^2 (L_M)$. Now, gauge transform $(\omega, \Omega)$ by $\alpha$.
	\end{rem}
	
	The notion of $+1$-shifted symplectic Atiyah form is Morita invariant in an appropriate sense. Indeed, the following facts can be proved by a translation from the symplectic case. They are also a consequence of the results in Section \ref{sec:+1-def_exmpl} (in particular Theorem \ref{theor:+1_theta_omega_bij}). First of all, let $F \colon L' \to L $ be a VB Morita map between LBGs and let $(\omega, \Omega)$ be a $+1$-shifted symplectic Atiyah form on $L$. Then $(F^\ast \omega, F^\ast \Omega)$ is a $+1$-shifted symplectic Atiyah form on $H$. Now, two $+1$-shifted symplectic Atiyah LBGs $(L_1, (\omega_1, \Omega_1)), (L_2, (\omega_2, \Omega_2))$ are \emph{symplectic Morita equivalent} if there exist an LBG $L'$, and VB Morita maps 
	\[
	\begin{tikzcd}
		& L' \arrow[dl, "F_1"'] \arrow[dr, "F_2"] \\
		L_1 & & L_2
	\end{tikzcd}
	\]
such that the $+1$-shifted symplectic structures $(F_1^\ast \omega_1, F_1^\ast \Omega_1), (F_2^\ast \omega_2, F_2^\ast \Omega_2)$ agree up to a gauge transformation. Symplectic Morita equivalence is an equivalence relation. Moreover, given a $+1$-shifted symplectic Atiyah LBG $(L_1, (\omega_1, \Omega_1))$ and a VB Morita equivalent LBG $L_2$, then there exists a $+1$-shifted symplectic Atiyah form $(\omega_2, \Omega_2)$ on $L_2$ such that $(L_1, (\omega_1, \Omega_1))$ and $(L_2, (\omega_2, \Omega_2))$ are symplectic Morita equivalent. The $+1$-shifted symplectic Atiyah form $(\omega_2, \Omega_2)$ is unique up to gauge transformations. These latter facts motivate the following 
		
	\begin{definition}
		A \emph{$+1$-shifted symplectic Atiyah form} on the LB stack $[L_M/L]$ is a symplectic Morita equivalence class of $+1$-shifted symplectic Atiyah LBGs $(L \rightrightarrows L_M, (\omega, \Omega))$ representing $[L_M/L]$. 
	\end{definition}

\subsection{Morita Kernel}\label{sec:Morita_kernel}
	Our next aim is to provide a definition of $+1$-shifted contact structure. We will follow the same strategy as in Section \ref{sec:0-shifted_cs} and go through all aspects (A1)--(A5) of the definition of a contact form in the Introduction, paying attention to Morita invariance. We say that a form $\theta \in \Omega^{\bullet}(G,L)$ is \emph{multiplicative} if $\partial\theta=0$. This multiplicativity condition agrees with the multiplicativity condition for VB valued forms in \cite{Craini:multiplicativeforms} under the isomorphism $L\cong t^{\ast}L_M$.
Notice that, from Remark \ref{rem:partial_components}, an Atiyah form is multiplicative if and only if so are its components. 	
	
	Now let $\theta \in \Omega^1(G,L)$ be a multiplicative $1$-form (aspect (A2)). The kernel of $\theta$ is not a well-defined subgroupoid of $TG\rightrightarrows TM$ in general, because its dimension jumps at points $g\in G$ where $\theta_g=0$. However, interestingly, we can always make sense of $\ker \theta$ as a differentiable stack (aspect (A3)). Namely, similarly as we did for $0$-shifted contact structures, inspired again by the homotopy kernel in Homological Algebra, we construct a VBG playing the role of $\ker\theta$ up to Morita equivalences. The latter VBG is given by $(TG\oplus L \rightrightarrows TM\oplus L_M; G\rightrightarrows M)$ with structure maps defined as follows: source and target are given by
	\begin{align*}
		s(v,\lambda)=\big(s(v),s(\lambda)\big), \quad t(v,\lambda)=\big(t(v), t\big(\lambda +\theta(v)\big)\big), \quad v\in TG,\quad \lambda \in L,
	\end{align*}
	the unit map is given by
	\begin{equation*}
		u(v,\lambda)=\big(u(v), u(\lambda)\big), \quad v\in TM,\quad \lambda \in L_M,
	\end{equation*}
	the inversion is given by
	\begin{equation*}
		(v,\lambda)^{-1}= \left(v^{-1}, s_{g^{-1}}^{-1}\left(t\big(\lambda+ \theta(v)\big)\right)\right), \quad v\in T_gG, \quad \lambda\in L_g.
	\end{equation*}
	Finally the multiplication is given by
	\begin{equation*}
		(v,\lambda)(v',\lambda')=\left(vv', s_{gg'}^{-1}\big(s(\lambda')\big)\right), \quad (v,v')\in T_{(g,g')}G^{(2)}, \quad (\lambda, \lambda')\in L^{(2)}_{(g,g')}.
	\end{equation*}
	\begin{prop/def}
		For any multiplicative $1$-form $\theta\in \Omega^1(G,L)$, $TG\oplus L\rightrightarrows TM \oplus L_M$ with the structure maps defined above is a VBG over $G\rightrightarrows M$ that we call the \emph{Morita kernel} of $\theta$ and we also denote by $\mk_\theta$.
	\end{prop/def}
	
	The proof is a straightforward computation to check the VBG axioms. The core of the Morita kernel is $A$, the Lie algebroid of $G$, and the core-anchor is $(\rho, \ell_\theta)\colon A \to TM\oplus L_M$, where $\rho$ is the anchor and $\ell_\theta \colon A \to L_M$ is the restriction of $\theta$ to $A$. 
	
	\begin{rem}
	Remember from \cite{Gracia:VBgroupoids} that there is an equivalence of categories between VBGs and $2$-term RUTHs. One should compare the Morita kernel of a multiplicative $1$-form being a VBG with the Morita kernel of a $\partial$-closed $1$-form $\theta \in \Omega^1(M, L_M)$ being a $3$-term RUTH (Remark \ref{rem:MK_RUTH}). As already announced, the situation is surprisingly simpler in the $+1$-shifted case.
	\end{rem}
	
	The Morita kernel does only depend on the $\partial$-cohomology class of $\theta$ in the complex \eqref{eq:complex_valuedform} up to VB Morita equivalences in the sense of the following 
	\begin{prop}
		\label{prop:invarianceoftheta}
		Let $\theta$, $\theta' \in \Omega^1(G,L)$ be multiplicative and $\partial$-cohomologous $1$-form: $\theta - \theta' =\partial \alpha$ for some $\alpha \in \Omega^1(M,L_M)$. Then the map $\mathsf A \colon \mk_\theta \to \mk_{\theta'}$ defined by setting $\mathsf A(v,\lambda)= (v, \lambda +s^{\ast}\alpha(v))$ is a VB Morita map.
	\end{prop}
	\begin{proof}
		A straightforward computation shows that  $\mathsf A$ is a VBG morphism. The induced cochain map on the fibers over $x \in M$ is
		\begin{equation}\label{eq:Phi_alpha}
		\begin{tikzcd}
			0 \arrow[r] & A_x \arrow[r, "(\rho{,} \ell_\theta)"] \arrow[d, equal] & T_xM\oplus L_{M,x} \arrow[r] \arrow[d, "\mathsf A"] & 0 \\
			0 \arrow[r] & A_x \arrow[r, "(\rho{,} \ell_{\theta'})"'] & T_xM\oplus L_{M,x} \arrow[r]  & 0
		\end{tikzcd},
	\end{equation} 
	where $\mathsf A (v, \lambda) = (v, \lambda + \alpha (v))$, $v \in T_x M$, $\lambda \in L_{M, x}$. But, from $\theta - \theta' = \partial \alpha$, we get $\ell_{\theta'} = \ell_\theta + \alpha \circ \rho$, whence \eqref{eq:Phi_alpha} is a quasi-isomorphism, and the claim follows from Theorem \ref{theo:caratterizzazioneVBmorita}.
	\end{proof}
	
	The Morita kernel is Morita invariant in the sense of the following
	\begin{prop}\label{prop:+1_Mor_ker_Mor_inv}
		Let $(F,f)\colon (L'\rightrightarrows L'_N; H\rightrightarrows N)\to (L\rightrightarrows L_M; G\rightrightarrows M)$ be a VB Morita map between LBGs, and let $\theta\in \Omega^1(G,L)$ be a multiplicative $L$-valued $1$-form. Then the map $\mathsf F \colon \mk_{F^{\ast}\theta} \to \mk_{\theta}$ defined by setting $\mathsf F(v,\lambda)= (df(v), F(\lambda))$ is a VB Morita map.
	\end{prop}
	\begin{proof}
		Denote by $A_G, A_H$ the Lie algebroids of $G,H$. From Theorem \ref{theo:caratterizzazioneVBmorita} and \cite[Corollary 3.8]{DelHoyo:VBmorita}, for all $x \in M$, the vertical arrows in
		\begin{equation}\label{eq:df_quasi_iso}
		\begin{tikzcd}
			0 \arrow[r] & A_{H,x} \arrow[r, "\rho_H"] \arrow[d, "df"']  & T_x N\arrow[r] \arrow[d, "df"] & 0 \\
			0 \arrow[r] & A_{G,f(x)} \arrow[r, "\rho_G"'] & T_x M\arrow[r]  & 0
		\end{tikzcd}
		\end{equation}
		form a quasi-isomorphism. As for $\mathsf F$, a straightforward computation shows that $\mathsf F$ is a VBG morphism. The induced cochain map on the fibers over $x \in N$ and $f(x) \in M$ is
		\begin{equation}\label{eq:F_quasi_iso}
		\begin{tikzcd}
			0 \arrow[r] & A_{H,x} \arrow[r, "(\rho_H{,} \ell_{F^\ast \theta})"] \arrow[d, "df"']  & T_x N\oplus L'_{N,x} \arrow[r] \arrow[d, "\mathsf F"] & 0 \\
			0 \arrow[r] & A_{G,f(x)} \arrow[r, "(\rho_G{,} \ell_{\theta})"'] & T_x M\oplus L_{M,x} \arrow[r]  & 0
		\end{tikzcd}
		\end{equation}
		But $\ell_{F^\ast \theta, x} = F_x^{-1}\circ \ell_{\theta, f(x)} \circ df$. It is easy to check from the latter formula, from the definition of $\mathsf F$, and from \eqref{eq:df_quasi_iso} being a quasi-isomorphism, that \eqref{eq:F_quasi_iso} is a quasi-isomorphism as well. The claim now follows from Theorem \ref{theo:caratterizzazioneVBmorita}.
	\end{proof}
	
	\begin{rem}\label{rem:+1_Mor_ker_regular}
	In the case when $\theta_g\neq 0$ for all $g\in G$, then $K_\theta := \ker \theta \rightrightarrows TM$ is a VB subgroupoid of $TG \rightrightarrows TM$ with core $C = A \cap \ker \theta$ (see, e.g., \cite[Lemma 3.6]{Craini:multiplicativeforms}). Actually, the inclusion $\mathrm{in}\colon K_{\theta} \to \mk_{\theta}$ is a VB Morita map. Indeed, it is a VBG morphism over the identity, which is a Morita map, and the induced cochain map on fibers over $x\in M$ is
		\begin{equation}
			\label{eq:inclusionMorita}
			\begin{tikzcd}
				0\arrow[r]& C_x\arrow[r, "\rho"] \arrow[d] & T_xM \arrow[d] \arrow[r] &0 \\
				0 \arrow[r] &A_x \arrow[r,"(\rho {,} \ell_\theta)"'] & T_xM \oplus L_{M,x} \arrow[r] &0
			\end{tikzcd},
		\end{equation}
		where the vertical arrows are the inclusions. Clearly \eqref{eq:inclusionMorita} is a quasi-isomorphism for all $x$. This motivates replacing $K_\theta$ by $\mk_\theta$ in the general case.
	\end{rem}
	
\subsection{Morita Curvature}\label{sec:Morita_curvature}
	Let $\theta \in \Omega^1 (G, L)$ be a multiplicative $1$-form. The next step is defining the ``curvature'' of $\theta$ in a Morita invariant way (aspect (A4)). First of all, applying the twisted dual VBG construction (Example \ref{ex:adjoint}) to $\mk_{\theta}$ we obtain the VBG $\mk_{\theta}^{\dag}:= T^{\dag}G\oplus \mathbbm{R}_G \rightrightarrows A^{\dag}$, where $\mathbbm{R}_G=G\times \mathbbm{R} \to G$ is the trivial line bundle over $G$. The structure maps of $\mk_\theta^\dag$ are explicitly given by:
	\begin{itemize}
		\item The source and the target of $(\psi,r)\in T_g^{\dag}G\oplus\mathbbm{R}$, $g \in G$, are given by
		\begin{align*}
			\big\langle s(\psi,r), a \big\rangle &= - s\big\langle \psi ,0_g \cdot a^{-1} \big\rangle - r \theta (a), & a\in A_{s(g)},\\
			\big\langle t(\psi,r), a' \big\rangle &= t \big\langle \psi , a' \cdot 0_g \big\rangle, & a'\in A_{t(g)}.
		\end{align*}
		\item The unit over $\psi\in A_x^{\dag}$, $x \in M$, is given by
		\begin{equation*}
			u(\psi)= (\psi \circ \pr_A, 0),
		\end{equation*}
		where $\pr_A \colon T_{x} G \to A_x$ is the projection with kernel $T_x M$.
		\item The multiplication between two composable arrows $(\psi,r)\in T_g^{\dag}G\oplus \mathbbm{R}$ and $(\psi',r')\in T_{g'}^{\dag}G \oplus \mathbbm{R}$, $(g, g') \in G^{(2)}$, is $(\Psi, r+r') \in T_{gg'}^{\dag}G \oplus\mathbbm{R}$ where $\Psi$ is given by
		\begin{equation*}
			\langle \Psi, vv'\rangle = s_{gg'}^{-1}\Big(g'^{-1} . s \langle \psi, v \rangle + s\big(r \theta (v') + \langle \psi', v' \rangle\big)\Big), \quad (v,v')\in T_{(g, g')}G^{(2)}.
		\end{equation*}
		\item The inverse of $(\psi, r)\in T_g^{\dag}G\oplus \mathbbm{R}$, $g \in G$, is $(\phi, -r)\in T_{g^{-1}}G\oplus\mathbbm{R}$ where
		\begin{equation*}
			\langle \phi, v\rangle = - s_{g^{-1}}^{-1}t\langle \psi, v^{-1} \rangle -r \theta (v), \quad v\in T_{g^{-1}}G.
		\end{equation*}
	\end{itemize}
	The core of $\mk_{\theta}^{\dag}$ is $T^{\dag}M\oplus \mathbbm{R}_M$ and the core-anchor is $ \rho^{\dag}+ \ell_\theta^{\dag}\colon T^{\dag}M\oplus \mathbbm{R}_M \to A^{\dag}$.
	
	It turns out that the role of the curvature is played, in this case, by an appropriate VBG morphism $\mc_\theta \colon \mk_{\theta}\to \mk_{\theta}^{\dag}$. In order to define $\mc_\theta$ we need a connection $\nabla$ on $L_M$. The difference $s^{\ast}\nabla - t^{\ast}\nabla$ between the pull-back connections on $L \to G$ along the source and the target $s, t \colon L \to L_M$ is a $1$-form $\eta_\nabla \in \Omega^1 (G)$.
	
	\begin{lemma}
	The $1$-form $\eta_{\nabla}$ is multiplicative. Moreover $\eta_\nabla |_A = F_\nabla$ (see Diagram \eqref{eq:fnabla} for the definition of the linear form $F_\nabla \colon A \to \mathbbm R_M$).
	\end{lemma}
	\begin{proof}
	For the first part of the statement, denote by $\pr_{1,2} \colon L^{(2)} = L \mathbin{{}_s \times_t} L \to L$ the projections onto the two factors. Then
	\begin{align*}
		m^{\ast} \eta_{\nabla}&= m^{\ast}s^{\ast} \nabla - m^{\ast}t^{\ast}\nabla =\pr_2^{\ast}s^{\ast}\nabla -\pr_1^{\ast} t^{\ast} \nabla\\
		&=\pr_2^{\ast}s^{\ast}\nabla - \pr_2^{\ast} t^{\ast}\nabla + \pr_1^{\ast}s^{\ast}\nabla -\pr_1^{\ast} t^{\ast} \nabla = \pr_2^{\ast} \eta_{\nabla} + \pr_1^{\ast} \eta_{\nabla},
	\end{align*}
	where we used that $s \circ m = s \circ \pr_2$ and $t \circ m = t \circ \pr_1$. 
	
For the second part of the statement, let $\lambda_M\in \Gamma(L_M)$, $\lambda = s^{\ast}\lambda_M$ and $x\in M$. Then locally, around $x$, $s^{\ast}\lambda_M = ft^{\ast}\lambda_M$, for some function $f\in C^{\infty}(G)$ such that $f|_M=1$. Hence, for any $a\in A_x$, we have
		\begin{equation*}
			\eta_{\nabla}(a)\lambda_x = (s^{\ast}\nabla)_a\lambda - (t^{\ast}\nabla)_a\lambda= - (t^{\ast}\nabla)_a(ft^{\ast}\lambda_M)= - a(f)\lambda_{M,x} - \nabla_{\rho(a)} \lambda_M ,
		\end{equation*}	
		and
		\begin{equation*}
			F_\nabla(a)\lambda_x=F_\nabla(a)\lambda_{M,x}=\mathcal{D}_a\lambda_M - \nabla_{\rho(a)}\lambda_M.
		\end{equation*}
		Now let $a=\tfrac{d}{d\varepsilon}|_{\varepsilon=0} g(\varepsilon)$ be the velocity of a curve $g(\varepsilon)\colon x\to x(\varepsilon)$ in the $s$-fiber over $x$. Then
		\begin{align*}
			\mathcal{D}_a \lambda_M &= \frac{d}{d\varepsilon}|_{\varepsilon=0} \, g(\varepsilon)^{-1}. \lambda_{M,x(\varepsilon)}= \frac{d}{d\varepsilon}|_{\varepsilon=0} \, t\left(s_{g(\varepsilon)^{-1}}^{-1}\left(\lambda_{M,s(g(\varepsilon)^{-1})}\right)\right)\\
			&=\frac{d}{d\varepsilon}|_{\varepsilon=0} \, t\left(\left(s^{\ast}\lambda_M\right)_{g(\varepsilon)^{-1}}\right)= \frac{d}{d\varepsilon}|_{\varepsilon=0} \, t\left(f\left(g(\varepsilon)^{-1}\right)\left(t^{\ast}\lambda_M\right)_{g(\varepsilon)^{-1}}\right) \\
			&=\frac{d}{d\varepsilon}|_{\varepsilon=0} \, f\left(g(\varepsilon)^{-1}\right) \lambda_{M,x} =di(a)\left(f\right)\lambda_{M,x}\\
			&= \left(\rho(a)(f)- a(f)t\right)\lambda_{M,x}=-a(f)\lambda_{M,x},
		\end{align*}
		where, in the last step, we used that $f$ is constant on $M$. This concludes the proof.
	\end{proof}
	
	\begin{rem}
	Let $\nabla'$ be another connection on $L_M$. The difference $\alpha_{\nabla, \nabla'} = \nabla - \nabla'$ is a $1$-form on $M$. Moreover, $\eta_{\nabla} - \eta_{\nabla'} = \partial \alpha_{\nabla, \nabla'}$. In other words the $\partial$-cohomology class $\varkappa_L := [\eta_{\nabla}]$ of $\eta_\nabla$ is independent of $\nabla$ and it is a ``characteristic class'' attached to the LBG $L$. Clearly $\varkappa_L$ is the obstruction to the existence of a $G$-invariant connection on $L_M$, i.e.~a connection $\nabla$ such that $s^\ast \nabla = t^\ast \nabla$.
	\end{rem}

	We are now ready to define a stacky version of the curvature of $\theta$. Recall from the discussion preceding Proposition \ref{prop:+1_Mor_ker_Mor_inv} the explicit description of the twisted dual VBG $MK_\theta^\dag = T^\dag M \oplus \mathbbm R_G$ of the Morita kernel and consider the VB morphism
	\begin{equation*}
		\mc_{\theta}=
		\begin{pmatrix}
			d^{t^{\ast}\nabla}\theta & \eta_{\nabla}\\
			-\eta_{\nabla} & 0
		\end{pmatrix}
		\colon \mk_{\theta} \to \mk_{\theta}^{\dag}, \quad 
		\begin{pmatrix}
			v \\
			\lambda
		\end{pmatrix} \mapsto \begin{pmatrix} \iota_v d^{t^{\ast}\nabla}\theta + \lambda \otimes \eta_{\nabla} \\
		-\eta_{\nabla}(v)
		\end{pmatrix}.
	\end{equation*}
		
	\begin{prop/def}\label{prop:1-shift_MC}
		The VB morphism $\mc_{\theta}$ is a VBG morphism which acts on units as follows:
		\begin{equation*}
		\mc_{\theta}
		\colon TM \oplus L_M \to A^{\dag}, \quad 
		\begin{pmatrix}
			v \\
			\lambda
		\end{pmatrix} \mapsto  \big(\iota_v d^{t^{\ast}\nabla}\theta\big) |_A + \lambda \otimes \eta_{\nabla}|_A 
	\end{equation*}
		(combine Diagram \eqref{eq:cm_fibers_MC_+1} with Remark \ref{rem:Spencer} to get a more intrinsic description of the action of $\mc_\theta$ on units) \color{black}. We call this VBG morphism the \emph{Morita curvature of $\theta$}.
	\end{prop/def}
	\begin{proof}
		The proof is an easy computation which uses Proposition \ref{prop:ker_omega_comp}. For instance, for the source, take $g \in G$, $a \in A_{s(g)}$, $v\in T_gG$ and $\lambda\in L_g$. Then
	\begin{align}
		\Big\langle s\big(\mc_{\theta}(v,\lambda)\big), a\Big\rangle &=\left\langle s\big(\iota_v d^{t^{\ast}\nabla}\theta + \lambda\otimes \eta_{\nabla}, -\eta_{\nabla}(v)\big), a\right\rangle \nonumber\\
		&=-s\left(d^{t^{\ast}\nabla}\theta\big(v, 0_g^{TG}\cdot a^{-1}\big)\right) -s\left(\lambda\eta_{\nabla}\big(0_g^{TG}\cdot a^{-1}\big)\right) + \eta_{\nabla}(v)\theta(a) \nonumber\\
		&=-s\left(d^{t^{\ast}\nabla}\theta\big(v, 0_g^{TG}\cdot a^{-1}\big)\right) + \eta_{\nabla}(v)\theta(a) + \eta_{\nabla}(a) s(\lambda), \label{eq:fdgstw}
	\end{align}
	where we used that $\eta_\nabla$ is a multiplicative form. 	In order to compute the first summand in \eqref{eq:fdgstw}, consider $\omega \rightleftharpoons (\theta, 0) \in \OA^2 (L)$, and apply Equation \eqref{eq:omegaandcomponents} to the case $\sigma (\delta) = v$ and $\delta' = 0^{DL}_g\cdot a^{-1}$. We get
	\begin{equation*}
		-s\left(d^{t^{\ast}\nabla}\theta\big(v, 0_g^{TG}\cdot a^{-1}\big)\right)= -s\big(\omega(\delta, 0^{DL}_g\cdot a^{-1})\big) + f_{t^{\ast}\nabla}(\delta)s\big(\theta(0^{TG}_g \cdot a^{-1})\big) - f_{t^{\ast}\nabla}\big(0^{DL}_g\cdot a^{-1}\big)s\big(\theta(v)\big).
	\end{equation*}
	But $\omega$ is a multiplicative Atiyah form, so
	\begin{align*}
		m_{(g,s(g))}^{-1}\big(\omega(\delta, 0^{DL}_g \cdot a^{-1})\big)= m_{(g,s(g))}^{-1}\big(\omega(\delta \cdot Ds (\delta)), 0^{DL}_g \cdot a^{-1})\big)= \pr_{2,(g,s(g))}^{-1}\big(\omega(Ds(\delta), a^{-1})\big),
	\end{align*}
	whence, using that $s \circ m = s \circ \pr_2$,
	\begin{equation}\label{eq:fsgdthe}
		s\big(\omega(\delta, 0^{DL}_g\cdot a^{-1})\big)= \omega\big(Ds(\delta), a^{-1}\big)= -\omega\big(Ds(\delta),a\big).
	\end{equation}
	From the multiplicativity of $\theta$ we also get
	\begin{equation}\label{eq:ywrdsg}
		s\big(\theta(0^{TG}_g \cdot a^{-1})\big)= \theta(a^{-1})= -\theta(a).
	\end{equation}
	Substituting \eqref{eq:fsgdthe} and \eqref{eq:ywrdsg} in \eqref{eq:fdgstw}, and using that $f_{t^\ast \nabla} = f_\nabla \circ Dt$, so that $f_{t^{\ast}\nabla}(0^{DL}_g\cdot a^{-1})=0$, we get 
	\begin{align*}
		\Big\langle s\big(\mc_{\theta}(v,\lambda)\big), a\Big\rangle & = \omega\big(Ds(\delta),a\big)- f_{t^{\ast}\nabla}(\delta)\theta(a) + \eta_{\nabla}(v)\theta(a) + \eta_\nabla (a) s(\lambda)\\
		&= \omega\big(Ds(\delta), a\big) - f_{\nabla}\big(Ds(\delta)\big)\theta(a)  + \eta_\nabla (a) s(\lambda),
	\end{align*}
	where, for the last equality, we used that
	\begin{equation*}
		\eta_{\nabla}(v)- f_{t^{\ast}\nabla}(\delta)= (s^{\ast}\nabla)_v-(t^{\ast}\nabla)_v - \delta + (t^{\ast}\nabla)_v= - f_{s^{\ast}\nabla}(\delta)=-f_\nabla\big(Ds(\delta)\big).
	\end{equation*}
	On the other hand, applying Equation \eqref{eq:omegaandcomponents} to the case $\delta \leadsto Ds(\delta)$ and $\delta' \leadsto a\in D_{s(g)}L$ we get
\begin{align*}
		\Big\langle \mc_{\theta}\big(s(v),s(\lambda)\big), a\Big\rangle & = d^{t^{\ast}\nabla}\theta \big(s(v), a\big) + \eta_{\nabla}(a) s(\lambda) \\
		& = \omega\big(Ds(\delta), a\big) - f_{t^{\ast}\nabla}\big(Ds(\delta)\big)\theta(a) + f_{t^{\ast}\nabla}(a)\theta\big(s(v)\big) + \eta_{\nabla}(a) s(\lambda)\\
		& = \omega\big(Ds(\delta), a\big) - f_{\nabla}\big(Ds(\delta)\big)\theta(a) + \eta_{\nabla}(a) s(\lambda) \\
		& = \Big\langle s\big(\mc_{\theta}(v,\lambda)\big), a\Big\rangle,
	\end{align*}
	where we used that $\theta(s(v))=0$ (from the multiplicativity). We conclude that $s\circ \mc_{\theta}= \mc_{\theta} \circ s$, as desired. Compatibility with all other structure maps is similar and we leave it to the reader.
	\end{proof}
	
	\begin{rem}	
	For a generic connection $\nabla$, $d^{t^\ast \nabla}\theta$ is not multiplicative. It rather satisfies the following identity:
	\begin{equation}\label{eq:partial_dtheta}
	\partial \big(d^{t^\ast \nabla}\theta\big) = \pr_1^\ast \eta_\nabla \wedge \pr_2 \theta.
	\end{equation}
To see this, compute
		\begin{align*}
			m^{\ast}d^{t^{\ast}\nabla}\theta &= d^{(tm)^{\ast}\nabla} m^{\ast}\theta\\
			&= d^{(tm)^{\ast}\nabla}(\pr_1^{\ast}\theta + \pr_2^{\ast}\theta)\\
			&= \big(d^{(tm)^{\ast}\nabla}\circ \pr_1^{\ast}\big )\theta + \big(d^{(t\pr_1)^{\ast}\nabla}\circ \pr_2^{\ast}\big)\theta \\
			&=\pr_1^{\ast}d^{t^{\ast}\nabla}\theta + \pr_2^{\ast}d^{t^{\ast}\nabla}\theta + \big(d^{(tm)^{\ast}\nabla} \circ \pr_2^{\ast} - \pr_2^{\ast}\circ d^{t^{\ast}\nabla}\big)\theta \\
			&=\pr_1^{\ast}d^{t^{\ast}\nabla}\theta + \pr_2^{\ast}d^{t^{\ast}\nabla}\theta +\big((d^{(tm)^{\ast}\nabla}- d^{(t\pr_2)^{\ast}\nabla})\circ \pr_2^{\ast}\big )\theta\\
			&= \pr_1^{\ast}d^{t^{\ast}\nabla}\theta + \pr_2^{\ast}d^{t^{\ast}\nabla}\theta + \big((t\pr_1)^{\ast}\nabla - (s\pr_1)^{\ast}\nabla\big)\wedge \pr_2^{\ast}\theta\\
			&=\pr_1^{\ast}d^{t^{\ast}\nabla}\theta + \pr_2^{\ast}d^{t^{\ast}\nabla}\theta + \pr_1^{\ast}(t^{\ast}\nabla - s^{\ast}\nabla)\wedge \pr_2^{\ast}\theta\\
			&=\pr_1^{\ast}d^{t^{\ast}\nabla}\theta + \pr_2^{\ast}d^{t^{\ast}\nabla}\theta - \pr_1^{\ast}\eta_{\nabla} \wedge \pr_2^{\ast}\theta,
		\end{align*}
whence the claim. Now, there exists an alternative straightforward proof of Proposition \ref{prop:1-shift_MC} using Equation \eqref{eq:partial_dtheta} rather than Atiyah forms and Proposition \ref{prop:ker_omega_comp}. We leave the details to the reader. Clearly, it follows from \eqref{eq:partial_dtheta} that, if the connection $\nabla$ is $G$-invariant, then $d^{t^\ast \nabla}\theta$ is a multiplicative form.
	\end{rem}
	
	As the Morita curvature $\mc_{\theta}\colon \mk_{\theta}\to \mk_{\theta}^{\dag}$ is a VBG morphism, it induces a cochain map on fibers:
	\begin{equation}\label{eq:cm_fibers_MC_+1}
		\begin{tikzcd}
			0 \arrow[r] & A_x \arrow[r, "(\rho{,} \ell_\theta)"] \arrow[d, "\mc_{\theta}"']& T_xM\oplus L_{M,x} \arrow[r] \arrow[d, "\mc_{\theta}"] & 0 \\
			0 \arrow[r] & T_x^{\dag}M \oplus \mathbbm{R} \arrow[r, "\rho^{\dag}+ \ell_\theta^{\dag}"'] & A_x^{\dag} \arrow[r] & 0
		\end{tikzcd}, \quad x \in M.
	\end{equation}
	Notice that the map $\mc_{\theta}\colon T_xM\oplus L_{M,x}\to A_x^{\dag}$ is actually minus the twisted transpose of the map $\mc_{\theta}\colon A_x\to T_x^{\dag}M \oplus \mathbbm{R}$ (in particular, the latter is an isomorphism in cohomology if and only if so is the former).
	
	\begin{rem}\label{rem:Spencer}
	A multiplicative vector valued form on a Lie groupoid determines (and, under appropriate connectedness assumptions, is determined by) certain \emph{infinitesimal data}, the associated \emph{Spencer operator} \cite{Craini:multiplicativeforms} (see also \cite{Drummond:DifferentialformsinVBG} for a more general case). For instance, the Spencer operator of a multiplicative $L$-valued $1$-form $\theta \in \Omega^1 (G, L)$ is the pair $(D_\theta, \ell_\theta)$ where $\ell_\theta \colon A \to L_M$, just as above, is the restriction of $\theta$ to $A$, while $D_\theta \colon \Gamma (A) \to \Omega^1 (M, L_M)$ is the differential operator defined by 
	\begin{equation}\label{eq:D_theta}
	D_\theta a = u^\ast \left( \mathcal L_{\vec{a}} \theta \right), \quad a \in \Gamma (A),
	\end{equation}
	(see \cite{Craini:multiplicativeforms}, see also \cite[Section 10]{Vitagliano:djbundles} for this precise version of the definition of $D_\theta$). Equation \eqref{eq:D_theta} requires some explanations. Here $\vec{a} \in \Gamma (DL)$ is the right invariant derivation corresponding to $a$, and $\mathcal L_{\vec{a}} \theta \in \Omega^1 (G, L)$ denotes the Lie derivative of $\theta$ along $\vec{a}$:
	\[
	\mathcal L_{\vec{a}} \theta (X) = \vec{a} \big( \theta (a) \big) - \theta \big([\sigma(\vec{a}), X]\big), \quad X \in \mathfrak X (G).
	\]
	Now, it is not hard to see that the map $\mc_\theta \colon A \to T^\dag M \oplus L_M$ takes a section $a \in \Gamma (A)$ to
	\[
	\left(D_\theta a - d^\nabla \ell_\theta (a), F_\nabla (a) \right) \in \Gamma (T^\dag M \oplus L_M).\qedhere
	\]
	\end{rem}
	
	We now prove three ``invariance properties'' of the Morita curvature, explaining in which precise sense the Morita curvature is Morita invariant:
	\begin{enumerate}
	\item the Morita curvature is independent of the connection $\nabla$ up to linear natural isomorphisms (see Proposition \ref{prop:+1_MC_conn});
	\item the Morita curvature does only depend on the $\partial$-cohomology class of $\theta$ up to Morita equivalence (see Proposition \ref{prop:+1_MC_cohom_class});
	\item the Morita curvatures of two $1$-forms related by a Morita equivalence are also related by a Morita equivalence (see Proposition \ref{prop:+1_MC_Mequiv}). 
	\end{enumerate}
	
	\begin{prop}\label{prop:+1_MC_conn}
		Let $\nabla$ and $\nabla'$ be two connections on $L_M$. Then there is a linear natural isomorphism  (Definition \ref{def:LNT}) between the corresponding Morita curvatures.
	\end{prop}
	\begin{proof}
		Denote by $\mc$ and $\mc'$ the Morita curvatures defined through $\nabla$ and $\nabla'$ respectively. Moreover, set $\alpha_{\nabla, \nabla'}=\nabla -\nabla' \in \Omega^1(M)$. The map 		
		\begin{equation*}
			h=\begin{pmatrix}
			0 & \alpha_{\nabla, \nabla'}\\
			-\alpha_{\nabla, \nabla'} & 0
			\end{pmatrix}
			\colon TM\oplus L_M \to T^{\dag}M\oplus \mathbbm{R}_M
		\end{equation*}
		is a homotopy between the cochain maps on the core complexes determined by $\mc, \mc'$:
		\[
		\begin{tikzcd}
					0 \arrow[r] & A \arrow[d, "\mc", shift left=0.5ex] \arrow[d, "\mc'"', shift right=0.5ex]\arrow[r]& TM \oplus L_M \arrow[d, "\mc", shift left=0.5ex] \arrow[d,"\mc'"', shift right=0.5ex] \arrow[r] \arrow[dl, "h"'] &0\\
					0 \arrow[r] &T^\dag M \oplus \mathbbm R_M\arrow[r] &A^\dag \arrow[r]&0
				\end{tikzcd}.
		\]
		 Indeed, in degree $-1$, for any $x\in M$, $a\in A_x$, and $(v,\lambda)\in T_xM\oplus L_{M,x}$ we have 
		\begin{align*}
			\Big\langle(\mc' - \mc)(a), (v,\lambda)\Big\rangle&= \big(d^{t^{\ast}\nabla}\theta-d^{t^{\ast}\nabla'}\theta\big)(a,v) - (\eta_{\nabla}- \eta_{\nabla'})(a)\lambda  \\
			&=(t^{\ast}\alpha_{\nabla, \nabla'}\wedge\theta) (a,v)+t^{\ast}\alpha_{\nabla, \nabla'} (a) \lambda\\
			&=-\alpha_{\nabla, \nabla'}(v) \theta(a)+\alpha_{\nabla, \nabla'}(\rho(a))\lambda \\
			& = \big\langle h(\rho(a), \theta(a)), (v,\lambda)\big\rangle.
		\end{align*}
			The homotopy condition in degree $0$ now follows from the skew-selfadjointness of $h$. Finally, use Theorem \ref{theo:VBtransformation}.
	\end{proof}
	
	The next two propositions are straightforward.
	
	\begin{prop}\label{prop:+1_MC_cohom_class}
		Let $\theta, \theta' \in \Omega^1(G,L)$ be $\partial$-cohomologous multiplicative $1$-forms: $\theta -\theta' =\partial \alpha$, for some $\alpha\in \Omega^1(M,L_M)$. Then the Morita curvatures fit in the following commutative square
		\begin{equation*}
			\begin{tikzcd}
				\mk_{\theta} \arrow[r, "\mathsf A"] \arrow[d, "\mc_{\theta}"']& \mk_{\tilde{\theta}}\arrow[d, "\mc_{\theta'}"] \\
				\mk^{\dag}_{\theta} & \mk^{\dag}_{\theta'} \arrow[l, "\mathsf A^{\dag}"] 
			\end{tikzcd},
		\end{equation*}
		where $\mathsf A$ is the VB Morita map defined in Proposition \ref{prop:invarianceoftheta}.
	\end{prop}
	
	\begin{prop}\label{prop:+1_MC_Mequiv}
		Let $(F,f)\colon (L'\rightrightarrows L'_N;H\rightrightarrows N)\to (L\rightrightarrows L_M;G\rightrightarrows M)$ be a VB Morita map of LBGs, and let $\theta \in \Omega^1(G,L)$ be a multiplicative $L$-valued $1$-form. Moreover, let $\nabla$ be a connection on $L_M$. Denote $\theta' = F^\ast \theta$ and $\nabla' = f^\ast \nabla$. Then the Morita curvatures $\mc_{\theta}, \mc_{\theta'}$ fit in the following ``commutative'' pentagon
		\begin{equation}\label{diag:pentagon}
			{\footnotesize
				\begin{tikzcd}[column sep={{{{4em,between origins}}}},
					row sep=2em]
					\mk_{\theta'} \arrow[rrrr, "\mathsf F"] \arrow[d, "\mc_{\theta'}"']  & & & & \mk_{\theta} \arrow[d, "\mc_{\theta}"] \\
					\mk^{\dag}_{\theta'} & & & & \mk^{\dag}_{\theta} \\
					& & f^{\ast}\mk^{\dag}_{\theta} \arrow[ull, "\mathsf F^\dag"] \arrow[urr, "f"'] & &
				\end{tikzcd}},
		\end{equation}
		where all the other arrows are VB Morita maps (see Proposition \ref{prop:+1_Mor_ker_Mor_inv} for the definition of $\mathsf F$).
	\end{prop}
	
	The statement of Proposition \ref{prop:+1_MC_Mequiv} requires some explanations. Let $A_G, A_H$ be the Lie algebroids of $G, H$. The pentagon \eqref{diag:pentagon} is ``commutative'' in the following sense: the corresponding pentagon on the bases is
	\[
	{\footnotesize
		\begin{tikzcd}[column sep={{{{4em,between origins}}}},
			row sep=2em]
			H \arrow[rrrr, "f"] \arrow[d, equal] & & & & G \arrow[d, equal] \\
			H & & & & G \\
			& & H \arrow[ull, equal] \arrow[urr, "f"'] & &
		\end{tikzcd}},
	\]
	which is trivially commutative, while the corresponding pentagon on the fibers over $x \in H$ and $f(x) \in G$ is
	\begin{equation}\label{eq:pentagono}
	{\scriptsize
		\begin{tikzcd}
			[column sep={{{{4em,between origins}}}},
			row sep=2em]
			0 \arrow[rd]&   &   &  \phantom{x} &   &   & 0  \arrow[rd] &   &   &   \\
			& A_{H,x} \arrow[rrrrrr, "df"] \arrow[rd] \arrow[ddd, swap, "\mc_{\theta'}"] &   &   & \phantom{x}  &   &   & A_{G,f(x)}  \arrow[rd] \arrow[ddd, swap, "\mc_{\theta}"] &   &   \\
			&   & T_xN\oplus L_{N,x} \arrow[rrrrrr, crossing over, swap, "\mathsf{F}"] \arrow[rd]   &   &   &  \phantom{x} &   &   & T_{f(x)}M\oplus L_{M,f(x)}  \arrow[rd] \arrow[ddd, "\mc_{\theta}"] & \phantom{x}   \\
			0 \arrow[rd] &   &   & 0 &   &   & 0 \arrow[rd] &   &   & 0 \\
			& T_y^\dagger N\oplus\mathbbm{R} \arrow[rd] \arrow[from=rrrd, swap, pos=.4, "\mathsf{F}^\dagger"] &   & 0 \arrow[rd] &   &   &   & T_{f(x)}^\dagger M\oplus\mathbbm{R} \arrow[rd] \arrow[llld, equal] &   &   \\
			& \phantom{x}  & A^\dagger_{H,x} \arrow[rd] \arrow[from=uuu, crossing over, "\mc_{\theta'}"] &   & T_{f(x)}^\dagger M\oplus\mathbbm{R} \arrow[rd] &   &   &   & A^\dagger_{M,f(x)} \arrow[rd] \arrow[llld, equal] &   \\
			& \phantom{x}  &   & 0 & \phantom{x}  & A^\dagger_{M,f(x)} \arrow[rd] \arrow[lllu, pos=.4, "df^\dagger"] &   & \phantom{x}  &   & 0 \\
			&   &   &  \phantom{x} &   &   & 0 &   & \phantom{x}  &    
	\end{tikzcd}}
\end{equation}
	which is a commutative diagram of cochain maps (with all cochain maps being quasi-isomorphisms except, at the most, for $\mc_\theta, \mc_{\theta'}$).
	
	\begin{rem}
	There is another (and more precise) way to express the ``commutativity'' of \eqref{diag:pentagon}. Namely, according to \cite[Proposition 6.2]{DelHoyo:VBmorita}, a VB Morita map covering the identity is an equivalence, i.e.~there exists an \emph{inverse up to linear natural isomorphisms} (covering the trivial natural isomorphism). Hence there exists an inverse $\mathsf G \colon \mk^\dag_{\theta'} \to f^{\ast}\mk^{\dag}_{\theta}$ up to linear natural isomorphisms for the VB Morita map $\mathsf F^\dag \colon f^{\ast}\mk^{\dag}_{\theta} \to \mk^\dag_{\theta'}$. Now, \eqref{diag:pentagon} is ``commutative'' in the sense that the diagram obtained by replacing $\mathsf F^\dag$ with $\mathsf G$  is commutative up to linear natural isomorphisms:
	\begin{equation*}
			\begin{tikzcd}
				\mk_{\theta'} \arrow[r, "\mathsf F"] \arrow[d, "\mc_{\theta'}"']& \mk_{\theta}\arrow[d, "\mc_{\theta}"] \\
				\mk^{\dag}_{\theta'} \arrow[r, "f \circ \mathsf G"'] \arrow[ur, Rightarrow]& \mk^{\dag}_{\theta} 			\end{tikzcd}.
		\end{equation*}
		The best way to see this is via Theorem \ref{theo:VBtransformation}. First of all the VBG morphisms $K:=f \circ \mathsf G \circ \mc_{\theta'}$ and $K':=\mc_\theta \circ \mathsf F$ both cover $f \colon H \to G$. So, it is enough, for our purpose, to find a smooth homotopy $\mathcal H$ between the induced cochain maps on the core complexes:
		\begin{equation*}
			\footnotesize
				\begin{tikzcd}
					0 \arrow[r] &A_H \arrow[d, "K", shift left=0.5ex] \arrow[d, "K'"', shift right=0.5ex]\arrow[r]& TN \oplus L_N \arrow[d, "K", shift left=0.5ex] \arrow[d, "K'"', shift right=0.5ex]\arrow[r] \arrow[dl, "\mathcal H"'] &0\\
					0 \arrow[r]& f^{\ast}T^\dag M \oplus \mathbbm R_N\arrow[r] &f^{\ast}A^\dag_G \arrow[r]&0
				\end{tikzcd}.
			\end{equation*}
			In order to construct $\mathcal H$, remember that there is a linear natural isomorphism $\mathsf G \circ \mathsf F^\dag \Rightarrow \operatorname{id}$. This means that there exists a homotopy $h$ between the induced cochain maps on the core complexes:
			\begin{equation*}
				\begin{tikzcd}
					0 \arrow[r] &f^\ast T^\dag M \oplus \mathbbm R_N \arrow[r]& f^\ast A^\dag_G \arrow[r] \arrow[l, bend left, "h"] &0				\end{tikzcd}.
			\end{equation*}
			Set $\mathcal H_x := h_{f(x)} \circ \mc_\theta \circ \mathsf F \colon T_xN \oplus L_{N,x} \to T _{f(x)}^\dag M \oplus \mathbbm R$. A direct computation, using that \eqref{eq:pentagono} is commutative, now shows that $\mathcal H$ is the homotopy we were looking for.  
\end{rem}

\begin{rem}\label{rem:kappa}
Let $\theta \in \Omega^1 (G, L)$ be a multiplicative $1$-form. Suppose we are also given a $2$-form $\kappa \in \Omega^2 (M, L_M)$. Define a new VB morphism $\mc_{(\theta, \kappa)} \colon \mk_\theta \to \mk_\theta^\dag$ by 
	\[
	\mc_{(\theta, \kappa)} =
		\begin{pmatrix}
			d^{t^{\ast}\nabla}\theta + \partial \kappa & \eta_{\nabla}\\
			-\eta_{\nabla} & 0
		\end{pmatrix}.
	\]
It is easy to see that $\mc_{(\theta, \kappa)}$ is a VBG morphism as well. Moreover, there is a linear natural isomorphism $\mc_\theta \Rightarrow \mc_{(\theta, \kappa)}$. This is a straightforward consequence of Theorem \ref{theo:VBtransformation}. Indeed the difference $\mc_{(\theta, \kappa)} - \mc_\theta$ induces the following null-homotopic cochain map on fibers:
\[
		\begin{tikzcd}
					0 \arrow[r] & A \arrow[d, "-\kappa \circ \rho"'] \arrow[r]& TM \oplus L_M \arrow[d, "\rho^\dag \circ \kappa"]  \arrow[r] \arrow[dl, "- \kappa"'] &0\\
					0 \arrow[r] &T^\dag M \oplus \mathbbm R_M\arrow[r] &A^\dag \arrow[r]&0
				\end{tikzcd}.
\]
It follows that $\mc_{(\theta, \kappa)}$ satisfies the obvious analogues of Propositions  \ref{prop:+1_MC_conn}, \ref{prop:+1_MC_cohom_class} and \ref{prop:+1_MC_Mequiv}. Moreover, $\mc_\theta$ is VB Morita if and only if so is $\mc_{(\theta, \kappa)}$. We will refer to $\mc_{(\theta, \kappa)}$ as the \emph{Morita curvature} of the pair $(\theta, \kappa)$.
\end{rem}		
			
	\subsection{Definition and Examples}\label{sec:+1-def_exmpl}
	
	We are finally ready to present a stacky version of the aspect (A5) in the definition of a contact structure in the Introduction and give a definition of $+1$-shifted contact structure. Let $(L \rightrightarrows L_M; G \rightrightarrows M)$ be an LBG.
	
	\begin{definition}\label{def:+1_shift_LBG}
		A \emph{$+1$-shifted contact structure} on $L$ is a pair $(\theta, \kappa)$ consisting of an $L$-valued multiplicative $1$-form $\theta\in \Omega^1(G,L)$ and an $L_M$-valued $2$-form $\kappa \in \Omega^2 (M, L_M)$ such that the Morita curvature $\mc_{(\theta, \kappa)}\colon \mk_{\theta}\to \mk_{\theta}^{\dag}$ is a VB Morita map. The \emph{gauge transformation} of a $+1$-shifted contact structure $(\theta, \kappa)$ on $L$ by a pair $(\beta, \gamma ) \in \Omega^1 (M, L_M) \oplus \Omega^2 (M, L_M)$ is the $+1$-shifted contact structure $(\theta + \partial \beta, \kappa + \gamma)$. A \emph{$+1$-shifted contact LBG} is an LBG equipped with a $+1$-shifted contact structure. 
	\end{definition}

	\begin{rem}\label{rem:gauge_equiv_cont}
	That the gauge transformation $(\theta + \partial \beta, \kappa + \gamma)$ of a $+1$-shifted contact structure $(\theta, \kappa)$ is a $+1$-shifted contact structure as well immediately follows from Propositions \ref{prop:invarianceoftheta}, \ref{prop:+1_MC_cohom_class}, Theorem \ref{theo:caratterizzazioneVBmorita}, and Remark \ref{rem:kappa}. 
	It is clear that every $+1$-shifted contact structure $(\theta, \kappa)$ can be gauge transformed into the $+1$-shifted contact structure $(\theta, 0)$.
	\end{rem}

	The notion of $+1$-shifted contact structure is Morita invariant in an appropriate sense. First of all we have the following
	
	\begin{theo}\label{theor:Mor_inv_+1-shifted_cont}
	Let $(F, f)\colon (L' \rightrightarrows L'_N; H \rightrightarrows N) \to (L \rightrightarrows L_M; G \rightrightarrows M)$ be a VB Morita map, let $\theta \in \Omega^1 (G, L)$ be a multiplicative $L$-valued $1$-form and let $\kappa \in \Omega^2 (M, L_M)$. Then the pair $(F^\ast \theta, F^\ast \kappa)$ is a $+1$-shifted contact structure if and only if so is $(\theta, \kappa)$.
	\end{theo}
	
	\begin{proof}
	It is enough to consider the case $\kappa = 0$, which follows from Proposition \ref{prop:+1_MC_Mequiv}, Lemma \ref{lemma:Moritacurvatureorbit} below, and the essential surjectivity of $f$.
	\end{proof}

	
%
		
	\begin{lemma}
		\label{lemma:Moritacurvatureorbit}
		Let $\theta\in \Omega^1(G,L)$ be a multiplicative $L$-valued $1$-form. If $\mc_{\theta}$ is a quasi-isomorphism on fibers over the point $x\in M$, then it is a quasi-isomorphism on fibers over every point in the orbit of $x$.
	\end{lemma}
	
	\begin{proof}
	There is a direct proof that shows that the canonical actions of an arrow $g \colon x \to y$ in $G$ on $\ker \rho_x$ and $\operatorname{coker} \rho_x$ induce isomorphisms in the cohomology of the fibers of the Morita kernel (and its twisted dual) over $x, y$, commuting with the Morita curvatures.
	
	There is a more informative proof using RUTHs (see Appendix \ref{app:RUTH}), which adopts the same argument as that in Remark \ref{rem:MC_RUTH}. We now sketch such proof. We will actually prove that, similarly as in the case of $\partial$-closed $1$-forms on $M$ in Remark \ref{rem:MC_RUTH}, the map on fibers determined by the Morita curvature at different points of the same orbit are related by isomorphisms in cohomology. 
	
	To do this, choose an Ehresmann connection on $G$ and notice that it allows to split both the exact sequences
	\begin{equation*}
		\begin{tikzcd}
			0 \arrow[r] & t^\ast A \arrow[r]& MK_\theta \arrow[r] & s^\ast \big(TM \oplus L_M\big) \arrow[r]  & 0 
		\end{tikzcd},
	\end{equation*}
	and
	\begin{equation*}
		\begin{tikzcd}
			0 \arrow[r] & t^\ast \big(T^\dag M \oplus \mathbbm R_M\big) \arrow[r]& MK^\dag_\theta \arrow[r] & s^\ast A^\dag \arrow[r]   & 0 
		\end{tikzcd}
	\end{equation*}
	in the obvious way. Now use the techniques of \cite{Gracia:VBgroupoids} to construct $2$-term RUTHs from $MK_\theta, MK_\theta^\dag$ and the splittings. Denote by $\{R_k\}_{k \geq 0}, \{R^\dag_k\}_{k \geq 0}$ the structure operators of such RUTHs. From the categorical equivalence between VBGs and $2$-term RUTHs \cite{Gracia:VBgroupoids, DelHoyo:VBmorita}, the Morita curvature induces a morphism of RUTHs, and the RUTH morphism identities \eqref{eq:struct_sect_RUTH} now say (among other things) that, for every $g \colon x \to y$ in $G$, the diagonal arrows in the diagram
	\[
	{\scriptsize
	\begin{tikzcd}
		0 \arrow[rr] & & A_x \arrow[rr] \arrow[dd] & & T_xM\oplus L_{M,x} \arrow[rr] \arrow[dd] & & 0 \\
		& 0 \arrow[rr, crossing over] & & A_y \arrow[rr, crossing over] \arrow[from=ul, "R_1(g)"] & & T_yM \oplus L_{M,y} \arrow[rr] \arrow[from=ul, "R_1(g)"]& & 0 \\
		0 \arrow[rr] & & T_x^{\dagger}M \oplus \mathbbm{R} \arrow[rr] \arrow[dr, "R^\dag_1(g)"] & & A_x^{\dagger} \arrow[rr] \arrow[dr, "R^\dag_1(g)"] & & 0 \\
		& 0 \arrow[rr] & & T_y^{\dagger}M \oplus \mathbbm{R} \arrow[rr] \arrow[from=uu, crossing over]  & & A_y^{\dagger} \arrow[rr] \arrow[from=uu, crossing over] & & 0
	\end{tikzcd}}
	\]
	induce isomorphisms in the cohomology of the fibers of $MK_\theta$ and $MK_\theta^\dag$ over $x, y$ which actually commute with the Morita curvature in cohomology. This concludes the proof.
	\end{proof}
	
	\begin{definition}\label{def:cont_Mor_equiv}
	Two $+1$-shifted contact LBGs $(L_1, (\theta_1, \kappa_1)), (L_2, (\theta_2, \kappa_2))$ are \emph{contact Morita equivalent} if there exist an LBG $L'$, and VB Morita maps
	\[
	\begin{tikzcd}
		& L' \arrow[dl, "F_1"'] \arrow[dr, "F_2"] \\
		L_1 & & L_2
	\end{tikzcd}
	\]
such that the $+1$-shifted contact structures $(F_1^\ast \theta_1, F_1^\ast \kappa_1), (F_2^\ast \theta_2, F^\ast_2 \kappa_2)$ agree up to a gauge transformation.
	\end{definition}

	\begin{rem}
	Notice that the $2$-forms $\kappa_1, \kappa_2$ don't play any role in Definition \ref{def:cont_Mor_equiv}. We decided to keep them therein in view of the relationship between $+1$-shifted contact structures and $+1$-shifted Atiyah forms (see Theorem \ref{theor:+1_theta_omega_bij} below).
	\end{rem}

	\begin{prop}\label{prop:cont_Mor_equiv}
	Contact Morita equivalence is an equivalence relation.
	\end{prop}
	
	\begin{proof}
	The proof is similar to the proof of the analogous statement for the symplectic case. We only sketch it. As the $2$-forms $\kappa_1, \kappa_2$ don't play any role in Definition \ref{def:cont_Mor_equiv}, we simply ignore them, but they can be easily restored in the obvious way. Reflexivity and symmetry are obvious. For the transitivity, let $(L_1, \theta_1),  (L_2, \theta_2), (L_3, \theta_3)$ be $+1$-shifted contact LBGs and let 
	\[
	\begin{tikzcd}[column sep=-3]
		& (L', \beta') \arrow[dl, "F_1"'] \arrow[dr, "F_2"]  & & (L'', \beta'')   \arrow[dl, "G_1"'] \arrow[dr, "G_2"]  &\\
		(L_1, \theta_1) & & (L_2, \theta_2) & & (L_3, \theta_3)
	\end{tikzcd}
	\]
be two contact Morita equivalences, i.e.
\[
F_2^\ast \theta_2 - F_1^\ast \theta_1 - \partial \beta' = 0, \quad \text{and} \quad G_2^\ast \theta_3 - G_1^\ast \theta_2 - \partial \beta'' = 0.
\]	
The \emph{homotopy fiber product} \cite{DelHoyo:stacks} $L'''$ of $L' \to L_2 \leftarrow L''$ is an LBG (whose base groupoid $G''' \rightrightarrows M'''$ is the homotopy fiber product of the bases). Moreover, the projections $L' \leftarrow L''' \rightarrow L''$ are VB Morita maps fitting in the following diagram:
\begin{equation}\label{eq:homot_fiber_prod}
\begin{tikzcd}[column sep=-3]
	& &L''' \arrow[dl, "F"'] \arrow[dr, "G"]  & & \\
		& \big(L', \beta'\big) \arrow[dl, "F_1"'] \arrow[dr, "F_2"]  & & \big(L'', \beta''\big)   \arrow[dl, "G_1"'] \arrow[dr, "G_2"]  &\\
		\big(L_1, \theta_1\big) & & \big(L_2, \theta_2\big) & & \big(L_3, \theta_3\big)
	\end{tikzcd}
\end{equation}
The middle square in \eqref{eq:homot_fiber_prod} commutes up to a natural transformation $T \colon F_2 \circ F \Rightarrow G_1 \circ G$. It follows that
\[
	(G_1 \circ G)^\ast \theta_2  - (F_2 \circ F)^\ast \theta_2 = \partial \beta
\]
where $\beta = T^\ast \theta_2 \in \Omega^1 (M''', L'''_{M'''})$.
A straightforward computation now shows that $(F_1 \circ F)^\ast \theta_1$ and $(G_2 \circ G)^\ast \theta_3$ agree up to the gauge transformation by
$
G^\ast \beta''- F^\ast \beta' + \beta
$.
	\end{proof}
	
	\begin{theo}\label{theor:new_Mor_inv_+1-shifted_cont}
	Let $(L_1, (\theta_1, \kappa_1))$ be a $+1$-shifted contact LBG and let $L_2$ be a VB Morita equivalent LBG. Then there exists a $+1$-shifted contact structure $(\theta_2, \kappa_2)$ on $L_2$ such that $(L_1, (\theta_1, \kappa_1))$ and $(L_2, (\theta_2, \kappa_2))$ are contact Morita equivalent. Moreover the $+1$-shifted contact structure $(\theta_2, \kappa_2)$ is unique up to gauge transformations.
	\end{theo}
	
	\begin{proof}
	We will ignore $\kappa_1, \kappa_2$ as in the proof of Proposition \ref{prop:cont_Mor_equiv}. From Corollary \ref{cor:VB_Mor_equiv_LBGs} there exist an LBG $L'$ and VB Morita maps
	\[
	\begin{tikzcd}
		& L' \arrow[dl, "F_1"'] \arrow[dr, "F_2"] \\
		L_1 & & L_2
	\end{tikzcd}
	\]
	From \cite[Theorem 8.8]{Drummond:DifferentialformsinVBG} the assignment $F_1^\ast \colon \theta \mapsto F_1^\ast \theta$ establishes a bijection between $\partial$-cohomology classes of multiplicative $L_1$-valued and $L'$-valued $1$-forms (likewise for $F_2$). Therefore, there exists $\theta_2 \in \Omega^1 (G_2, L_2)$ and $\beta \in \Omega^1 (M', L'_{M'})$ such that $F_2^\ast \theta_2 - F_1^\ast \theta_1 = \partial \beta$. By Theorem \ref{theor:Mor_inv_+1-shifted_cont} and Remark \ref{rem:gauge_equiv_cont} $\theta_2$ is a $+1$-shifted contact structure. For the uniqueness, let $\tilde{\theta}_2 \in \Omega^1 (G_2, L_2)$ be another $+1$-shifted contact structure on $L_2$ such that $(L_1, \theta_1), (L_2, \tilde \theta_2)$ are contact Morita equivalent. Then, by \cite[Theorem 8.8]{Drummond:DifferentialformsinVBG}, $\theta_2$ and $\tilde \theta_2$ are in the same $\partial$-cohomology class. This concludes the proof.
	\end{proof}

	Proposition \ref{prop:cont_Mor_equiv} and Theorem \ref{theor:new_Mor_inv_+1-shifted_cont} motivate the following
	
	\begin{definition}\label{def:+1_shift_stack}
	A \emph{$+1$-shifted contact structure} on the LB stack $[L_M/L]$ is a contact Morita equivalence class of $+1$-shifted contact LBGs $(L, (\theta, \kappa))$ representing $[L_M/L]$.	\end{definition}
	
	\begin{rem}
	When $\theta_g \neq 0$ for all $g \in G$, then the Morita kernel can be replaced by the plain kernel $K_\theta$ up to Morita equivalence (see Remark \ref{rem:+1_Mor_ker_regular}). In this case, the standard curvature $R_\theta \colon K_\theta \to K_\theta^\dag$ is a VBG morphism and it is related to the Morita curvature by the following commutative diagram
	\begin{equation*}
			\begin{tikzcd}
				K_{\theta} \arrow[r, "\mathrm{in}"] \arrow[d, "R_{\theta}"'] & \mk_{\theta} \arrow[d, "\mc_{(\theta, \kappa)}"] \\
				K_{\theta}^{\dag} & \mk_{\theta}^{\dag} \arrow[l, "\mathrm{in}^\dag"]
			\end{tikzcd}
		\end{equation*}
	where both the inclusion $\mathrm{in} \colon K_\theta \to \mk_\theta$ and its twisted transpose map are VB Morita maps. It now follows that, in this case, $(\theta, \kappa)$ is a $+1$-shifted contact structure if and only if $R_\theta$ is a VB Morita map.
	\end{rem}

	Finally, we discuss the relationship between $+1$-shifted contact structures and $+1$-shifted Atiyah forms. The following theorem is a stacky analogue of Theorem \ref{prop:corrispondenza}, and it is yet another motivation for Definitions \ref{def:+1_shift_LBG} and \ref{def:+1_shift_stack}.
	
	\begin{theo}\label{theor:+1_theta_omega_bij}
		Let $(L \rightrightarrows L_M; G \rightrightarrows M)$ be an LBG. The assignment 
		\[
		\Omega^\bullet (G, L) \oplus \Omega^{\bullet +1}(M, L_M) \to \Omega^{\bullet + 1}_D(L)\oplus \OA^{\bullet+2}(L_M), \quad
		(\theta, \kappa) \mapsto \big(\omega \rightleftharpoons (\theta, \partial \kappa), \Omega \rightleftharpoons (\kappa, 0)\big),
		\]
		 establishes a bijection between $+1$-shifted contact structures and $+1$-shifted symplectic Atiyah forms on $L$. This bijection intertwines gauge equivalence and contact/symplectic Morita equivalence. 
	\end{theo}
	
	\begin{proof}
		Let $\theta\in \Omega^1(G,L)$, and $\kappa \in \Omega^2 (M, L_M)$. Consider the Atiyah forms $\omega\rightleftharpoons 
		(\theta,\partial \kappa)\in \OA^2(L)$ and $\Omega \rightleftharpoons 
		(\kappa,0)\in \OA^3(L_M)$. First notice that 
		\begin{equation*}\label{eq:delta=0}
		\dA \omega = \partial \Omega \quad \text{and} \quad \dA \Omega = 0.
		\end{equation*}
		Moreover, from Remark \ref{rem:partial_components}, $\theta$ is multiplicative if and only if so is 
		$\omega$. Even more, $(\theta, \kappa)$ is a $+1$-shifted contact structure
		 on $L$ if and only if $(\omega, \Omega)$ is a $+1$-shifted symplectic 
		 Atiyah form. To prove the latter 
		 claim it is enough to show that the Morita curvature of $(\theta, \kappa)$, equivalently of $(\theta, 0)$, 
		 is a quasi-isomorphism on fibers if and only if so is $\omega$. To do this, we follow a similar 
		 strategy as we did for $0$-shifted contact structures. So, 
		 take $x \in M$. The mapping cone of \eqref{eq:+1_map_fiber_omega_Atiyah} is
		\begin{equation}\label{eq:mapp_cone_+1-shift_At}
			\begin{tikzcd}
				0 \arrow[r] & A_x \arrow[r, "(-\mathcal{D} {,} \omega)"] & D_xL_M \oplus J^1_xL_M \arrow[r, "\omega + \mathcal{D}^{\dag}"] & A_x^{\dag} \arrow[r] &0
			\end{tikzcd}.
		\end{equation}
		Now, let $\nabla$ be a connection on $L_M$. Under the direct sum decompositions $DL_M= TM\oplus \mathbbm{R}_M$ and $J^1L_M= T^{\dag}M \oplus L_{M}$, \eqref{eq:mapp_cone_+1-shift_At} becomes
		\begin{equation*}
			\begin{tikzcd}
				0 \arrow[r] & A_x \arrow[rrr, "(-\rho {,} -F_\nabla {,} d^{t^{\ast}\nabla}\theta {,} -\theta)"] & & & T_xM \oplus \mathbbm{R} \oplus T_x^{\dag}M \oplus L_{M,x} \arrow[rrr, "d^{t^{\ast}\nabla}\theta + \theta^{\dag} + \rho^{\dag} + F^\dag_{\nabla}"] & & & A_x^{\dag} \arrow[r] & 0
			\end{tikzcd}.
		\end{equation*}
		But this is exactly the mapping cone of \eqref{eq:cm_fibers_MC_+1}, whence the claim. The bijectivity is straightforward. 
		
		For the second claim, just notice that the bijection in the statement intertwines the gauge transformation (of $+1$-shifted contact structures) by $(\beta, \gamma)$ and the gauge transformation (of $+1$-shifted symplectic Atiyah forms) by $\alpha \rightleftharpoons (\beta, \gamma)$. Likewise for contact/symplectic Morita equivalence. 
		\end{proof}
	
	
	\begin{rem}
	Let $(L\rightrightarrows L_M; G \rightrightarrows M)$ be an LBG equipped with a $+1$-shifted contact structure $(\theta, \kappa)$ and let $(\omega, \Omega)$ be the corresponding $+1$-shifted symplectic Atiyah form. It then follows by dimension counting, from $\omega$ determining a quasi-isomorphism on fibers, that $\dim G = 2 \dim M +1$.
	\end{rem}
	
	\begin{rem}
	Let $Q \rightrightarrows P$ and $(L \rightrightarrows L_M; G \rightrightarrows M)$ be as in Remark \ref{rem:hom_0-shift_sympl_grp}. Let $\theta \in \Omega^1 (G, L)$ (resp.~$\kappa \in \Omega^2 (M, L_M)$) and let $\Theta \in \Omega^1 (Q)$ (resp.~$K \in \Omega^2 (P)$) be the corresponsing degree $1$ homogeneous $1$-form (resp.~$2$-form), given by interpreting $\Gamma (L)$ (resp.~$\Gamma (L_M)$) as degree $1$ homogeneous functions on $Q$ (resp.~$P$). Similarly as in the $0$-shifted case (Remark \ref{rem:hom_0-shift_sympl_grp}), the assignment $(\theta, \kappa) \mapsto (d\Theta, dK)$ establishes a bijection between $+1$-shifted contact structures on $L$ and homogeneous $+1$-shifted symplectic structures of degree $1$ on $Q \rightrightarrows P$. This can be seen using Theorem \ref{theor:+1_theta_omega_bij} and the relationship between degree $1$ homogeneous differential forms and Atiyah forms \cite{Vitagliano:holomorphic} again.
	\end{rem}

We conclude the paper with some examples. We will discuss more interesting examples in a forthcoming work \cite{MTV:preparation}.

\begin{example}
Let $(G, K)$ be a \emph{contact groupoid} in the sense of Dazord \cite{Dazord}, i.e.~$K \subseteq TG$ is a contact distribution and a Lie subgroupoid of the tangent groupoid $TG \rightrightarrows TM$. Then $L = TG/K$ is an LBG and the projection $\theta \colon TG \to L$ is a multiplicative contact form (see, e.g., \cite{CS2015}, see also \cite{Grabowski:remarks}), hence a $+1$-shifted contact structure.
\end{example}

\begin{example}[Trivial $+1$-shifted contact structure]
Let $L_M \to M$ be a line bundle. Consider the general linear groupoid $G := \operatorname{GL}(L_M) \rightrightarrows M$ of $L_M$ and denote by $(L \rightrightarrows L_M; G \rightrightarrows M)$ the $LBG$ coming from the tautological action of $G$ on $L_M$. The core complex of $DL \rightrightarrows DL_M$ is
\[
\begin{tikzcd}
				0 \arrow[r] & DL_M \arrow[r, equal] & DL_M \arrow[r] & 0
			\end{tikzcd}.
\]
Given any $1$-form $\theta_M \in \Omega^1 (M, L_M)$, and any $2$-form $\kappa \in \Omega^2 (M, L_M)$, the pair $(\partial \theta_M , \kappa)$ is a $+1$-shifted contact structure on $L$, which is actually trivial up to gauge transformations.
\end{example}

	\begin{example}[Prequantization of $+1$-shifted symplectic structures]
		Recall from \cite{BX2003,Xu:prequantization} the notion of \emph{prequantization of a $+1$-shifted symplectic structure}. Let $G\rightrightarrows M$ be a Lie groupoid. An \emph{$S^1$-central extension of $G$} is a Lie groupoid $H\rightrightarrows M$ together with a Lie groupoid morphism $\pi\colon H\to G$ being the identity on objects, and an $S^1$-action on $H$, making $\pi\colon H\to G$ a principal \emph{$S^1$-bundle groupoid}, i.e.~the following compatibility between the principal action and the groupoid structure holds: if we denote by $\star$ the multiplication in $S^1$, and by a dot $.$ the $S^1$-action, then $(\phi .h)(\phi'.h')=(\phi \star \phi').hh'$ for all $\phi, \phi'\in S^1$, and $(h,h')\in H^{(2)}$.

Let $\pi \colon H \to G$ be an $S^1$-central extension. Then a \emph{pseudoconnection} in $H$ is a pair $(\theta, \kappa)$ consisting of a $1$-form $\theta \in \Omega^1(H)$ on $H$ and a $2$-form $\kappa \in \Omega^2(M)$ on $M$, such that $\theta$ is a principal connection $1$-form on $H$. Finally, let $(\omega, \Omega)$ be a $+1$-shifted symplectic structure on $G$. A  \emph{prequantization of $(\omega, \Omega)$} is an $S^1$-central extension $\pi\colon H \to G$ with a pseudo-connection $(\theta, \kappa)$, such that 
		\begin{equation}\label{eq:prequant}
			\partial \theta=0, \quad d\theta=\partial \kappa - \pi^{\ast}\omega, \quad \text{and} \quad d\kappa=\Omega.
		\end{equation}
		A prequantization $(\theta, \kappa)$ exists if $(\omega, \Omega)$ is an \emph{integral cocycle} in the total complex of the Bott-Shulmann-Stasheff double complex \cite[Proposition 3.3]{BX2003}.
		
		We can regard $\theta$ as a $1$-form with values in the \emph{trivial LBG} $\mathbbm R_G \to G$, i.e.~the LBG corresponding to the trivial action of $G$ on the trivial line bundle $\mathbbm R_M$. Then $\theta$ is multiplicative. By definition, $\theta_h\neq 0$ for all $h\in H$ and its curvature is $R_{\theta} = d\theta|_{K_\theta} \colon K_{\theta} \to K_{\theta}^{\dag}$ which is VB Morita. To see this, denote by $A_H, A_G$ the Lie algebroids of $H, G$. The fiber of $K_\theta$ over $x \in M$ is
		\begin{equation}\label{eq:fiber_K_theta_prequant}
		\begin{tikzcd}
				0 \arrow[r] & A_{H, x} \cap K_{\theta, x} \arrow[r, "\rho"] & T_x M \arrow[r] & 0
			\end{tikzcd}.
		\end{equation}
	If we use $d\pi \colon TH \to TG$ to identify $K_\theta$ with $\pi^\ast H$, then \eqref{eq:fiber_K_theta_prequant} identifies with
	\begin{equation}
		\begin{tikzcd}
				0 \arrow[r] & A_{G, x} \arrow[r, "\rho"] & T_x M \arrow[r] & 0
			\end{tikzcd}.
		\end{equation}
		Now, from the second one of \eqref{eq:prequant}, the curvature $R_\theta$ determines the following cochain map on fibers:
		\begin{equation}\label{eq:R_theta_fiber_prequant}
		\begin{tikzcd}
				0 \arrow[r] & A_{G, x} \arrow[r, "\rho"] \arrow[d, "- \omega - \kappa \circ \rho"']& T_x M \arrow[r] \arrow[d, "- \omega- \rho^\ast \circ \kappa"] & 0 \\
				0 \arrow[r] & T_x^\ast M \arrow[r, "\rho^\ast"'] & A_{G, x}^\ast \arrow[r] & 0
			\end{tikzcd},
		\end{equation}
	As $(\omega, \Omega)$ is a $+1$-shifted symplectic structure, then $\omega \colon TG \to T^\ast G$ is a VB Morita map and the vertical arrows in \eqref{eq:R_theta_fiber_prequant} are a quasi-isomorphism. We conclude that $R_\theta$ is a VB Morita map as claimed and $(\theta, \kappa)$ is a $+1$-shifted contact structure.
	\end{example}

	\begin{example}[Dirac-Jacobi structures]
	$+1$-shifted symplectic structures are the global structures on Lie groupoids integrating Dirac structures twisted by a closed $3$-form, seen as infinitesimal structures on the corresponding Lie algebroid \cite{Zhu:twisted}. Similarly, $+1$-shifted contact structures are the global structures integrating \emph{Dirac-Jacobi structures} \cite{Vitagliano:djbundles}.
	
	Let $L_M \to M$ be a line bundle. The \emph{omni-Lie algebroid} \cite{CL2010} of $L_M$ is the VB $\mathbbm{D}L_M:= DL_M \oplus J^1L_M$ together with the following structures:
	\begin{itemize}
	\item the projection $\pr_D \colon \mathbbm DL_M \to DL_M$ onto $DL_M$;
	\item the non-degenerate, symmetric $L$-valued bilinear form of split signature:
	\begin{equation*}
			\big\langle\hspace{-3pt}\big\langle -, - \big\rangle\hspace{-3pt}\big\rangle \colon \mathbbm{D}L_M \otimes \mathbbm{D}L_M \to L_M, \quad \big\langle\hspace{-3pt}\big\langle (\delta ,\psi), (\delta' , \psi') \big\rangle\hspace{-3pt}\big\rangle= \langle \psi, \delta'\rangle + \langle \psi', \delta\rangle;
		\end{equation*}
		\item the bracket on sections:
		\begin{equation*}
			\big[\hspace{-3pt}\big[-,-\big]\hspace{-3pt}\big]\colon \Gamma (\mathbbm D L_M) \times \Gamma (\mathbbm D L_M)  \to \Gamma (\mathbbm DL_M), \quad \big[\hspace{-3pt}\big[(\Delta,\psi), (\Delta',\psi')\big]\hspace{-3pt}\big]:= \Big([\Delta,\Delta'], \mathcal{L}_{\Delta}\psi' - \iota_{\Delta'}\dA\psi\Big),
		\end{equation*}
		where $\mathcal L_\Delta$ is the \emph{Lie algebroid Lie derivative} along $\Delta$.
	\end{itemize}
	The omni-Lie algebroid is an instance of an \emph{$E$-Courant algebroid} \cite{CLS2010}, an \emph{$AV$-Courant algebroid} \cite{LB2011} and a \emph{contact-Courant algebroid} \cite{G2013}, and can be regarded as a \emph{contact version} of the standard Courant algebroid: the generalized tangent bundle $\mathbbm T M = TM \oplus T^\ast M$.
	
	A \emph{Dirac-Jacobi structure} on $L_M$ is a vector subbundle $\mathbbm L \subseteq \mathbbm DL_M$ which is Lagrangian with respect to the inner product $\langle\!\langle -, - \rangle\!\rangle$ and whose sections are preserved by the bracket $[\![-,-]\!]$ (\cite{Vitagliano:djbundles}, see also \cite{W2000,W2004}). A Dirac-Jacobi structure is a contact version of a Dirac structure. Any Dirac-Jacobi structure $\mathbbm L \to M$ is a Lie algebroid, with anchor given by $\rho := \sigma \circ \pr_D \colon \mathbbm L \to TM$, and Lie bracket given by the restriction of $[\![-,-]\!]$. Moreover $\mathbbm L$ acts on $L_M$ via $\pr_D \colon \mathbbm L \to DL_M$. Dirac-Jacobi structures encompass contact and pre-contact structures, flat line bundles, locally conformally symplectic and locally conformally pre-symplectic structures, Jacobi structures, Poisson and Dirac structures as distinguished examples. Additionally, generalized complex structures in odd dimensions (aka generalized contact structures \cite{VW2016,SV2020}) can be seen as certain complex Dirac-Jacobi structures. This shows the wide range of applications of their theory.
	
	One can also define \emph{twisted Dirac-Jacobi structures}, in the same spirit as twisted Dirac structures (see \cite{NdCP2006} for the trivial line bundle case). To do this first notice that one can \emph{deform} the bracket $[\![-,-]\!]$ on sections of the omni-Lie algebroid $\mathbbm D L_M$ via a $\dA$-closed Atiyah $3$-form $\Omega \in \OA^3 (L_M)$ as follows: define the new \emph{deformed bracket}
	\begin{equation*}
			\big[\hspace{-3pt}\big[-,-\big]\hspace{-3pt}\big]_\Omega \colon \Gamma (\mathbbm D L_M) \times \Gamma (\mathbbm D L_M)  \to \Gamma (\mathbbm DL_M), \ \  \big[\hspace{-3pt}\big[(\Delta,\psi), (\Delta',\psi')\big]\hspace{-3pt}\big]_\Omega := \big[\hspace{-3pt}\big[(\Delta,\psi), (\Delta',\psi')\big]\hspace{-3pt}\big] + \big(0, \iota_\Delta \iota_{\Delta'} \Omega\big).
		\end{equation*}
		Then $(\mathbbm DL_M, \pr_D, \langle\!\langle -, - \rangle\!\rangle, [\![-,-]\!]_\Omega)$ is again a contact Courant algebroid. An \emph{$\Omega$-twisted Dirac-Jacobi structure} on $L_M$ is a vector subbundle of $\mathbbm L \subseteq \mathbbm D L_M$ which is Lagrangian with respect to $\langle\!\langle -, - \rangle\!\rangle$ and whose sections are now preserved by the deformed bracket $[\hspace{-3pt}[-,-]\hspace{-3pt}]_\Omega$. 
		
	\emph{Twisted Dirac-Jacobi structures integrate to $+1$-shifted contact structures} in the following sense. First of all, a \emph{Dirac-Jacobi algebroid} is a Lie algebroid $A \to M$ together with a Lie algebroid isomorphism $A \cong \mathbbm L \subseteq \mathbbm D L_M$ onto a twisted Dirac-Jacobi structure. Now, let $(L \rightrightarrows L_M; G \rightrightarrows M)$ be an LBG, and let $A$ be the Lie algebroid of $G$. Out of a $+1$-shifted contact structure $(\theta, \kappa)$ on $L$ one can construct a Dirac-Jacobi algebroid structure $A \cong \mathbbm L \subseteq \mathbbm D L_M$ twisted by $\Omega \rightleftharpoons (\kappa, 0)$ in a canonical way. Moreover, if $G$ is source-simply connected, then the latter construction establishes a one-to-one correspondence between $+1$-shifted contact structures on $L$ and twisted Dirac-Jacobi algebroid structures $A \cong \mathbbm L \subseteq \mathbbm D L_M$. In the untwisted case $\Omega = 0$ (i.e.~$\kappa = 0$) this is essentially \cite[Theorem 10.11]{Vitagliano:djbundles} together with the simple remark that a \emph{pre-contact groupoid} in the sense of \cite{Vitagliano:djbundles} is just a rephrasing of an LBG equipped with a $+1$-shifted contact structure of the form $(\theta, 0)$. In the twisted case, the proof is essentially the same and we omit it.
	
	Notice that, unlike the case of twisted Dirac structures, twisted Dirac-Jacobi structures are not really new structures with respect to untwisted Dirac-Jacobi structures (i.e.~$\Omega = 0$). The reason is essentially the acyclicity of the der-complex. Indeed, from $\dA \Omega = 0$, we get $\Omega = \dA B$, where $B = \iota_{\mathbbm I} \Omega \in \OA^2 (L_M)$. The map
		\[
		\mathbbm F_B \colon \mathbbm DL \to \mathbbm DL,\quad \mathbbm F_B (\delta, \psi) := \big(\delta, \psi - \iota_\delta B\big),
		\]
	 bijectively transforms $\Omega$-twisted Dirac-Jacobi structures to untwisted Dirac-Jacobi structures. Moreover, for every $\Omega$-twisted Dirac-Jacobi structure $\mathbbm L$, the restriction $\mathbbm F_B \colon \mathbbm L \to \mathbbm F_B (\mathbbm L)$ is a Lie algebroid isomorphism identifying the infinitesimal actions on $L_M$. This shows that $\Omega$-twisted Dirac-Jacobi structures are essentially the same as untwisted Dirac-Jacobi structures, and it is actually the infinitesimal counterpart of the remark that every $+1$-shifted contact structure can be gauge transformed into a $+1$-shifted contact structure of the form $(\theta, 0)$ (Remark \ref{rem:gauge_equiv_cont}).
	\end{example}
	
	\appendix
	
	\section{Representations up to Homotopy}\label{app:RUTH}
	
	 Recall from \cite{AC:RUTHs} that a representation up to homotopy (RUTH) of a Lie groupoid $G \rightrightarrows M$ is a graded vector bundle 
	\[
	E = \bigoplus_{m \in \mathbbm Z} E^m \to M
	\] 
	over $M$ (with $E$ usually assumed to be bounded from both sides) equipped with a differential $\partial^E \colon C(G; E)^\bullet  \to C(G; E)^{\bullet +1}$ on the $C(G)^\bullet := C^\infty (G^{(\bullet)})$-module
	\[
	C(G; E) = \bigoplus_{n \in \mathbbm Z} C(G; E)^n, \quad C(G; E)^n := \bigoplus_{k+m = n} \Gamma (t^\ast E^m \to G^{(k)})
	\]
	where we denote by $t \colon G^{(k)} \to M$ the composition of the projection $\mathrm{pr}_1 \colon G^{(k)} \to G$ onto the first factor followed by the target. Similarly, in the following, we will denote by $s \colon G^{(k)} \to M$ the composition of the projection onto the last factor followed by the source. The differential $\partial^E$ is required to give to $C(G; E)$ the structure of a DG module over $C(G)$. A morphism of RUTHs is just a morphism of DG modules. 	
	
	According to \cite[Proposition 3.2]{AC:RUTHs}, a RUTH on $E$ is equivalent to a sequence $\{R_k\}_{k \geq 0}$, where $R_k$ is a section of the vector bundle
	\[
	\operatorname{Hom}^{-k + 1}\big(s^\ast E, t^\ast E \big) \to G^{(k)}
	\]	
	of degree $-k + 1$ homomorphisms between the indicated pullback graded vector bundles, and the $R_k$'s satisfy the following identities: for all $k \geq 0$ and all $(g_1, \ldots, g_k) \in G^{(k)}$, 
	\begin{equation}\label{eq:struct_sect_RUTH}
	\sum_{j=1}^{k-1}(-)^j R_{k-1}(g_1, \ldots, g_jg_{j+1}, \ldots, g_k) = \sum_{j= 0}^k (-)^j R_j (g_1, \ldots, g_j) \circ R_{k-j} (g_{j+1}, \ldots, g_k).
	\end{equation}
	In particular, $R_0\colon E^\bullet \to E^{\bullet + 1}$ is a differential, hence it gives to $E$ the structure of a complex of vector bundles over $M$. 
The $R_k$'s will be called the \emph{structure operators of the RUTH}. Similarly, a morphism $\Phi\colon E \to E'$ of RUTHs is equivalent to a sequence $\{\Phi_k\}_{k \geq 0}$, where $\Phi_k$ is a section of the vector bundle 
	\[
	\operatorname{Hom}^{-k}\big(s^\ast E, t^\ast E' \big) \to G^{(k)},
	\]
	 and the $\Phi_k$'s satisfy the following identities: for all $k\geq 0$ and all $(g_1, \ldots, g_k) \in G^{(k)}$,
	 \begin{equation}\label{eq:struct_sect_RUTH_mor}
	 \begin{aligned}
	 & \sum_{i+j = k}(-)^j \Phi_j (g_1, \ldots, g_j) \circ R_i (g_{j+1}, \ldots, g_k) \\
	 & = \sum_{i+j = k} R'_j (g_1, \ldots, g_j) \circ \Phi_i (g_{j+1}, \ldots, g_k) + \sum_{j = 1}^{k-1} (-)^j \Phi_{k-1}(g_1, \ldots, g_jg_{j+1}, \ldots, g_k),
	 \end{aligned}
	 \end{equation}
	 where $R_k, R_k'$ are the structure operators of $E, E'$. In particular, $\Phi_0 \colon (E, R_0) \to (E', R'_0)$ is a cochain map.
	 The $\Phi_k$'s will be called the \emph{components of the morphism of RUTHs}. 
	 
	 If $E = E^0$ is concentrated in degree $0$, and $G$ acts on $E^0$, then letting $R_1\colon s^\ast E^0 \to t^\ast E^0$ being the (left) action, and $R_k = 0$ for $k \neq 1$, defines a RUTH (giving the standard cohomology of $G$ with coefficients in $E$). 
	 
	 In this paper, we also need the \emph{adjoint RUTH}: choose once for all an Ehresmann connection on $G$ \cite[Definition 2.8]{AC:RUTHs}, i.e.~a right splitting $h\colon s^\ast TM \to TG$ of the exact sequence of VBs over $G$
\begin{equation*}
		\begin{tikzcd}
			0 \arrow[r] & t^\ast A \arrow[r]& TG \arrow[r] & s^\ast TM \arrow[r]  \arrow[l, bend left, "h"] & 0 
		\end{tikzcd}
	\end{equation*}
	which agrees with $du \colon TM \to TG$ on units, and use it to promote the core complex of $TG$
	\begin{equation}\label{eq:core_compl_ad_RUTH:bis}
	\begin{tikzcd}
			0 \arrow[r] & A \arrow[r, "\rho"] & TM \arrow[r]&0
		\end{tikzcd}
	\end{equation}
	to a RUTH $(C(G; A \oplus TM), \partial^{\mathrm{Ad}})$, the adjoint RUTH \cite{AC:RUTHs}, as follows. Besides the differential $R_0^T = \rho$ of the core complex \eqref{eq:core_compl_ad_RUTH:bis}, the latter RUTH has got only two more non-trivial structure operators, the $1$st and the $2$nd, denoted $R^T_1, R^T_2$, which are defined as follows: let $\varpi = h \circ ds - \operatorname{id}\colon TG \to t^\ast A$ be the left splitting corresponding to $h$, then
	\[
	R^T_1(g)a =  - \varpi_g \big(dL_g (a)\big)\quad \text{and} \quad R^T_1 (g) v = dt \big( h_g (v)\big), \quad a \in A_{s(g)}, \quad v \in T_{s(g)}M,
	\]
	where $L_g$ is the left translation along $g \in G$, and, moreover
	\[
	R^T_2 (g_1, g_2) v = - \varpi_{g_1g_2} \Big(h_{g_1}\big(R^T_1 (g_2)v\big) \cdot h_{g_1} (v) \Big), \quad v \in T_{s(g_2)} M,
	\]
	where we used the multiplication in $TG \rightrightarrows TM$.
	
	There is also a RUTH coming from the VBG $DL \rightrightarrows DL_M$. Namely, an Ehresmann connection $h$ in $G$ also induces a right splitting $h^D\colon s^\ast DL_M \to DL$ of the exact sequence
\begin{equation}\label{eq:SES_Appendix}
		\begin{tikzcd}
			0 \arrow[r] & t^\ast A \arrow[r] & DL \arrow[r] & s^\ast DL_M \arrow[r]  \arrow[l, bend left, "h^D"] & 0
		\end{tikzcd},
	\end{equation}
as follows. Given $\delta \in D_xL_M$, $x \in M$, for every arrow $g\colon x \to y$, there exists a unique derivation $h^D_g (\delta) \in D_g L$ such that $\sigma (h^D_g (\delta)) = h_g (\sigma (\delta))$, and $Ds (h^D_g (\delta)) = \delta$ (see \cite[Section 4.1]{EspositoTortorellaVitagliano:infinitesimal_automorphism} for more details). In its turn $h^D$ determines a $2$-term RUTH \cite{Gracia:VBgroupoids}. The underlying cochain complex of vector bundles is
\begin{equation}\label{eq:core_compl_DL_RUTH}
\begin{tikzcd}
			0 \arrow[r] & A \arrow[r, "\mathcal D"] & DL_M \arrow[r] &  0,
		\end{tikzcd}
\end{equation}
and besides the differential $R_0^D = \mathcal D$ of  \eqref{eq:core_compl_DL_RUTH}, there are only two more non-trivial structure components, the $1$st and the $2$nd, denoted $R^D_1, R^D_2$, which are defined in a similar way as $R_1^T, R_2^T$ above:
	\[
	R^D_1(g)a =  R^T_1 (g)a \quad \text{and} \quad R^D_1 (g) \delta = Dt \big( h^D_g (\delta)\big), \quad a \in A_{s(g)}, \quad \delta \in D_{s(g)}L_M,
	\]
	and, moreover
	\begin{equation}\label{eq:gst36}
	R^D_2 (g_1, g_2) \delta = R^T_2 (g_1, g_2) \sigma(\delta).
	\end{equation}
It is easy to see that the symbol intertwines the structure operators $R^D$ with the structure operators $R^T$ \cite{EspositoTortorellaVitagliano:infinitesimal_automorphism}. Moreover $R_1^D(g)$ acts as the identity on endomorphisms $\mathbbm R \cong \operatorname{End} L_{M, s(g)} \subseteq D_{s(g)}L_M$.
Finally, $R_2^D(g_1, g_2)$ vanishes on endomorphisms.

		We conclude the appendix recalling that there is an equivalence between the category of RUTHs of a fixed Lie groupoid $G$, concentrated in non-positive degrees, and the category of \emph{higher VB groupoids} on $G$, i.e.~simplicial vector bundles over $G^{(\bullet)}$ (up to some technical aspects, see \cite{dHT2021}). Under such equivalence, RUTHs concentrated in degrees $-1,0$ correspond to VBGs \cite{Gracia:VBgroupoids, DelHoyo:VBmorita}.  For instance, the adjoint RUTH of $G \rightrightarrows M$ corresponds to the tangent VBG $TG \rightrightarrows TM$.

	\section{Proof of Proposition \ref{prop:MR_RUTH_morph}}\label{app:RUTH_mor}
	\begin{proof}
		We have to prove the RUTH morphism identities \eqref{eq:struct_sect_RUTH_mor}, where the $R_k$ are the structure operators of the RUTH living on the Morita kernel, and the $R'_k = R^\dag_k$ are the structure operators of the $L$-twisted dual RUTH. For simplicity of notation, in this proof, for any $g \in G$, we will denote $g_T. := R_1^T(g)$ and $g_D. := R_1^D (g)$.
		
		For $k=0$, we have to prove that $\Phi_0 \circ R_0 = R^{\dagger}_0 \circ \Phi_0$.
		This is true because the $0$-th component of the Morita curvature is a cochain map.
		For $k=1$, we have to prove that for any $g$ in $G$
		\begin{equation*}
			\Phi_0 \circ R_1(g) - R_1^{\dagger}(g) \circ \Phi_0 = R_0^{\dagger} \circ \Phi_1(g) + \Phi_1 (g) \circ R_0,
		\end{equation*}
		or, in other words, that for any $a\in A_{s(g)}$ 
		\begin{equation}\label{eq:whstf}
			F_{\nabla}\big(g_T.a\big) - g_D.F_{\nabla}(a)= \nabla_{g_T.\rho(a)} - g_D.\nabla_{\rho(a)}, 
		\end{equation}
		and that for any $v\in T_{s(g)}M$ and $v' \in T_{t(g)}M$
		\begin{equation}\label{eq:hsfet}
			d^{\nabla}\theta(g_T. v, v') - g.d^{\nabla}\theta (v, g^{-1}_T. v') =
			\big(\nabla_{g_T.v}-g_D. \nabla_v\big)\theta(v') + \big(\nabla_{g^{-1}_T.v}- g_D^{-1}. \nabla_v\big) g . \theta(v).
		\end{equation}
		
		In order to prove \eqref{eq:whstf}, let $\delta \in D_{s(g)} L_M$, and compute
				\[
				\begin{aligned}
		f_\nabla (g_D. \delta) & = g_D.\delta - \nabla_{\sigma (g_D.\delta)} = g^D.(\delta - \nabla_{\sigma (\delta)}) + g^D.\nabla _{\sigma (\delta)}-\nabla_{g_T.\sigma (\delta)} \\
		& = g_D.f_\nabla(\delta)+ g^D.\nabla _{\sigma (\delta)}-\nabla_{g_T.\sigma (\delta)}.
		\end{aligned}
		\]
		This shows that 
		\begin{equation}\label{eq:gshsy}
		f_\nabla (g_D. \delta) - g_D.f_\nabla (\delta ) = g^D.\nabla _{\sigma (\delta)}-\nabla_{g_T.\sigma (\delta)}
		\end{equation}
		and  \eqref{eq:whstf} follows immediately by putting $\delta = \mathcal D_a$.

		In order to prove \eqref{eq:hsfet}, use Equation \eqref{eq:omegaandcomponents} to compute
		\begin{equation*}
			d^{\nabla}\theta(g_T.v, v')= \omega(g_D. \delta, \delta') - f_\nabla(g_D.\delta)\theta(v') + f_\nabla(\delta')\theta(g_T.v),
		\end{equation*}
		where $v = \sigma (\delta), v' = \sigma (\delta')$. Similarly
		\begin{equation*}
			d^{\nabla}\theta(v, g_T^{-1}.v')= \omega(\delta, g^{-1}_D.\delta') - f_\nabla(\delta)\theta(g^{-1}_T.v') + f_\nabla(g^{-1}_D.\delta')\theta(v).
		\end{equation*}
		But it easily follows from the multiplicativity of $\theta$ and $\omega$ that  $\theta(g_T.v)=g.\theta(v)$ and $\omega(g_D. \delta, \delta') = g.\omega(\delta, g^{-1}_D.\delta')$, whence
		\begin{equation}
			d^{\nabla}\theta(g_T. v, w) - g.d^{\nabla}\theta (v, g^{-1}_T. w) = \Big(f_\nabla(\delta)- f_\nabla(g_D.\delta)\Big)\theta(w) + \Big(f_\nabla(\delta')-f_\nabla(g^{-1}_D.\delta')	\Big)g. \theta(v).
		\end{equation}
		Using \eqref{eq:gshsy} again, we get \eqref{eq:hsfet}.
		
		For $k=2$, we have to prove that for any $(g, g') \in G^{(2)}$		\begin{equation*}
			\Phi_0\circ R_2(g, g') - \Phi_1(g)\circ R_1(g') = R_1^{\dagger}(g)\circ \Phi_1(g') + R_2^{\dagger}(g,g')\circ \Phi_0 - \Phi_1(gg').
		\end{equation*}
		This means that, for any $v\in T_{s(g')}M$,
		\begin{equation}
			\label{eq:7.2}
			F_{\nabla}\big(R_2^T(g,g')v\big)+ \big(\nabla_{g_{T}.(g'_{T}.v)}-g_{D}. \nabla_{g'_{T}. v}\big) + g_D.\big(\nabla_{g'_{T}.v}- g'_{D}.\nabla_v\big) - \big(\nabla_{gg'_T.v}- gg'_D. \nabla_v\big) = 0,
		\end{equation}
		and that, for any $\lambda\in L_{M,s(g')}$ and $v\in T_{t(g)} M$,
		\begin{equation}\label{eq:2536dh}
		\begin{aligned}
			&gg'. \Big(g'{}^{-1}_D.\big(\nabla_{g^{-1}_{T}.v} - g^{-1}_{D}. \nabla_v\big) + F_{\nabla}\big(R_2^T(g'{}^{-1}, g^{-1})v\big) \\ &+ \big(\nabla_{g'{}^{-1}_{T}. (g_{T}^{-1}.v)} - g'{}^{-1}_{D}. \nabla_{g^{-1}_{T}.v}\big) - \big(\nabla_{(gg')^{-1}_T. v} - (gg')^{-1}_D .\nabla_v\big) \Big) \lambda=0.
			\end{aligned}
		\end{equation}
		But
		\begin{equation*}
			g_{T}.(g'_{T}.v) - (gg')_T.v = \rho\big(R_2^T(g,g')v\big), \quad v \in T_{s(g')} M,
		\end{equation*}
		and, similarly,
		\begin{equation*}
			g_{D}.(g'_{D}.\delta) - (gg')_D.\delta= \mathcal D_{R_2^D(g,g')\delta} = \mathcal D_{R_2^T(g,g')\sigma(\delta)}, \quad \delta \in D_{s(g')} L_M,
		\end{equation*}
		where we also used \eqref{eq:gst36}. It follows that the left hand side of \eqref{eq:7.2} is
		\begin{equation*}
		\begin{aligned}
			& F_{\nabla}\big(R_2^T(g,g')v\big)+ \big(\nabla_{g_{T}.(g'_{T}.v)}-g_{D}. \nabla_{g'_{T}. v}\big) + g_D.\big(\nabla_{g'_{T}.v}- g'_{D}.\nabla_v\big) - \big(\nabla_{gg'_T.v}- gg'_D. \nabla_v\big) \\
			& = \big(\mathcal D_{R_2^T(g,g')v} - \nabla_{\rho(R_2^T(g,g')v)}\big) + \big(\nabla_{g_{T}.(g'_{T}.v)}-g_{D}. \nabla_{g'_{T}. v}\big) +g_D. \big(\nabla_{g'_{T}.v}- g'_{D}.\nabla_v\big) - \big(\nabla_{gg'_T.v}- gg'_D. \nabla_v\big)\\
			&= g_{D}.(g'_D. \nabla_v)- g_D.\nabla_{g_T'.v} - g_D. \big( g'_{D}.\nabla_v - \nabla_{g'_{T}.v}\big) \\
			& =0.
			\end{aligned}
		\end{equation*}
		
		Clearly, \eqref{eq:2536dh} follows from \eqref{eq:7.2}.
				
For $k=3$, the RUTH morphism identity is trivially satisfied because $L_M$ is a plain representation (no higher homotopies). We leave the straightforward details to the reader. As, for degree reasons, there are no higher identities to check, this concludes the proof.
			\end{proof}
		
	\bibliographystyle{amsplain}

\end{document}